\documentclass[12pt, reqno, a4paper,oneside]{amsart}
\usepackage{graphicx}
\usepackage{subfig}
\usepackage{color}
\usepackage{xcolor}
\usepackage[top=2.5cm,bottom=2.0cm,left=2.0cm,right=2.0cm]{geometry}
\usepackage{amsfonts, amsmath, amssymb, amsbsy, amsthm}
\usepackage{environments}
\usepackage{mathrsfs}
\usepackage{listings}
\usepackage{algpseudocode}
\usepackage{algorithm, algorithmicx}
\usepackage{diagbox}
\numberwithin{theorem}{section}
\numberwithin{equation}{section}

\definecolor{yscol}{HTML}{6622AA}
\definecolor{hwcol}{rgb}{0, 0, 0.9}
\definecolor{mlcol}{rgb}{0, 0.7, 0}
\definecolor{todocol}{rgb}{0.0, 0.4, 0.0}

\title[ \ ]{Adaptive Quasicontinuum Methods and Simulations for Crystal Defects with a Theory based Unified a Posteriori Error Estimate}

\author{Hao Wang}
\address{Hao Wang\\
	School of Mathematics\\
	Sichuan University\\
	No. 24 Yihuan Road\\
	Chengdu\\
	China
}
\email{wangh@scu.edu.cn}

\author{Yangshuai Wang}
\address{Yangshuai Wang\\
	Department of Mathematics\\
	Faculty of Science\\
	National University of Singapore\\
	10 Lower Kent Ridge Road\\
	Singapore
}
\email{yswang@nus.edu.sg}

\date{\today}

\begin{document}
	
	\maketitle
	
\renewcommand{\cases}[1]{\left\{ \begin{array}{rl} #1 \end{array} \right.}
\newcommand{\smfrac}[2]{{\textstyle \frac{#1}{#2}}}
\newcommand{\myvec}[1]{\left[ \begin{vector} #1 \end{vector} \right]}
\newcommand{\mymat}[1]{\left[ \begin{matrix} #1 \end{matrix} \right]}

\def\Xint#1{\mathchoice
{\XXint\displaystyle\textstyle{#1}}%
{\XXint\textstyle\scriptstyle{#1}}%
{\XXint\scriptstyle\scriptscriptstyle{#1}}%
{\XXint\scriptscriptstyle\scriptscriptstyle{#1}}%
\!\int}
\def\XXint#1#2#3{{\setbox0=\hbox{$#1{#2#3}{\int}$ }
\vcenter{\hbox{$#2#3$ }}\kern-.6\wd0}}
\def\mint{\Xint-}

\def\b{\big}
\def\B{\Big}
\def\bg{\bigg}
\def\Bg{\Bigg}


\def\diam{{\textrm{diam}}}
\def\conv{{\textrm{conv}}}
\def\t{\top} 
\def\sign{{\textrm{sgn}}}
\def\id{{\textrm{id}}}
\def\supp{{\textrm{supp}}}
\def\diam{{\textrm{diam}}}

\def\R{\mathbb{R}}
\def\N{\mathbb{N}}
\def\Z{\mathbb{Z}}
\def\C{\mathbb{C}}
\def\bbV{\mathbb{V}}

\def\WW{W}
\def\CC{C}
\def\HH{H}
\def\LL{L}
\def\DD{{D}'}
\def\Ys{\mathscr{Y}}

\def\WWh{\dot{W}}
\def\Ycb{Y}
\def\WWhz{\dot{W}_0}
\def\Ycbz{\Ycb_0}

\def\dx{{\,{\textrm{d}}x}}
\def\dy{\,{\textrm{d}}y}
\def\dz{\,{\textrm{d}}z}
\def\dr{\,{\textrm{d}}r}
\def\dt{\,{\textrm{d}}t}
\def\ds{\,{\textrm{d}}s}
\def\dd{\textrm{d}}
\def\pp{\partial}
\def\dV{\,\textrm{dV}}
\def\dA{\,{\textrm{dA}}}
\def\db{\,{\textrm{db}}}
\def\dlam{\,{\textrm{d}}\lambda}

\def\<{\langle}
\def\>{\rangle}

\def\ol{\overline}
\def\ul{\underline}
\def\ot{\widetilde}
\newcommand{\ut}[1]{\underset{\widetilde{\hspace{2.5mm}}}{#1}}

\def\mA{{\textsf{A}}}
\def\mB{\textsf{B}}
\def\mC{\textsf{C}}
\def\mF{\textsf{F}}
\def\mG{\textsf{G}}
\def\mH{\textsf{H}}
\def\mI{\textsf{I}}
\def\mJ{\textsf{J}}
\def\mP{\textsf{P}}
\def\mQ{\textsf{Q}}
\def\mR{\textsf{R}}
\def\mM{\textsf{M}}
\def\mS{\textsf{S}}
\def\mO{\textsf{0}}
\def\mL{\textsf{L}}

\def\sym{\textsf{sym}}
\def\tr{\textsf{tr}}
\def\el{\textsf{el}}

\def\bfa{\textbm{a}}
\def\bfg{\textbm{g}}
\def\bfrho{\mathbm{\rho}}
\def\bfv{\textbm{v}}
\def\bfh{\textbm{h}}
\def\bfO{\textbm{0}}

\def\bbA{\mathbb{A}}
\def\bbB{\mathbb{B}}
\def\bbC{\mathbb{C}}
\def\bbI{\mathbb{I}}

\def\Hs{\mathcal{H}}


\newcommand{\transpose}{{\!\top}}
\newcommand{\Da}[1]{D_{\!#1}}
\newcommand{\Dc}[1]{\D_{#1}}
\def\D{\nabla}
\def\del{\delta}
\def\ddel{\delta^2}
\def\dddel{\delta^3}

\def\loc{\textrm{loc}}

\def\qc{\textrm{qc}}
\def\c{\textrm{c}}
\def\h{\textrm{h}}

\def\eps{\varepsilon}
\def\tot{\textrm{tot}}
\def\cb{\textrm{cb}}
\def\a{\textrm{a}}
\def\c{\textrm{c}}
\def\ac{\textrm{ac}}
\def\i{\textrm{i}}
\def\nn{\textrm{nn}}
\def\refl{\textrm{rfl}}
\def\qnl{\textrm{qnl}}
\def\stab{\textrm{stab}}
\def\conv{\textrm{conv}}
\def\supp{\textrm{supp}}

\def\L{\Lambda}
\def\Is{\mathcal{I}}
\def\oIs{\ol{\Is}}
\def\As{\mathcal{A}}
\def\Cs{\mathcal{C}}
\def\Fs{\mathcal{F}}
\def\Ks{\mathcal{K}}
\def\Us{\mathscr{U}}
\def\Usz{\Us_0}
\def\Ush{\dot{\Us}^{1,2}}
\def\Ushd{\dot{\Us}^{-1,2}}
\def\Usp{{\Us}^{1,p}}
\def\Usc{\Us^c}

\def\Bs{\mathcal{B}}
\def\Ls{\mathcal{L}}
\def\bbL{\mathbb{L}}

\def\yF{y_\mF}
\def\uF{u_\mF}

\def\E{\mathcal{E}}
\def\Ea{\E^\a}
\def\Eb{\E^\textrm{b}}
\def\Ha{H^\a}
\def\Ei{\E^\i}
\def\Ec{\E^\c}
\def\Eh{\E_\h}
\def\F{\mathscr{F}}
\def\Hc{H^\c}
\def\Eqnl{\E^\qnl}
\def\Hqnl{H^\qnl}
\def\Erefl{\E^\refl}
\def\Hrefl{H^\refl}
\def\Eac{\E^\ac}
\def\Hac{H^\ac}
\def\Estab{\E^\stab}
\def\Hstab{H^\stab}

\def\dW{W'}
\def\ddW{W''}

\def\RO{\mathcal{R}}

\def\Es{\Phi}

\def\Eatot{\E^\a_\tot}
\def\Eatot{\E^\c_\tot}

\def\Om{{\R^d}}
\def\Vi{V^\i}
\def\Vc{V^\c}
\def\Vs{\mathscr{V}}

\def\tily{\tilde y}
\def\tilz{\tilde z}
\def\tilu{\tilde u}
\def\tilv{\tilde v}
\def\tile{\tilde e}
\def\tilw{\tilde w}
\def\tilf{\tilde f}

\def\bary{y}
\def\barz{z}
\def\barv{v}
\def\barw{w}
\def\baru{u}
\def\bare{e}
\def\barf{f}


\def\ve{\varepsilon}
\def\L{\Lambda}
\def\La{\L^{\a}}
\def\Li{\L^{\i}}
\def\Lc{\L^{\c}}
\def\yd{y_0}
\def\Nhd{\mathcal{N}}
\def\rcut{r_{\textrm{cut}}}
\def\Rg{\mathscr{R}}
\def\Rgp{\Rg^{+}}
\def\vsig{\varsigma}
\def\T{\mathcal{T}}
\def\Tp{T^+}
\def\Tm{T^-}
\def\np{\nu^+}
\def\nm{\nu^-}
\def\T{\mathcal{T}}
\def\Th{\mathcal{T}_\h}
\def\Ta{\T_\a}
\def\Te{\mathscr{T}_\varepsilon}
\def\Fc{\mathscr{F}_\c}
\def\Fi{\mathscr{F}_\i}
\def\Fh{\mathscr{F}_\h}
\def\UsT{\Us_\h}
\def\Ih{I_\h}
\def\Ia{I_\a}
\def\Ie{I_{\varepsilon}}
\def\Tmu{\mathscr{T}_\mu}
\def\vor\textrm{vor}
\def\s{\sigma}
\def\sh{\sigma^\h}
\def\sa{\sigma^\a}
\def\sc{\sigma^\c}
\def\sac{\sigma^\ac}
\def\Oma{\Omega^\a}
\def\PO{\textrm{P}_0}
\def\O{\Omega}

\def\OmDef{\Omega^{\rm DEF}}

	\def\interface{{\rm interface}}
	\def\M{\mathcal{M}}
	\def\Use{\Us^{1,2}}
	\def\Nh{\mathcal{N}_h}
	\def\Lhom{\Lambda^{\rm hom}}
	\def\b{\rm b}

	\newtheorem{assumption}{Assumption}[section]
	\newtheorem{example}{Example}[section]
	\renewcommand{\theequation}{\arabic{section}.\arabic{equation}}
	\renewcommand{\thetheorem}{\arabic{section}.\arabic{theorem}}
	\renewcommand{\theassumption}{\arabic{section}.\arabic{assumption}}
	\renewcommand{\thedefinition}{\arabic{section}.\arabic{definition}}
	\renewcommand{\theexample}{\arabic{section}.\arabic{example}}

	\begin{abstract}
		Adaptive quasicontinuum (QC) methods are important methodologies in molecular mechanics for the simulations of materials with defects, intending to achieve the optimal balance of accuracy and efficiency on the fly. In this study, we propose a residual-force based {\it a posteriori} error estimate that is {\it simple} and is {\it unified} for consistent quasicontinuum methods, as opposed to the widely adopted residual-stress based {\it a posteriori} error estimates which are complicated and need to be derived for the particular QC method under consideration. The {\it simple and unified} formulation of the estimator, together with certain sampling techniques, leads to a highly efficient and adaptable implementation. We also prove in theory that the unified error estimator provides an upper bound of the true error. We develop adaptive algorithms based on this unified estimator and validate the algorithms by several representative quasicontinuum methods for various types of crystalline defects, in terms of convergence and efficiency. In particular, the adaptive simulations of the anti-plane crack, of which we possess little {\it a priori} knowledge, demonstrate the necessity and significance of the proposed {\it a posteriori} estimates and the adaptive strategies.
	\end{abstract}
	

\section{Introduction}
\label{sec:intro}

In the past two decades, multiscale coupling methods, as an important class of concurrent multiscale methods, have attracted great attention from various academic communities, including those who focus on mechanics, material science, biochemistry, and mathematics~\cite{2003_RM_ET_QCM_JCAMD, 2004_SX_TB_Bridging_Domain_CMAME, 2005_ST_TH_WKL_Bridging_Sacle_IJNME, 2009_Miller_Tadmor_Unified_Framework_Benchmark_MSMSE, 2013_ML_CO_AC_Coupling_ACTANUM, 2016_EV_CO_AS_Boundary_Conditions_for_Crystal_Lattice_ARMA}. Quasicontinuum (QC) methods, also known as atomistic-to-continuum (a/c) coupling methods, form a typical class of multiscale coupling methods in molecular mechanics (MM) aiming to achieve the (quasi-)optimal balance of accuracy and efficiency for modeling crystalline defects~\cite{1989_Kohlhoff_Coupled_AtoMod_ASM, 1996_AC_Ana_Solid_Defect_PMA, 1999_VS_RM_ETadmor_MOrtiz_AFEM_QC_JMPS, 2001_Knap_Ortiz_QCM_JMPS, 2004_Shimokawa_QCM_ErrAna_PRB, 2006_WE_JL_ZY_GRC_PRB, 2007_PL_QCL_2D_SIAMNUM, 2008_SP_BD_PB_TO_Anal_Coupling_Error_Arlequin_CMAME, 2013_ML_CO_BK_BQCE_CMAME, 2013_QY_EB_AT_Multiresolution_MM_CMAME, 2014_CO_LZ_GRAC_Coeff_Optim_CMAME, 2014_DO_PB_ML_Opt_Based_AtC_SINNUM, 2015_JA_GV_DK_Summation_QCE_JMPS, 2016_CO_LZ_GForce_Removal_SISC}. The fundamental idea behind the QC methods is to employ the relatively more accurate molecular mechanics model in the immediate vicinity of localized defects (also termed as the atomistic region), while utilizing the continuum model, such as the Cauchy-Born approximation, in the elastic far field (also termed as the continuum region) .

The modeling and {\it a priori} analysis for different QC methods have received extensive and comprehensive investigations \cite{2008_MD_ML_Ana_Force_Based_QC_M2NA, 2009_PM_ZY_1D_QC_Nonlocal_MMS, 2009_MD_ML_Optimal_Order_SIMNUM, 2011_CO_1D_QNL_MATHCOMP, 2011_CO_HW_QC_A_Priori_1D_M3AS, 2012_JL_PM_Convergence_BQCF_3D_No_Defects_CPAM, 2012_CO_LZ_GRAC_Construction_SIAMNUM, 2013_CO_AS_ACC_2D_MATHCOMP, 2014_CO_AS_LZ_Stabilization_MMS, 2016_XL_CO_AS_BK_BQC_Anal_2D_NUMMATH, 2016_CO_LZ_GForce_Removal_SISC, 2016_EV_CO_AS_Boundary_Conditions_for_Crystal_Lattice_ARMA, 2021_LF_LZ_3D_BGFC_CICP}. We also refer to~\cite{2009_Miller_Tadmor_Unified_Framework_Benchmark_MSMSE, 2020_EG_Roadmap_Multiscale_Modeling_MSMSE} for the extensive overview and benchmarking, and~\cite{2013_ML_CO_AC_Coupling_ACTANUM} for the framework of rigorous {\it a priori} error analysis. However, it is equally important for the QC methods to be self-adaptive. This is because the material systems may be so complicated that there is no enough {\it a priori} knowledge to guide the allocation of the atomistic and continuum regions as well as the construction of the appropriate mesh structures so that the optimal balance between accuracy and efficiency is achieved. Therefore, the {\it a posteriori} estimates and the corresponding adaptive algorithms are essential for the real-world applications of the QC methods \cite{2016_IT_JA_LM_DK_QCM_Coarse_Grain_IJNME, 2016_EB_TC_Adapt_Molecular_Mechanics_CMAME}.

There are two major approaches for the theory based {\it a posteriori} error estimates and both of which have been employed by the adaptive QC methods. The first one is the goal-oriented approach which aims for {\it a small error in a given functional of the solution}~\cite{2007_DB_FEM} often known as the {\it quantities of interest} \cite{2000_MA_TO_a_Post_Est_FEM, 2003_WB_RR_Adaptive_FEM, 2002_TO_SP_Est_Modeling_Error_JCP, 2006_TO_SP_AR_PB_MM_Adaptive_Model_SISC}. Adaptive QC methods based on such approach have been applied in the adaptive simulations of the one-dimensional Frenkel-Kontorova model \cite{2007_MA_ML_Goal_Oriented_Adaptive_AC_IJMCE, 2008_MA_ML_Adaptiv_AC_FK_Model_MMS, 2008_MA_ML_Goal_Oriented_Mesh_Refinement_AC_CMAME} and the three dimensional nanoindentation problems \cite{2006_SP_PB_TO_Error_Control_Molecular_Statics_IJMCE}, and later the adaptive modeling error control for the method based on the Arlequin framework \cite{2008_SP_BD_PB_TO_Anal_Coupling_Error_Arlequin_CMAME, 2009_SP_LC_BD_PB_Adaptive_AC_Model_Error_2D_CMAME}. The primary disadvantage of this approach is that a good estimate in the quantities of interest may not be easy to obtain, which results in the difficulty for proper stopping criterions of the adaptive processes \cite{2011_BD_LC_TO_SP_Adaptive_AC_Optimal_Control_CMAME}. 

The present work concentrates on the second approach which is the residual based {\it a posteriori} error estimate that aims for {\it a small norm of the error}~\cite{2007_DB_FEM}. This approach was first adopted for the {\it a posteriori} error control problem for a consistent QC method in one dimension  in~\cite{2011_CO_1D_QNL_MATHCOMP, 2014_CO_HW_A_Post_ACC_IMANUM}. It was then extended to geometric reconstruction atomistic-to-continuum (GRAC) method for two dimensional point defects with nearest neighbor interactions in~\cite{2018_HW_ML_PL_LZ_A_Post_GRAC_2D_SISC} and later generalized to finite range interactions in~\cite{2020_ML_PL_LZ_Finite_Range_A_Post_2D_CICP}. However, it was later noticed that the residual based estimates suffer from high computational cost which essentially stems from the mismatch of the lattices and the meshes for models of different scales. Our recent effort has been devoted to dealing with this problem and results in a significant improvement in efficiency and extensions of adaptivity simulations to more realistic defects such as dislocations and interactive vacancies \cite{2023_YW_HW_Efficient_Adaptivity_AC_JSC}.

Despite the possible further extensions of the existing residual based approach to a wider range of QC methods and types of defects, there are critical limitations for such approach. The first one is that the {\it a posteriori} error estimator depends on the specific coupling scheme which leads to little reusability of the code and thus the inefficiency of implementation. The second one is that the formulations of the estimators are often involved which results in the high cost in both implementation and computation. In particular, the residual based {\it a posteriori} error estimate for the blended QC methods, which are widely adopted in multiscale science and engineering \cite{2003_TB_SX_Couple_CM_MM_IJMCE, 2004_SX_TB_Bridging_Domain_CMAME, 2006_WL_HP_DQ_EK_HK_GW_Bridge_Scale_CMAME, 2008_PB_BD_NE_TO_SP_Arlequin_AC_CM, 2008_SB_MP_PB_MG_AC_Blending_MMS, 2020_EG_Roadmap_Multiscale_Modeling_MSMSE}, are absent in more than one dimension due to such complication.



The purpose of the current work is to address the critical limitations just described. By deriving an error estimator based on the {\it residual force}, we are able to give a unified {\it a posteriori} error estimate for general consistent QC methods. The unified error estimator is then employed to design the adaptive algorithms for several representative QC methods. These algorithms are then implemented for the adaptive simulations of a few crystalline defects with practical importance. To be precise, our contribution lies in the following aspects.

First, we derive a unified {\it a posteriori} error estimate based on the {\it residual force} as apposed to previous estimates which are based on the {\it residual stress} \cite{2011_CO_1D_QNL_MATHCOMP, 2014_CO_HW_A_Post_ACC_IMANUM, 2018_HW_ML_PL_LZ_A_Post_GRAC_2D_SISC, 2020_ML_PL_LZ_Finite_Range_A_Post_2D_CICP, 2023_YW_HW_Efficient_Adaptivity_AC_JSC}. The residual-force based estimate is independent of the particular QC method we consider and we provide a theoretical framework to establish that such estimator provides an upper bound of the true error in a discrete $H^1$-norm. 

Second, we develop novel adaptive algorithms for a series of QC methods ranging from sharp interfaced to blended, energy-based to force-based. The algorithms aim to automatically adjust the mesh structures in the continuum region as well as the allocations of the atomistic region and the blending region, if applicable, on the fly. All the adaptive algorithms are essentially based on the force based unified error estimator and only algorithmic adjustments are needed. We note that theory based {\it a posteriori} error estimate and fully adaptive algorithms (taking both mesh refinement and allocation of the regions into account at the same time) for the widely adopted blended QC methods \cite{2016_XL_CO_AS_BK_BQC_Anal_2D_NUMMATH, 2013_ML_CO_BK_BQCE_CMAME, 2018_DO_XL_CO_BK_Force_Based_AC_Complex_Lattices_NUMMATH} are developed for the first time in more than 1D (cf. \cite{2019_HW_SL_FY_A_Post_QCF_1D_NMTMA, 2023_YW_LZ_Adaptive_Multigrid_AC_3D_JCP}). 

%

Third, we provide numerical validations of our error control strategies in the adaptive simulations for various types of crystalline defects, including crack, which to the best knowledge of the authors is implemented for the first time in the context of QC. The adaptive simulations demonstrate that the proposed algorithms produce optimal convergence rates of the error and (quasi-)optimal decomposition of the domain. The systematic study of the adaptive QC methods for cracks represents a pioneering effort, leveraging the attributes of the unified residual-force based {\it a posteriori} error estimator. Conversely, the {\it a priori} analysis poses challenges that require advanced technicalities to be addressed effectively. The inherent effectiveness and remarkable flexibility of this approach not only broaden the horizons but also provide a robust framework for investigating practical crystalline defects such as grain boundaries.

To ensure the clarity of presentation, we focus on the atomistic system with finite range interactions in two dimensions. However, the proposed unified framework may be potentially extended to other consistent multiscale coupling schemes and three-dimensional problems. Further investigations including more efficient strategies for other coupling schemes and the extension to realistic crystalline defects in three dimensions, such as partial dislocations and grain boundaries, are also discussed at the end of this paper.

\subsection{Outline}
\label{sec:sub:outl}
The current work is structured as follows. In Section~\ref{sec:formu}, we introduce the mathematical concept of crystalline defects and the general formulations of the atomistic and the quasicontinuum methods. The {\it a priori} estimates and the definition of consistency of the QC methods are also given in Section \ref{sec:sub:a_priori}. In Section~\ref{sec:err} we systematically derive the residual-force based error estimate along with the proof that the estimator provides an upper bound of the true approximation error. In addition, an efficient method for the evaluation of the local error contributions is given in Section~\ref{sec: local error contribution}. In Section~\ref{sec:alg}, we develop the adaptive algorithms for sharp interfaced and blended QC methods respectively based on the error estimator just derived. We note that the algorithm for adaptively allocating the blending region is developed for the first time in the literature. In Section~\ref{sec:numerics}, we present the adaptive simulations for several crystalline defects we consider and provide a comprehensive discussion and explanation of our findings. We conclude by summarizing our results and discussing possible future directions in Section~\ref{sec:con}.


\subsection{Notations}
\label{sec:sub:not}

We use the symbol $\langle\cdot,\cdot\rangle$ to denote an abstract duality
pairing between a Banach space and its dual space. The symbol $|\cdot|$ normally
denotes the Euclidean or Frobenius norm, while $\|\cdot\|$ denotes an operator
norm.
%
For $E \in C^2(X)$, the first and second variations are denoted by
$\<\delta E(u), v\>$ and $\<\delta^2 E(u) v, w\>$ for $u,v,w\in X$.
%
For second order tensors $\mA$ and $\mB$, we denote $\mA:\mB :=\sum_{i,j}\mA_{ij}\mB_{ij}$.
The closed ball with radius $r$ and center $x$ is denoted by $B_r(x)$, or $B_r$ if the center is the origin.


\section{The Atomistic Model and Quasicontinuum Methods}
\label{sec:formu}

In this section, we provide an exposition of the atomistic model and a very general formulation of the quasicontinuum methods which could accommodate a wide range of different QC methods ~\cite{2013_ML_CO_AC_Coupling_ACTANUM, 2016_XL_CO_AS_BK_BQC_Anal_2D_NUMMATH, 2016_CO_LZ_GForce_Removal_SISC, 2013_ML_CO_BK_BQCE_CMAME, 2021_LF_LZ_3D_BGFC_CICP}. We defer the presentation of the detailed formulations of specific methods to Section~\ref{sec:alg} where the adaptive strategies are developed, so that the coupling schemes and the corresponding algorithms are better related. For the sake of simplicity, we consider the single-species Bravais lattices but note that both the analysis and the algorithms given in this work can be applied to {\it multilattice} crystals \cite{2018_DO_XL_CO_BK_Force_Based_AC_Complex_Lattices_NUMMATH} with suitable minor modifications.

\subsection{Atomistic model}
\label{sec:sub:atomic}

\def\Rcore{R_{\rm DEF}}
\def\UsH{{\mathscr{U}}^{1,2}}
\def\Adm{{\rm Adm}}
\def\Lhom{{\Lambda^{\rm hom}}}

We first give a heuristic introduction to the problem we solve. Let $\L\subset \R^2$ be an infinite single lattice with defects and $y$ be a lattice function defined on $\L$ representing a deformed configuration. In molecular mechanics, we consider an atomistic energy functional of the form 
\begin{align}\label{eq: Ea_y}
\Ea(y) :=\sum_{\ell\in\L} \Phi_\ell\big(\{y(\ell+\rho)-y(\ell)\}_{\rho \in \Rg_\ell}), 
\end{align}
where $\Phi_{\ell}:(\R^2)^{\Rg_\ell}\rightarrow\R$ is the energy distributed to each atomic site or simply the site potential, and $\Rg_\ell := \{\ell'-\ell ~|~ \ell'\in \Nhd_\ell\}$ is the interaction range with an interaction neighborhood  $\Nhd_{\ell} := \{ \ell' \in \L ~|~ 0<|\ell'-\ell| \leq \rcut \}$ of a given cut-off radius $\rcut$. The atomistic problem is to locally minimize $\Ea(y)$, i.e., by denoting ``$\arg\min$'' as the set of local minimizers, we seek $y^{\rm a}$ such that
\begin{equation}\label{eq: min_Ea_y}
y^{\rm a} \in \arg\min \big\{ \E^{\rm a}(y) \big\}.
\end{equation}

However, \eqref{eq: min_Ea_y} is in fact not a meaningful problem since $\L$ is infinite. A more mathematically rigorous description of the problem is given as follows. 

The defected lattice $\L$ stems from the perfect single lattice $\Lhom:=\mA\Z^2$ for some non-singular matrix $\mA\in\R^{2\times 2}$.  The mismatch between $\L$ and $\Lhom$ represents possible defects that are often contained in some localized defect cores $\OmDef$ so that $\L \backslash \OmDef  = \Lhom \backslash  \OmDef$. We illustrate three typical defects considered in this work in Figure \ref{fig:geom_defects}.
\begin{figure}[!htb]
	\centering 
	\subfloat[Micro-crack]{
		\label{fig:geom_mcrack}
		\includegraphics[height=5cm]{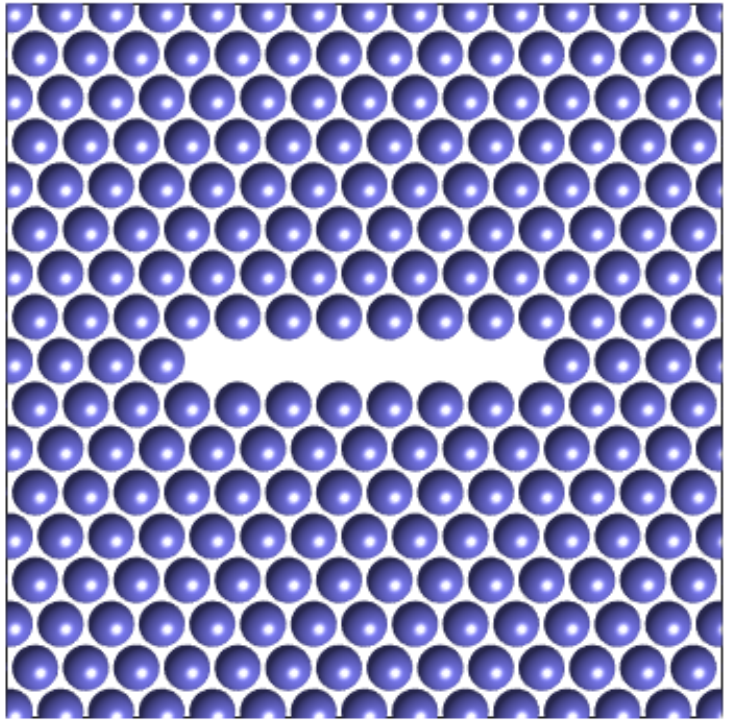}}
		~
		\subfloat[Anti-plane screw dislocation]{
		\label{fig:geom_screw} 
		\includegraphics[height=5cm]{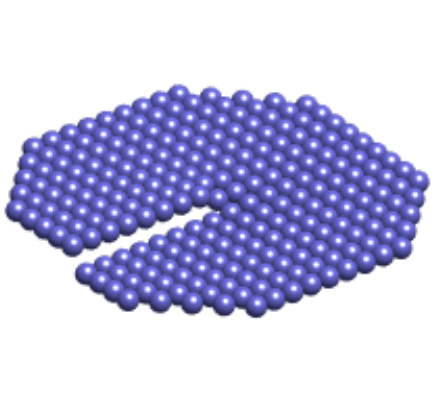}}
		~
	\subfloat[Anti-plane crack]{
		\label{fig:geom_multivac}
		\includegraphics[height=5cm]{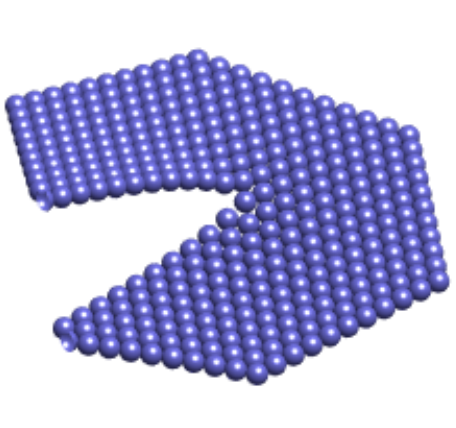}}		
	\caption{Illustration of three typical defects considered in this work.}
	\label{fig:geom_defects}
\end{figure} 

We often decomposed the deformed configuration $y \in \Us$ with $\Us:=\{v:\L\to \mathbb{R}^2\}$ as
\begin{eqnarray}\label{y-u}
y(\ell) = \mF \ell + u_0(\ell) + u(\ell),  \qquad\forall~\ell\in\Lambda,
\end{eqnarray}
where $\mF$ is a macroscopic deformation gradient, $u_0\in\Us$ is a {\it far-field predictor} resulting from the presence of the defect which is often precomputed by solving a continuum linearized elasticity (CLE) (see \ref{sec:appendixU0}), and $u\in\Us$ is a {\it corrector} which is the solution we actually seek for. 

To measure the local ``regularity" of a displacement $u\in\Us$, it is convenient to use a background mesh, for example, a {\it canonical} triangulation $\T_{\L}$ of $\R^2$ whose nodes are the reference lattice sites in $\L$, see \cite[Figure 1]{2016_EV_CO_AS_Boundary_Conditions_for_Crystal_Lattice_ARMA} for an illustration. We define $I_{\a} u$ as the standard piecewise affine interpolation of $u$ with respect to $\T_{\L}$ \cite[Appendix 1]{2018_HW_ML_PL_LZ_A_Post_GRAC_2D_SISC}. 
To be more precise, if we let $\zeta_{\ell}(x) \in W^{1, \infty}(\mathbb{R}^2; \mathbb{R})$ be the $\mathcal{P}_1$ nodal basis function associated with $\T_{\L}$, we can then extend all lattice displacement $u\in\Us$ via their nodal interpolants,
\begin{eqnarray}
\label{def: I_a}
I_{\a} u(x) := \sum_{\ell\in\L} u(\ell) \zeta_{\ell}(x).
\end{eqnarray}
When no confusion arises, we identify $u$ with its interpolation $I_{\rm a}u$ and denote the piecewise constant gradient $\nabla u := \nabla I_{\rm a} u$. We then introduce the functional space of finite-energy displacements
\begin{align}\label{space:UsH}
\UsH(\L) := \big\{u:\L\rightarrow\mathbb{R}^{d} ~\big\lvert~ \|\nabla u\|_{L^2} < \infty \big\}
\end{align}
with the associated norm $\|v\|_{\UsH} := \|\nabla v\|_{L^2}$ for $v \in \UsH(\L)$.
We also define the following subspace of compact displacements which is dense in $\UsH$ \cite[Proposition 3.3]{2013_ML_CO_AC_Coupling_ACTANUM}:
\begin{align}\label{space:Uc}
\Us^{\rm c}(\L) := \big\{u:\L\rightarrow\mathbb{R}^{d} ~\big\lvert~ \textrm{supp}(\nabla u)~~\textrm{is compact} \big\}.
\end{align}
By introducing the finite difference stencil for $u\in\Us$
\begin{align*}
Du(\ell):= \{D_\rho u(\ell)\}_{\rho \in \Rg_\ell} :=\{u(\ell+\rho)-u(\ell)\}_{\rho \in \Rg_\ell},
\end{align*}
and with a slight abuse of notation, we can redefine the atomistic energy-difference functional by
\begin{align}\label{energy-difference}
\Ea(u) &=~\sum_{\ell\in\Lambda}\Big(V_{\ell}\big(Du_0(\ell)+Du(\ell)\big)-V_{\ell}\big(Du_0(\ell)\big)\Big) =:\sum_{\ell\in\Lambda} V'_\ell\big(Du(\ell)\big), 
\end{align}
where $V_{\ell}(Du(\ell)):= \Phi(\{\mF \rho + u(\ell+\rho)-u(\ell)\}_{\rho \in \Rg_\ell})$ and $V'_{\ell}$ is the {\it first renormalization} of $V_{\ell}$ \cite{2016_CO_LZ_GForce_Removal_SISC}. 

It was shown that $\E^{\a}(u)$ is well-defined on the space $\UsH$ under suitable assumptions \cite[Section~2.1]{2013_ML_CO_AC_Coupling_ACTANUM}. The atomistic problem can then be reformulated as 
\begin{equation}\label{eq:variational-problem}
u^{\rm a} \in \arg\min \big\{ \E^{\rm a}(u)~|~ u \in \Us^{1,2}(\L) \big\},
\end{equation}
with the following first and second order optimality conditions
\begin{eqnarray}\label{eq:optimality}
\big\< \delta\Ea(u) , v\big\> = 0 , \qquad \big\< \delta^2\Ea(u) v , v\big\> \geq 0, \qquad\forall~v\in\UsH(\L) .
\end{eqnarray}
To establish the {\it a posteriori} error bounds, a stronger stability assumption is made:
\begin{assumption}\label{as:LS}
	$\exists~\gamma>0$, such that for all $v\in\UsH(\L)$, 
		$\big\< \delta^2\E^{\rm a}(u) v , v\big\> \geq \gamma \|v\|^2_{\UsH}$.
\end{assumption}
%

\subsection{The general formulation of the quasicontinuum methods}
\label{sec:sub:gen_ac}

Since the atomistic model \eqref{eq:variational-problem} is defined on an infinite domain and considers every atom as a degree of freedom with {\it nonlocal} interactions (the energy of an atom is influenced by all neighbors within its interaction range), it is not solvable in practice. To make the problem efficiently computable, the QC methods \cite{2013_ML_CO_AC_Coupling_ACTANUM, 2016_XL_CO_AS_BK_BQC_Anal_2D_NUMMATH, 2016_CO_LZ_GForce_Removal_SISC} aim to construct and solve the following approximated variational nonlinear system 
\begin{eqnarray}\label{eq:variational-problem-AC}
\<\mathcal{F}_h^{\rm qc}(u_{h}), v_h\>_{\ac} = 0, \quad \forall v_h \in \Us_h,
\end{eqnarray}
by certain approximations which are described immediately.

The first approximation is the restriction of the computational domain from infinite to a simply connected polygon $\Omega \subset \R^d$. The error committed by such restriction is called the {\it truncation error}. 

The second approximation is essentially a model reduction. Recall that the displacement is smooth in the region away from the defect core so that the {\it nonlocal} interactions may be approximated by a {\it local} continuum elasticity model in the far field with sufficient accuracy. A typical choice of the model is the Cauchy-Born nonlinear elasticity ~\cite{2007_EM_CB_StaCry_AMAS, 2013_CO_FT_Cauchy_Born_ARMA}, where the strain energy function $W:\R^2\rightarrow\R^2$ measures the energy per unit volume under a globally homogeneous deformation and is defined by
\begin{equation}\label{eq:W}
  W(\mF) := \det (\mA^{-1}) \cdot V(\mF \Rg).
\end{equation}
The definition of the {\it first renormalization} of $W$ is then given by $W'(\mF) := W(\mF) - W({\bf 0})$. 

The third approximation is a domain decomposition with further coarse-graining by the finite element discretization. We essentially decompose the computational domain $\Omega$ into the {{\it atomistic region}} $\Omega_{\a}$, the {{\it interface region}} $\Omega_{\rm i}$ and the {{\it continuum region}} $\Omega_{\c}$, so that $\Omega=\Omega_{\a} \cup \Omega_{\rm i} \cup \Omega_{\c} \subset \R^2$. We assume that the defect core is contained in the atomistic region, i.e., $\OmDef \subset \Omega_{\a}$ and the displacement $u$ is sufficiently smooth in $\Omega_{\c}$. Here $\Omega_{\rm i}$ is the region between $\Omega$ and $\Omega_{\c}$ where the transition of the {\it nonlocal} and the {\it local} models takes place. The set of core atoms can then be defined by $\L_{\a} := \L \cap \Omega_{\a}$, the set of the interface atoms by $\L_{\rm i} := \L \cap \Omega_{\rm i}$ and the set of atoms in the continuum region by $\L_{\rm c} := \L \cap \Omega_{\rm c}$. Subsequently, we can define $\T_{\rm a}$ to be the {\it canonical} triangulation induced by $\L_{\a}\cup \L_{\rm i}$, $\T_{\rm c}$ be a shape-regular simplicial partition of the continuum region and $\T_{h} = \T_{\a} \cup \T_{\c}$ be the triangulation of the computational domain for the QC methods.  

According to the three approximations, we define the {\it coarse-grained} space of displacements by
\begin{eqnarray}
\label{eq: ac_solution_space}
  \Us_{h} := \big\{ u_{h} : \Omega_{h} \to \R^2 ~\big|~
  \text{$u_{h} \in \mathcal{P}_1(\T_h)$, $u_{h} = 0$ in $\R^d \setminus \Omega_{h}$}\big\},
\end{eqnarray}
where $\Omega_h = \bigcup_{T\in \T_h}T$ if a $\mathcal{P}_1$ finite element discretization is utilized in the continuum region so that $\Us_h$ serves as an approximation of $\UsH(\L)$. 

We now illustrate the construction of $\mathcal{F}^{\rm qc}_h$ by an example of the energy-based QC methods \cite{2012_CO_LZ_GRAC_Construction_SIAMNUM, 2013_ML_CO_BK_BQCE_CMAME}. For $u_h\in\Us_h$, we define the QC energy functional
\begin{align}\label{eq:general_E_ac}
  \E^{\rm qc}_{h}(u_{h}) :=& \sum_{\ell \in \L_{\a}} V'_\ell(Du_{h}) + \sum_{\ell \in \L_{\rm i}} V^{\rm i}_\ell(Du_{h}) + \sum_{T \in \T_{c}} \omega_T \big(W'(\nabla u_{h}|_T+\nabla u_0|_T) - W'(\nabla u_{0}|_T)\big), 
\end{align}
where $V^{\rm i}_\ell \in C^k((\R^2)^{\mathcal{R}})$ is an {\it interface potential} and $\omega_T$ is the {\it effective volume} of $T$ which specify individual coupling schemes and are subtly defined to deal with the spurious forces (usually referred to as ``ghost force") around the interface where the {\it nonlocal} and {\it local} models meet. In this context, the nonlinear map $\mathcal{F}^{\rm qc}_h:\Us_h \rightarrow (\Us_h)^{*}$ in \eqref{eq:variational-problem-AC},  which is defined by the first variation of the coupling energy functional $\E^{\rm qc}_{h}$ with respect to the space $\Us_h$ that $\mathcal{F}_h^{\rm qc}:=\delta\E^{\rm qc}_h$, serves as an approximation of the first order optimality of the atomistic model \eqref{eq:optimality} and the duality $\<\cdot, \cdot\>_{\ac}$ is then induced by $\Us_h$. The energy-based QC problem we would like to solve is to find
\begin{equation}\label{eq:gen_E}
u^{\rm qc}_{h} \in \arg\min \big\{ \E^{\rm qc}_{h}(u_{h}),~u_{h} \in \Us_{h} \big\}.
\end{equation}
The solution to the energy minimization problem in \eqref{eq:gen_E} then satisfies the variational problem in \eqref{eq:variational-problem-AC}. However, it is important to note that the reverse is only true when the corresponding coupling scheme is stable~\cite{2014_CO_AS_LZ_Stabilization_MMS, 2013_ML_CO_AC_Coupling_ACTANUM}. 

We note that the general formulation \eqref{eq:variational-problem-AC} fits most (if not all) of the quasicontinuum methods. However, since our {\it a posterior} error estimates in Section \ref{sec:err} {\it do not depend on} the individual formulation of a coupling scheme, we postpone the construction detail of specific QC schemes to Section \ref{sec:alg} where the adaptive algorithms, which {\it do depend on} the corresponding coupling scheme, are developed. What essentially makes our {\it a posterior} error estimates possible is the {\it a priori} estimate/consistency of the QC method which we introduce in the immediate following section.

\subsection{A priori estimates and consistency of the quasicontinuum methods}
\label{sec:sub:a_priori}

Let $N$ denote the total number of degrees of freedom of the QC method. We say that the method is consistent if the approximation error decreases as $N$ increases. To be more precise, under certain assumptions (e.g., mesh qualities~\cite[Assumption 1]{2021_LF_LZ_3D_BGFC_CICP}), for $N$ sufficiently large, 
\begin{eqnarray}\label{eq:ap_gen_ac}
    \| u -  u^{\rm qc}_{h}\|_{\UsH}
     \leq C N^{-k}(\log N)^s,
\end{eqnarray}
where $C$ is a constant independent of any model parameters, and $s, k\geq0$ are related to the types of defects and specify the convergence rate of the approximation error with respect to $N$. We will present the {\it a priori} estimates for different methods and defects considered in this work in Table \ref{tab: decay rate} in Section~\ref{sec:alg}. We note that these {\it a priori} estimates provide a solid theoretical foundation of the current work and plays a key role in developing the {\it a posteriori} error estimate.


\section{The Unified a Posteriori Error Estimates based on Residual Forces}
\label{sec:err}

In this section, we propose a unified framework for the {\it a posteriori} error estimate which is applicable to generally consistent QC methods satisfying~\eqref{eq:ap_gen_ac}. Such estimate provides the groundwork for the adaptive algorithms for a range of QC methods in Section~\ref{sec:alg}. We essentially adapt the analytical framework originally introduced by our previous study on the adaptive QM/MM methods for simple point defects~\cite{2019_HC_ML_HW_YW_LZ_Adaptive_QMMM_CMAME} to the setting of QC and extend it to substantially more complicated defects. The framework can be summarized in four steps: the proof of the equivalence of the true approximation error and the residual in certain dual norm, the derivation of the {\it a posteriori} error estimator based on the residual forces which gives the upper bound of the residual, the estimate of the error committed by the truncation of the computational domain for different types of defects and an efficient sampling of the residual forces which forms the local error contributions. We present the four steps in detail in the following subsections. 


\subsection{Equivalence of the residual and the approximation error}
\label{sec:sub:residual_forces}

\def\sjump{\llbracket\sigma^{\b}_{\h}\rrbracket_f}
\def\dx {\rm dx}
\def\Fh {\mathcal{F}_h}
\def\Fhi{\Fh\bigcap {\rm int}(\Omega_R)}

For the simplicity of presentation, we denote the solution to the atomistic problem \eqref{eq:variational-problem} as $u$ and that to the QC problem \eqref{eq:variational-problem-AC} or \eqref{eq:gen_E} as $u_{h}$. Recall that $\Us^{\rm c}$ is dense in $\UsH$, we have $I_{\a}u_h\in\UsH$ where $I_{\a}$ is the standard piecewise affine interpolation operator defined in \eqref{def: I_a}. The residual $\mR(I_{\rm a} u_{h})$ is then defined as an operator on $\Use$ such that
\begin{equation}\label{eq:def_res}
\mR(I_{\rm a} u_{h})[v] := \< \del\Ea(\Ia u_{h}), v\>, \quad \forall v\in \Use.
\end{equation}
The theorem below shows that the dual norm of the residual $\mR(I_{\rm a} u_{h})$ gives both the upper and the lower bounds for the true approximation error $\|u - I_{\rm a} u_{h}\|_{\Use}$. 

\begin{theorem}\label{thm:res}
	Let $u$ be a strongly stable solution of \eqref{eq:variational-problem}. Suppose that $\delta\E^{\rm a}$ and $\delta^2\E^{\rm a}$ are Lipschitz continuous in $B_{r}(u)$ with uniform constants $L_1$ and $L_2$, the QC method is consistent in the sense of \eqref{eq:ap_gen_ac} and $u_{h}$ satisfies the same decay as $u$. Then for $N$ sufficiently large, there exist constants $c$ and $C$ independent of the approximation parameters such that 
	\begin{eqnarray}\label{res-bound}
	c\|u - I_{\rm a} u_{h} \|_{\UsH} \leq \| \mathsf{R}(I_{\rm a} u_{h}) \|_{(\UsH)^*} \leq C\|u - I_{\rm a} u_{h}\|_{\UsH},
	\end{eqnarray}
	where the residual $\mathsf{R}(I_{\rm a} u_{h})$ is defined by \eqref{eq:def_res}.
\end{theorem} 
\begin{proof}

Let $r>0$ satisfies $B_r(u) \subset \UsH$. By the Lipschitz continuity of $\delta\Ea$ and $\delta^2\Ea$ in $B_{r}(u)$ with uniform constants $L_1$ and $L_2$, for any $w \in B_{r}(u)$, we have
\begin{align*}
    \|\delta\Ea(u)-\delta\Ea(w)\| \leq L_1\|u-w\|_{\UsH}, \\
    \|\delta^2\Ea(u)-\delta^2\Ea(w)\| \leq L_2\|u-w\|_{\UsH}.
\end{align*}
By the first optimality condition of the atomistic problem $\<\delta\Ea(u), v\>=0$, we have that for $N$ sufficiently large
\[
\mR(I_{\rm a} u_{h})[v] = \< \del\Ea(\Ia u_{h}) -\del\Ea(u), v\> \leq L_1 \|u - I_{\rm a} u_{h}\|_{\UsH} \|v\|_{\UsH}, \quad \forall v\in \Use.
\]

To estimate $\|u - I_{\rm a}u_h\|_{\UsH}$, we obtain from the {\it a priori} estimate \eqref{eq:ap_gen_ac} that 
\begin{align}\label{eq:u_Iauh}
\|u - I_{\rm a}u_h\|_{\UsH} &\leq \|u - u_h\|_{\UsH} + \|u_h - I_{\rm a}u_h\|_{\UsH} \lesssim N^{-k}(\log N)^s,
\end{align}
where the last inequality holds by the assumption that $u_h$ and $u$ share the same decay estimate.

The above result leads to the lower bound of the true approximation error
\begin{eqnarray}\label{eq:lo}
 \| \mR(I_{\rm a} u_{h}) \|_{(\UsH)^*} \leq L_1 \|u - I_{\rm a} u_{h}\|_{\UsH} =:C \|u - I_{\rm a} u_{h}\|_{\UsH}.
\end{eqnarray}

By considering $\mR(I_{\rm a} u_{h})$ as an operator on $\Use$ and the Galerkin orthogonality introduced by the first optimality condition of the atomistic problem, we obtain the upper bound
\begin{align}
\label{eq:up1}
\| \mR(I_{\rm a} u_{h}) \|_{(\UsH)^*}\|u - I_{\rm a} u_{h}\|_{\UsH} \geq \<\delta\Ea(I_{\a}u_{h}), u - I_{\a}u_{h}\> = \<\delta\Ea(I_{\a}u_{h})-\delta\Ea(u), u - I_{\a}u_{h}\>.
\end{align}
A further application of the intermediate value theorem yields
\begin{align}
\label{eq:up2}
\<\delta\Ea(I_{\a}u_{h})-\delta\Ea(u), u - I_{\a}u_{h}\> = \<\delta^2\Ea(w)(u - I_{\a}u_{h}), u - I_{\a}u_{h}\> ,
\end{align}
where $w=tu+(1-t)I_{\a}u_h$ for some $t\in(0,1)$. To analyze the second variation, we add and subtract the same term and separate it into two parts
\begin{align}\label{eq:sp_stab}
    \<\delta^2\Ea(w)v, v\> &= \<\delta^2\Ea(w)v, v\> - \<\delta^2\Ea(u)v, v\> + \<\delta^2\Ea(u)v, v\> \nonumber \\
    &=: S_1 + S_2. 
\end{align}
For $N$ sufficiently large (e.g. $N\geq (\frac{\gamma}{2}L_2)^k$), we bound $S_1$ by the Lipschitz continuity of $\delta^2\Ea$ that
\begin{eqnarray}\label{eq:stab_s1}
    \big|S_1\big| \leq L_2 \|u - w\|_{\UsH} \|v\|^2_{\UsH} \leq \frac{\gamma}{2} \|v\|^2_{\UsH},
\end{eqnarray}
where the second inequality holds by the fact that $\|u-w\|_{\UsH}\leq \|u-I_{\a}u_h\|_{\UsH}$ and \eqref{eq:u_Iauh}. Since $u$ is a strongly stable solution of the atomistic problem \eqref{eq:variational-problem} satisfying Assumption \ref{as:LS}, we obtain the following stability estimate by \eqref{eq:sp_stab} and \eqref{eq:stab_s1} that
\begin{eqnarray}\label{eq:stab_w}
\<\delta^2\Ea(w)v, v\> \geq \frac{\gamma}{2}\|v\|^2_{\UsH}.
\end{eqnarray}

%

Let $v=u-I_{\rm a}u_h$ in \eqref{eq:stab_w}, we have
\begin{eqnarray}\label{eq:up3}
 \<\delta^2\Ea(w)(u - I_{\a}u_{h}), u - I_{\a}u_{h}\> \geq \frac{\gamma}{2} \|u - I_{\rm a} u_{h}\|^2_{\UsH}.
\end{eqnarray}
Combining \eqref{eq:up1}, \eqref{eq:up2} and \eqref{eq:up3} and dividing both sides by $\|u - I_{\rm a} u_{h}\|_{\UsH}$ lead to the upper bound of the true approximation error
\begin{eqnarray}\label{eq:up}
 \| \mR(I_{\rm a} u_{h}) \|_{(\UsH)^*} \geq \gamma/2 \|u - I_{\rm a} u_{h}\|_{\UsH} =:c \|u - I_{\rm a} u_{h}\|_{\UsH}.
\end{eqnarray}
The stated results are obtained by \eqref{eq:lo} and \eqref{eq:up}.
\end{proof}
Theorem~\ref{thm:res} establishes the equivalence of the residual $\mR(I_{\rm a} u_{h})$ in the dual norm $\| \cdot \|_{(\UsH)^*}$ and the true approximation error $\|u - I_{\rm a} u_{h}\|_{\Use}$. We note that such equivalence is independent of the coupling scheme by which $u_{h}$ is obtained. The conditions on which the equivalence does rely are essentially the following three: the Lipschitz continuity of the derivatives or the variations of $\Ea$, the consistency of the coupling method defined by \eqref{eq:ap_gen_ac} and the assumption that $u_h$ shares the same decay estimates as $u$. The first condition is often guaranteed by the regularity of the (empirical) interatomic potential (cf. \cite[Lemma 2.1]{2016_EV_CO_AS_Boundary_Conditions_for_Crystal_Lattice_ARMA}). The second condition is the focus of the {\it a priori} analysis of any QC method (cf. \cite{2013_ML_CO_AC_Coupling_ACTANUM, 2016_XL_CO_AS_BK_BQC_Anal_2D_NUMMATH}). The third condition is not easy to prove but often proposed as an assumption (cf.\cite{2023_YW_HW_Efficient_Adaptivity_AC_JSC}). One possibility to rigorously derive this condition is to utilize the technique for constructing the effective Green's functions of QC methods which is initially introduced in~\cite{2016_EV_CO_AS_Boundary_Conditions_for_Crystal_Lattice_ARMA}. We also refer the interested reader to the analysis in \cite[Section 5.1]{2021_YW_HC_ML_CO_HW_LZ_A_Post_QMMM_SISC} for further detail. In the current work, we numerically verify this condition in~\ref{sec:appendix:numerics}.

We define the {\it ideal} {\it a posteriori} error estimator to be
\[
\eta^{\rm ideal}(u_{h}) := \| \mR(I_{\rm a} u_{h}) \|_{(\UsH)^*},
\]
which is only a abstract quantity defined in a dual norm. Hence, in the following section we derive a {\it concrete} {\it a posteriori} error estimator which is an upper bound of the {\it ideal} estimator.

%


\subsection{{\it a Posteriori} error estimator based on residual forces}
\label{sec:sub:resF}

Since $\E^{\a}(\cdot)$ is an energy functional on ${\Use}$, the first variation of $\E^{\a}$ naturally defines a stress-strain relation in the language of continuum mechanics. For example, the direct calculation of the first variation of $\E^{\a}$ at $I_{\a} u_{h}$ {\it often} gives
\begin{equation}
	\label{eq:a stress}
	\mR(I_{\a} u_{h})[v] = \< \delta\E^{\a}(I_{\rm a}u_{h}), v \> =  \sum_{T \in \T_{\Lambda}} |T| \sa(I_{\rm a}u_{h}) : \nabla_T v, \quad \forall v\in\UsH,
\end{equation}
where $\T_{\Lambda}$ is the {\it canonical} triangulation induced by the reference lattice $\Lambda$ and $\sa$ can be defined as the atomistic stress tensor which we refer to \cite{2013_CO_FT_Cauchy_Born_ARMA} for the detailed formulation. We note that, by a proper application of the summation by parts or Abel transform in higher dimension, the residual $\mR(I_{\rm a} u_{h})$ can also be expressed in terms of the {\it residual forces} such that 
\begin{equation}
	\label{eq:forcenew}
	\mR(I_{\a} u_{h})[v] = \< \delta\E^{\a}(I_{\rm a}u_{h}), v \> = \sum_{T \in \T_{\Lambda}} |T| \sa(I_{\rm a}u_{h}) : \nabla_T v = \sum_{\ell \in \L} \F^{\rm a}_{\ell}(I_{\rm a}u_{h}) \cdot v(\ell), \quad \forall v\in\UsH
\end{equation}
where the {\it residual force} on each atom $\ell$ is given by $\F^{\rm a}_{\ell}(I_{\rm a}u_{h}) := \partial_{u(\ell)} \E^{\a}(u)\big|_{u = {I_{\rm a}u_{h}}}$. Adopting the residual-force based representation of the residual given in \eqref{eq:forcenew}, the {\it concrete} {\it a posteriori} error estimator $\eta(u_h)$ is given by the following lemma and theorem.

\begin{lemma}
\label{le}
Let $d=2$. Then for any $\ell \in \L$ there exists a constant $C^{\rm d2}>0$ such that
\[
|v(\ell)-v(0)| \leq C^{\rm d2}\log(2+|\ell|)\cdot \|v\|_{\UsH}, \qquad \forall v\in \UsH.
\]
Let $d=3$. Then for each $v\in\UsH$, there exist a constant $C^{\rm d3}>0$ and a $v_{\infty}\in\R^3$ such that the following estimate holds
\[
\|v - v_{\infty}\|_{\ell^6} \leq C^{\rm d3} \|v\|_{\UsH}.
\]
\end{lemma}

\begin{theorem}\label{thm:resF}
   Let the residual force $\F^{\rm a}_{\ell}(u_{h})$ be defined by~\eqref{eq:forcenew}. Under the conditions of Theorem~\ref{thm:res}, there exists a constant $C^{\rm resF}$ such that 
    \begin{eqnarray}\label{res-F-bound}
    \|u - I_{\rm a} u_{h}\|_{\Use} \leq C 
	\|\mathsf{R}(I_{\rm a} u_{h})\|_{(\Use)^{*}}
	\leq C^{\rm resF} \eta(u_{h}),
	\end{eqnarray}
	where 
	\begin{align}\label{eq:est}
	    \eta(u_h) := \left\{ \begin{array}{ll}
		\big\|\log(2+|\ell|) \cdot \F^{\rm a}_{\ell}(I_{\rm a}u_h)\big\|_{\ell^1}, \quad & d=2
		\\[1ex]
		\|\F^{\rm a}_{\ell}(I_{\rm a}u_h)\|_{\ell^{\frac{6}{5}}}, \quad & d=3 
	\end{array} \right. .  
	\end{align}
\end{theorem}

Lemma \ref{le} essentially serves as the Poincar\'{e} inequality for $\UsH$ and Theorem \ref{thm:resF} establishes an {\it concrete} upper bound of $\| \mR(I_{\rm a} u_{h}) \|_{(\UsH)^*}$ using the residual force $\F^{\rm a}_{\ell}(I_{\rm a}u_h)$ on each lattice point $\ell \in \Lambda$. The proof of Lemma~\ref{le} can be found in~\cite[Lemma A.2]{2016_HC_CO_QM_MM_P2_MMS} while the proof of Theorem~\ref{thm:resF} follows the same line as that for {\it a posteriori} error estimate in QM/MM coupling method \cite[Theorem 3.1]{2019_HC_ML_HW_YW_LZ_Adaptive_QMMM_CMAME} and thus are both omitted. 

The {\it a posteriori} analysis up to now is valid for both two and three dimensions and the theoretical framework can be easily extended to three dimensions with suitable modifications, e.g. estimates of the truncation error that will be discussed immediately. However, the adaptive QC in three dimensions requires a reliable three-dimensional mesh generator and adaptive techniques, which are out of the scope of the current work. We therefore concentrate on the two dimensional case from now on and refer the interested readers to \cite{2023_YW_LZ_Adaptive_Multigrid_AC_3D_JCP} for a recent effort in this direction. 

We note that the evaluation of the {\it concrete} {\it a posteriori} error estimator $\eta(u_{h})$ requires the computation of the {\it residual force} at each lattice point of the infinite lattice $\L$ and is thus not computable. Therefore, our next task is to develop a computable approximation of $\eta(u_{h})$.


%

\subsection{Estimates of the truncation error}
\label{sec:sub:prac_est}

In this section, we consider the estimate of the error originated from the truncation of the infinite domain. Recall that $r_{\rm cut}$ is the maximum radius of the interaction range, we can define the extension of the computational domain $\Omega$ as $\Omega_{\rm ext} := \bigcup_{\ell \in \L\cap \Omega} B_{r_{\rm cut}+1}(\ell)$. 
We then define $\rho^{\rm tr}(u_h)$ and the {\it coupling error estimator} $\eta^{\rm qc}(u_h)$ respectively by
 \begin{align}
 \label{eq:tr_est_eta}
\rho^{\rm tr}(u_h):=\sum_{\ell\in\L \cap (\Omega_{\rm ext}\setminus \Omega)} \log(2+|\ell|)\cdot \big|\F^{\rm a}_{\ell}(I_{\rm a}u_h)\big| \text{ and }
\eta^{\ac}(u_h):=\sum_{\ell\in\L\cap\Omega} \log(2+|\ell|)\cdot \big|\F^{\rm a}_{\ell}(I_{\rm a}u_h)\big|.
\end{align}
The following Lemma shows that the {\it concrete} {\it a posteriori} error estimator $\eta(u_h)$ can be bounded by the {\it coupling error estimator} $\eta^{\rm qc}(u_h)$ and the {\it truncation error estimator} $\eta^{\rm tr}(u_h)$ based on $\rho^{\rm tr}(u_h)$, and is defined in the body of the Lemma.
%
%
%

\begin{lemma}\label{le:tr}
Let $\eta(u_h)$ and $\eta^{\rm qc}(u_h)$ be defined by \eqref{eq:est} and \eqref{eq:tr_est_eta}. We have
\[
\eta(u_h) \leq \eta^{\rm qc}(u_h) + \eta^{\rm tr}(u_h),
\]
where the truncation error estimator $\eta^{\rm tr}(u_h)$ is defined by
\begin{align}\label{eq:tr_est}
    \eta^{\rm tr}(u_h) := \left\{ \begin{array}{ll}
		\rho^{\rm tr}(u_h), \quad & {\rm Point~defects}
		\\[1ex]
		\rho^{\rm tr}(u_h) + C^{\rm disloc}\big(3+\log(R_{\rm ext})\big) \cdot R_{\rm ext}^{-1}, \quad & {\rm Dislocations}
	\end{array} \right.,   
\end{align}
where $R_{\rm ext}:=R+r_{\rm cut}+1$, $\rho^{\rm tr}(u_h)$ is defined by \eqref{eq:tr_est_eta} and $C^{\rm disloc}$ is a generic constant.
\end{lemma}
\begin{proof}
Let $\L_{\rm ext}:=\L\cap\Omega_{\rm ext}$. We have by the definitions of $\eta(u_h)$, $\eta^{\ac}(u_h)$ and $\rho^{\rm tr}(u_h)$ that 
\begin{eqnarray}\label{eq:decomp}
\eta(u_h) = \eta^{\ac}(u_h) + \rho^{\rm tr}(u_h) +  \sum_{\ell\in\L\setminus\L_{\rm ext}}\log(2+|\ell|)\cdot \big|\F^{\rm a}_{\ell}({\bf 0})\big|.
\end{eqnarray}
For the case of point defects, the far-field predictor is $u_0=0$.  Hence, for $\ell \in \L \setminus \L_{\rm est}$, we have $\F^{\rm a}_{\ell}({\bf 0})=0$, which leads to
\[
\eta(u_h) = \eta^{\ac}(u_h) + \rho^{\rm tr}(u_h).
\] 
For general straight dislocations, the far-field predictor is not zero. Therefore the above equation no longer holds. However, according to the {\it sharp} decay estimates of the residual force $|\F^{\rm a}_{\ell}({\bf 0})| \leq C^{\rm disloc} \cdot |\ell|^{-3}$ (cf. \cite[Lemma 5.8]{2016_EV_CO_AS_Boundary_Conditions_for_Crystal_Lattice_ARMA}), we have
\begin{align*}
\sum_{\ell\in\L\setminus\L_{\rm ext}}\log(2+|\ell|)\cdot \big|\F^{\rm a}_{\ell}({\bf 0})\big| &\leq~ C^{\rm disloc} \int^{\infty}_{R_{\rm ext}} \log(2+r) \cdot r^{-2} \dr \\
&\leq~ C^{\rm disloc} \big(3+\log(R_{\rm ext})\big)\cdot R_{\rm ext}^{-1}.
\end{align*}
The stated result for dislocations is then obtained by applying the above estimate in \eqref{eq:decomp}.
\end{proof}

\begin{remark}[cracks]
\label{rmk:trun_crack}
The sharp decay estimate of the {\it residual force} for cracks are yet to be rigorously established. However, it is reasonable to expect that $|\F^{\rm a}_{\ell}({\bf 0})| \leq C^{\rm crack} \cdot |\ell|^{-2.5}$ by exploiting the fact that the decay of the far-field predictor for a crack is $|D u_0(\ell)|\lesssim |\ell|^{-0.5}$ \cite{2019_BM_TM_CO_Anal_Antiplane_Fracture_M3AS, 2021_JB_TH_CO_Def_Exp_arXiv} whereas $|D u_0(\ell)|\lesssim |\ell|^{-1}$ for a straight dislocation. This speculation will be numerically verified in Figure~\ref{figs:decay_crack_ff}. Hence, for anti-plane crack we consider in this work, the truncation error indicator is given by
\[
\eta^{\rm tr}(u_h) := \rho^{\rm tr}(u_h) + C^{\rm crack}\big(3+\log(R_{\rm ext})\big) \cdot R_{\rm ext}^{-0.5}.
\]
\end{remark}

The  {\it truncation error estimator} $\eta^{\rm tr}(u_h)$ serves as a measure of the error resulting from the finite computational domain. The constants $C^{\rm disloc}$ and $C^{\rm crack}$ in the definition of $\eta^{\rm tr}(u_h)$ can be empirically determined by fitting the residual forces (see Figure~\ref{figs:decay_screw_ff} and Figure~\ref{figs:decay_crack_ff} for anti-plane screw dislocation and anti-plane crack). More importantly, $\eta^{\rm tr}(u_h)$ is computable since the sum of the residual forces $\F^{\rm a}_{\ell}(I_{\rm a}u_h)$ is only evaluated in $\L \cap (\Omega_{\rm ext}\setminus \Omega)$ which is in a small circle around the the computational domain $\Omega$, but is estimated by an integral in an infinite domain. 






\subsection{Local error contribution by sampling}
\label{sec: local error contribution}

We then consider the {\it coupling error estimator} $\eta^{\rm qc}(u_h)$ defined in \eqref{eq:tr_est_eta}, which is {\it globally} defined on the computational domain $\Omega$. However, it is necessary to distribute $\eta^{\ac}(u_{h})$ into local contributions for the mesh refinement in the continuum region and the optimal allocations of the atomistic region (and the blending region for the blended type of methods). Recall that $\T_h$ is the finite element partition of the computational domain $\Omega$ defined in Section \ref{sec:sub:gen_ac}, we can rewrite $\eta^{\ac}(u_{h})$ in terms of the element $T\in\T_h$ as 
\begin{align}\label{eq:exact_resF}
    \eta^{\ac}(u_{h}) &= \sum_{\ell\in\L\cap\Omega} \log(2+|\ell|)\cdot \big|\F^{\rm a}_{\ell}(I_{\rm a}u_h)\big| \nonumber \\
    &=~ \sum_{T \in \T_{h}} \sum_{\ell \in T} \log(2+|\ell|) \cdot \big|\F^{\rm a}_{\ell}(I_{\rm a}u_{h})\big| \nonumber \\
    &=: \sum_{T \in \T_h} \eta^{\ac}_T(u_h).
\end{align}
It is important to note that, although the definition of the elementwise error estimator $\eta^{\ac}_T(u_h)$ is conceptually straightforward, the computation can be expensive. In particular, the evaluation of $\eta^{\ac}_T(u_h)$ for each $T\in\T_h$ requires a careful consideration of the geometric relationship between $\Lambda$ and $T$. The estimated computational cost for this process would be $O(N^2)$, where $N$ denotes the total number of degrees of freedom of the QC problem. See our recent work~\cite{2023_YW_HW_Efficient_Adaptivity_AC_JSC} for more detail for such geometric consideration. On the other hand, since $u_h$ is smooth in the region far from the defect core, the residual force $\F^{\rm a}_{\ell}(I_{\rm a}u_h)$ should also vary smoothly so that certain approximation for $\eta^{\ac}(u_{h})$ and $\eta^{\ac}_T(u_h)$ should be accurate enough. 



Motivated by the sampling strategy given in \cite[Section 3.3]{2019_HC_ML_HW_YW_LZ_Adaptive_QMMM_CMAME}, we propose the following {\it computable} approximation of the {\it coupling error estimator} $\eta^{\ac}(u_{h})$ by
\begin{align}\label{eq:approxresF}
    \tilde{\eta}^{\ac}(u_h) &:= \sum_{T \in \T_{h}} \omega(T) \log(2+|\tilde{{\ell}}(T)|) \cdot \big|\F^{\rm a}_{\tilde{{\ell}}(T)}(I_{\rm a}u_{h})\big| =: \sum_{T \in \T_{h}} \tilde{\eta}^{\ac}_{T}(u_{h}),
\end{align}
with the corresponding approximated {\it computable elementwise coupling error estimator} 
\begin{eqnarray}\label{eta_local}
 \tilde{\eta}^{\ac}_{T}(u_{h}) := \omega(T) \log(2+|\tilde{\ell}(T)|)\cdot\big|\F^{\rm a}_{\tilde{\ell}(T)}(I_{\rm a}u_{h})\big|,
\end{eqnarray}
where $\tilde{\ell}(T)$ is the ``repatom" of $T$ and $w(T)$ gives the weight of the element $T\in\T_{h}$. In Section~\ref{sec:numerics}, we choose $\tilde{\ell}(T)$ to be the barycenter of $T$ and $w(T)$ to be the area of $T$. By such sampling strategy, the evaluation of $\tilde{\eta}^{\ac}(u_h)$ requires a significantly lower computational cost as that of $\eta^{\ac}(u_{h})$, which is now only proportional to $N$ and is discussed in detail in Section \ref{sec:numerics}.


The following theorem shows that $\tilde{\eta}^{\ac}(u_h)$ provides an accurate enough approximation for $\eta^{\ac}(u_{h})$, given that the residual force is sufficiently smooth and the atomistic region is sufficiently large. The proof is based on an interpolation error analysis and is given in ~\ref{sec:appendix:proof}.

\begin{theorem}\label{thm:interp}
    Let $\tilde{\eta}^{\rm qc}(u_h)$ and $\eta^{\rm qc}(u_{h})$ be defined by \eqref{eq:exact_resF} and \eqref{eq:approxresF}, respectively. Denote the radius of $\Omega$ by $R_{\Omega}$ and let $\widetilde{\F}^{\rm a}$ be a $C^2$-interpolation of $\F^{\rm a}$ in $\Omega$. For sufficiently large $R_{\rm a}$, we have
    \begin{equation}
        \big|\tilde{\eta}^{\rm qc}(u_h) - \eta^{\rm qc}(u_h)\big| \lesssim \log(R_{\Omega}) \cdot \|\nabla^2 \widetilde{\F}^{\rm a}(I_{\rm a}u_{h})\|_{L^2(\Omega)}.
    \end{equation}
\end{theorem}



\subsection{Discussion}

\label{sec: discussion}

Having the {\it computable truncation error estimator} $\eta^{\rm tr}(u_h)$ and the {\it computable coupling error estimator} $\tilde{\eta}^{\ac}(u_h)$ in hand, we are at the stage of developing adaptive algorithms. However, before we proceed, a thorough discussion on the idea behind which we derive the {\it concrete} error estimator $\eta(u_{h})$ in Section \ref{sec:sub:resF} is necessary. The idea is essentially the key for our unified {\it a posteriori} estimate that is independent of the individual QC method we consider, 





We first compare our approach with the classical residual based {\it a posteriori} error estimate in the finite element method. The usual procedure in FEM is to insert the numerical solution into the variational formulation of the elliptic equation and apply an elementwise divergence theorem to obtain elementwise residual forces together with gradient jumps on the edges of the elements (cf. \cite[Chapter I, Equation (1.14)]{1996_RV_A_Post_Adapt} and \cite[Chapter III, Equation (8.16)]{2007_DB_FEM}). In contrast, \eqref{eq:forcenew} holds by a global, rather than a local or elementwise, {\it 2D/3D summation by parts}, which is an analogy of a global divergence theorem. We adopt such approach because the variational formulation of the atomistic problem \eqref{eq:a stress} is essentially discrete and a local summation by parts may not give proper elementwise residual since $\T_{\Lambda}$ and $\T_{h}$ have no hierarchical structure, i.e., $\T_{\Lambda}$ does not corresponds to a refinement of $\T_{h}$. We refer to \cite[Section 3.3]{2023_YW_HW_Efficient_Adaptivity_AC_JSC} and in particular Figure 1 therein for our recent effort on the explanation of the challenges brought by the lack of hierarchy between $\T_{h}$ and $\T_{\Lambda}$. 

We then discuss the difference between the current estimate and our previous analysis, for example in \cite{2014_CO_HW_A_Post_ACC_IMANUM, 2018_HW_ML_PL_LZ_A_Post_GRAC_2D_SISC, 2019_HW_SL_FY_A_Post_QCF_1D_NMTMA, 2023_YW_HW_Efficient_Adaptivity_AC_JSC}. The common practice of the very first step in the previous works is to introduce the corresponding variational or stress-strain formulation of the QC method analyzed. Take the energy-based QC method in \eqref{eq:general_E_ac} as an example. By adding the first variation of $\E^{\rm qc}_{h}(u_{h})$, we obtain the following representation of the residual
\begin{equation}
	\label{eq:a-ac stress}
	\mR(I_{\a} u_{h})[v] =  \sum_{T \in \T_{\Lambda}} |T| \sa(I_{\rm a}u_{h}) : \nabla_T v - \sum_{T\in\T_h}|T|\sac(u_h;T) : \nabla_T v_h,
\end{equation}
where $\sac$ is the coupling stress tensor and $v_h$ is certain interpolation of $v$ in $\Us_{h}$. Such approach, which we refer to as the {\it residual-stress based} {\it a posteriori} error estimates for QC methods,  essentially stems from the error analysis for nonconforming finite element methods and in particular the First Lemma of Strang \cite[Lemma, $\S$1, Chapter III]{2007_DB_FEM}. By properly introducing additional terms, we are able to clearly separate the error into different parts. This procedure is crucial in the {\it a priori} analysis of the QC methods to obtain the convergence \cite{2013_ML_CO_AC_Coupling_ACTANUM, 2011_CO_HW_QC_A_Priori_1D_M3AS, 2011_CO_1D_QNL_MATHCOMP} . However, such approach is also responsible for the two critical limitations described in the introduction which are the dependence of the estimator on the QC method and the complicated formulation of the estimator. In contrast, our {\it residual-force based estimate} $\eta(u_{h})$ is unified that only depends on the atomistic model.

We finally comment on the efficiency. We note that $\eta(u_{h})$ only provides an upper bound for the {\it ideal} estimator $\|\mathsf{R}(I_{\rm a} u_{h})\|_{(\Use)^{*}}$ and thus for the true approximation error $\|u - I_{\rm a} u_{h}\|_{\Use}$. It is ideal to prove that the estimator, up to some constant, also provides a lower bound, which are called the efficiency of the estimator. We refer to \cite{2018_HW_SY_Efficiency_A_Post_1D_MMS} in this direction for the residual-stress based  {\it a posteriori} error estimator and a more recent progress in \cite{2021_YW_HC_ML_CO_HW_LZ_A_Post_QMMM_SISC} on the Riesz representation of the residual force $\F^{\rm a}_{\ell}(u_h)$ in the context of QM/MM methods which, in principle, could be extended to the current work but with substantial additional effort. However, we illustrate in Section \ref{sec:numerics} that the efficiency of the estimator has little impact on our adaptive simulations.


\section{Adaptive Algorithms for Various Quasicontinuum Methods}
\label{sec:alg}

\def\Ta{\T_\a}
\def\Th{\T_{\rm h}}
\def\sh{\sigma^{\rm h}}
\def\sjump{\llbracket\s\rrbracket}

\newcommand{\fig}[1]{Figure \ref{#1}}
\newcommand{\tab}[1]{Table \ref{#1}}



In this section, we develop the adaptive algorithms for the {\it a posteriori} error control problems for various QC coupling schemes based on our {\it unified residual-force based a posteriori} error estimator.  We concentrate on two major classes of QC methods most often encountered in the literature of QC, which are the sharp interfaced QC methods and the blended QC methods \cite{2020_EG_Roadmap_Multiscale_Modeling_MSMSE, 2013_ML_CO_AC_Coupling_ACTANUM, 2009_Miller_Tadmor_Unified_Framework_Benchmark_MSMSE}. Though the QC methods can be further categorized as energy based and force based from the modeling perspective, the classification for the methods into sharp interfaced type and blended type is suitable in the context of the {\it a posteriori} error control as the methods in each class share the same adaptive algorithm which we will develop subsequently. For each class of methods, we will first present the detailed formulations of typical coupling schemes and then develop the corresponding adaptive algorithms. Possible applications of our {\it unified residual-force based a posteriori} error estimate to other classes of multiscale coupling methods are also briefly commented at the end of this section.



\subsection{The adaptive algorithm for the sharp interfaced QC methods}
\label{sec:sub:adapresF}
We first consider the consistent QC methods with sharp interface. The most important feature of this class of methods is to have a very narrow transition from the atomistic region to the continuum region (often a few layers of atoms depending on the interaction range). This is achieved either by introducing an modified interface potential (e.g. GRAC method \cite{2012_CO_LZ_GRAC_Construction_SIAMNUM, 2014_CO_LZ_GRAC_Coeff_Optim_CMAME} and GRC method \cite{2006_WE_JL_ZY_GRC_PRB, 2017_HL_PM_GSC_Finite_Range_CiCP}) or using a sharp characteristic function to couple force balance equations (e.g. QCF method \cite{2008_MD_ML_Ana_Force_Based_QC_M2NA,2014_JL_PM_Convergence_QCF_2D_Planar_Inter_No_Defects_SIAMNUM}). We also note that most of the existing literatures on the adaptive QC focus on this class of methods \cite{2014_CO_HW_A_Post_ACC_IMANUM, 2018_HW_ML_PL_LZ_A_Post_GRAC_2D_SISC, 2019_HW_SL_FY_A_Post_QCF_1D_NMTMA} and thus it provide us a good benchmark problem to illustrate the advantage of our unified {\it a posteriori} error estimates.



%


%


We now present the formulations of two consistent QC methods with sharp interface.


The first method is the geometric reconstruction based consistent QC (GRAC) method which is an energy based method initially proposed in \cite{2006_WE_JL_ZY_GRC_PRB, 2012_CO_LZ_GRAC_Construction_SIAMNUM} and further developed in \cite{2014_CO_LZ_GRAC_Coeff_Optim_CMAME}. The GRAC method match exactly the formulation as shown in \eqref{eq:general_E_ac}, that is,
\begin{align}
  \label{eq:grac_energy}
  \E^{\rm grac}_{h}(u_{h}) :=& \sum_{\ell \in \L_{\a}} V'_\ell(Du_{h}) + \sum_{\ell \in \L_{\rm i}} V^{\rm i}_\ell(Du_{h}) + \sum_{T \in \T_{c}} \omega_T \big(W'(\nabla u_{h}|_T+\nabla u_0|_T) - W'(\nabla u_{0}|_T)\big), 
\end{align}
where $V_\ell^\i$ is a modified interface site potential and $\omega_T$ is the effective volumes of elements (see \cite[Section 2.2]{2012_CO_LZ_GRAC_Construction_SIAMNUM} for a detailed discussion). The modified interface site potential $V_\ell^\i$ is determined by the geometric reconstruction parameters $C_{\ell;\rho,\vsig}$ so that for each $\ell \in\L^\i, \rho, \vsig \in \Rg_\ell$,  $V_\ell^\i(Du_h) := V \Big( \big( {\textstyle \sum_{\vsig \in
      \Rg_\ell} C_{\ell;\rho,\vsig} D_\vsig u_h(\ell) } \big)_{\rho \in
    \Rg_\ell} \Big)$ and the energy and force patch test consistency conditions are satisfied as follows \cite{2012_CO_Patch_Test_2D_M2NA}:
\begin{equation}
  \label{eq:energy_pt}
  V_\ell^\i({\bf 0}) = V({\bf 0}), \quad \forall\ell \in \L_\i \quad \text{and} \quad 
  \< \del \E^{\rm grac}_h({\bf 0}), v \> = 0 \quad \forall v \in \Us_h.
\end{equation}
The QC problem of the GRAC method is to solve the energy minimization problem
\begin{equation}
  \label{eq:min_ac}
  u^{\rm grac}_h \in \arg\min \big\{ \E^{\rm grac}_{h}(u_{h}),~u_h\in
  \Us_{h} \big\}.
\end{equation}



%


The second method is the QCF method which is a force based QC method with sharp interface. We first define the pure Cauchy-Born finite element functional \cite{2016_XL_CO_AS_BK_BQC_Anal_2D_NUMMATH} for $u_h \in \Us_h$ by, 
\begin{eqnarray}
\label{eq:cb_energy}
\E^{\rm cb}_h(u_h) := \int_{\Omega_h} Q_h \big[W'(\nabla u_h + \D u_0) - W'(\nabla u_0) \big] \dx,
\end{eqnarray}
where the piecewise constant mid-point interpolant $Q_{h}v \in \mathcal{P}_0(\T_h)$ for a function $v: \Omega_h \rightarrow \R$ is defined such that $Q_h v(x):=v(x_T)$ for $x\in T\in \T_h$, where $x_T:=\frac{1}{|T|}\int_T x\dx$. The variational formulation of the QCF method is then given by
\begin{eqnarray}\label{eq:qcf}
\<\mathcal{F}_h^{\rm qcf}(u_h), v_h\>:= \<\delta \E^{\rm a}(u_h), \chi_{\L_{\a}}v_h\> + \<\delta\E^{\rm cb}_h(u_h), (1-\chi_{\L_{\a}}) v_h\>, 
\end{eqnarray}
where $\chi_{\L_{\a}}$ is the characteristic function of the atomistic region $\L_{\a}$ and $\E^{\rm a}$ is the atomistic energy defined in \eqref{energy-difference}. The QC problem of the QCF method is to solve the variational nonlinear system
\begin{eqnarray}\label{eq:variational-problem-QCF}
\<\F_h^{\rm qcf}(u^{\rm qcf}_{h}), v_h\> = 0, \quad \forall v_h \in \Us_h.
\end{eqnarray}


The consistency of the QC methods with sharp interface, in particular the GRAC and the QCF methods just presented have been studied in \cite{2013_ML_CO_AC_Coupling_ACTANUM, 2012_CO_LZ_GRAC_Construction_SIAMNUM}. We summarize the consistency of the two methods (together with blended type of methods for later discussion) in Table \ref{tab: decay rate}. We note that the rigorous {\it a priori} error estimates of the QC methods with sharp interface for dislocations and cracks are still lacking. The rates shown below are in fact speculated based on the methodologies employed in \cite{2016_EV_CO_AS_Boundary_Conditions_for_Crystal_Lattice_ARMA}, which are beyond the scope of this work but deserve a future study.
 \begin{table}[htbp]
 \centering
  \begin{tabular}{|c|c|c|c|c|c|}
 \hline
 \diagbox{Defect types}{$k$}{QC coupling methods} & GRAC & QCF & BQCE & BQCF & BGFC \\ \hline
 Point defects ($s=0$)  & 1.0 & 1.0 & 0.5 & 1.0 & 1.0 \\ 
 Anti-plane dislocation ($s=0.5$)  & 1.0 & 1.0 & 1.0 & 1.0 & 1.0 \\ 
 Anti-plane crack ($s=0.5$) & 0.25 & 0.25 & 0.25 & 0.25 & 0.25 \\ 
 \hline
 \end{tabular}
 \caption{Convergence rates of various QC coupling methods for three typical defects considered in this work.}
 \label{tab: decay rate}
 \end{table}



The adaptive algorithm for the QC methods with sharp interface is developed in Algorithm~\ref{alg:size}. The algorithm essentially follows the one previously developed, for example, in~\cite{2018_HW_ML_PL_LZ_A_Post_GRAC_2D_SISC, 2023_YW_HW_Efficient_Adaptivity_AC_JSC} with only the replacement of the error estimators by $\tilde{\eta}^{\ac}$ and $\tilde{\eta}^{\ac}_{T}$ defined in \eqref{eq:approxresF} and \eqref{eta_local}. However, we note that $\tilde{\eta}^{\ac}$ and $\tilde{\eta}^{\ac}_{T}$ are independent of the coupling scheme and the adaptive algorithm is now truly {\it uniform} for any QC methods with sharp interface. We thus prevent the the complexity of deriving and implementing the (often involved) {\it a posteriori} error estimator for the individual method as what have been done previously in \cite{2018_HW_ML_PL_LZ_A_Post_GRAC_2D_SISC, 2020_ML_PL_LZ_Finite_Range_A_Post_2D_CICP, 2019_HW_SL_FY_A_Post_QCF_1D_NMTMA}. Further discussion and comparison of the algorithm, including the convergence and the efficiency of the implementation are given in Section~\ref{sec:numerics}.

\begin{algorithm}[!ht]
\caption{Adaptive algorithm for sharp interfaced QC methods}
\label{alg:size}
\begin{enumerate}
	\item[Step 0] \textit{Prescribe:} Set $\Omega$, $\T_h$, $N_{\max}$, $\rho_{\rm tol}$, $\tau_1$, $\tau_2$, $K$ and $R_{\max}$.
	
	\item[Step 1] \textit{Solve:} Solve the GRAC solution $u_{h}$ of \eqref{eq:min_ac} on the current mesh $\T_{h}$.  
	\item[Step 2] \textit{Estimate:} Compute the local error estimator $\tilde{\eta}^{\ac}_{T}$ by \eqref{eta_local} for each $T\in\T_h$, and the truncation indicator $\rho^{\rm tr}$ by \eqref{eq:tr_est}. Compute the degrees of freedom $N$ and $\tilde{\eta}^{\ac}:=\sum_{T} \tilde{\eta}^{\ac}_{T}$. If $\rho^{\rm tr} > \tau_1 \tilde{\eta}^{\ac}$, enlarge the computational domain. Stop if $N>N_{\max}$ or $\tilde{\eta}^{\ac} < \eta_{\rm tol}$. 
	
	\item[Step 3] \textit{Mark:} 
	\begin{enumerate}
	\item[Step 3.1]:  Choose a minimal subset $\M\subset \T_h$ such that
	\begin{displaymath}
		\sum_{T\in\M}\tilde{\eta}^{\ac}_{T}\geq\frac{1}{2} \tilde{\eta}^{\ac}.
	\end{displaymath}	 
	\item[Step 3.2]: Find the interface elements which are within $k$ layers of atomistic distance, $\M^k_\i:=\{T\in\M\bigcap \T_h^\c: {\rm{dist}}(T, \Li)\leq k \}$. Find the first $k\leq K$ such that 
	\begin{equation}
		\sum_{T\in \M^k_\i}\tilde{\eta}^{\ac}_{T}\geq \tau_2\sum_{T\in\M}\tilde{\eta}^{\ac}_{T}.
		\label{eq:interface2}
	\end{equation}
	Let $\M = \M\setminus \M^{k}_\i$. Expand the interface $\L^\i$ outward by $k$ layers. 
	\end{enumerate}
	
	\item[Step 4] \textit{Refine:} Bisect all elements $T\in \M$. Go to Step 1.
\end{enumerate}
\end{algorithm}

\subsection{The adaptive algorithm for the blended QC methods}
\label{sec:sub:alg}
We then consider the blended QC methods and the corresponding adaptive algorithm. The primary advantage of the blended methods is the simplicity of modeling and implementation, which make these methods well-received for the simulations of complex defects in real-world applications. 

All the blended QC methods depend on a so-called blending function $\beta \in C^{2,1}(\R^d)$ satisfying $\beta=0$ in $\Omega_{\a}$, $\beta=1$ in $\Omega_{\c}$ and ${\rm supp}(\nabla \beta) \subset \Omega_{\b}$ which characterizes the transition of the atomistic model in the defect core to the continuum model in the far field. The choice of the blending function $\beta$ significantly influceds the accuracy of the corresponding blended QC coupling methods and is detailed analyzed in \cite{2013_ML_CO_AC_Coupling_ACTANUM, 2013_ML_CO_BK_BQCE_CMAME, 2016_XL_CO_AS_BK_BQC_Anal_2D_NUMMATH}. In the current work we employ the same blending function as that in \cite{2016_CO_LZ_GForce_Removal_SISC, 2013_ML_CO_BK_BQCE_CMAME}, which is obtained in a preprocessing step by approximately minimizing $\|\nabla^2 \beta\|_{L^2}$. 

We introduce three typical blended QC methods. The first one is the BQCE method which is an energy-based method \cite{2013_ML_CO_BK_BQCE_CMAME, 2016_XL_CO_AS_BK_BQC_Anal_2D_NUMMATH}. For $u_{h} \in \Us_{h}$, the BQCE energy functional is obtained by combining the atomistic and continuum energy functionals via a blending function $\beta$, that is,
\begin{align}\label{eq:Ebqce}
  \E_h^{\rm bqce}(u_{h}) := \sum_{\ell \in \L \cap \Omega_{\rm a}}V'_{\ell}(D u_h) &+ \sum_{\ell \in \L \cap \Omega_{\rm b}}\big(1-\beta(\ell)\big)V'_{\ell}(D u_h) + \sum_{T \in \T_{\rm b}} \beta_T \big(W'(\nabla u_{h}|_T+\nabla u_0|_T) - W'(\nabla u_{0}|_T)\big) \nonumber \\
  &+ \int_{\Omega_{\c}} Q_{h} \big(  W'(\D u_h + \D u_0) - W'(\D u_0) \big) \dx
\end{align}
where $V'_{\ell}$ and $W'$ are defined by \eqref{energy-difference} and \eqref{eq:W} respectively, $Q_{h}$ is the same piecewise constant mid-point interpolation operator as that in \eqref{eq:cb_energy}, $\beta_{T}:=\beta(x_{T})$ with $x_T$ being the barycenter of $T$ and $\T_{\rm b}$ is the {\it canonical} triangulation induced by $\L \cap \Omega_{\rm b}$. The BQCE problem is then to find
\begin{equation}\label{eq:variational-problem-BQCE}
u^{\rm bqce}_{h} \in \arg\min \big\{ \E^{\rm bqce}_{h}(u_{h}),~u_{h} \in \Us_{h} \big\}.
\end{equation}


While the BQCE method \eqref{eq:Ebqce} blends the atomistic and the continuum energies, the BQCF method \cite{2014_HL_ML_CO_AS_BQCF_Compt_CMAME, 2016_XL_CO_AS_BK_BQC_Anal_2D_NUMMATH} blends the atomistic and the continuum force balance equations and is thus a force-based method. Recall the definition of the blending function $\beta \in C^{2,1}(\R^2)$, the atomistic energy $\E^{\rm a}$ and the pure Cauchy-Born energy $\E^{\rm cb}_h$ given in \eqref{energy-difference} and \eqref{eq:cb_energy} respectively.  The BQCF operator is  the nonlinear map $\F^{\rm bqcf}_h:\Us_h \rightarrow (\Us_h)^{*}$, defined by
\begin{eqnarray}
\<\F^{\rm bqcf}_h(u_h), v_h\>:= \<\delta \E^{\rm a}(u_h), (1-\beta)v_h\> + \<\delta\E^{\rm cb}_h(u_h), \beta v_h\>. 
\end{eqnarray}
In the BQCF method we solve the following variational nonlinear system
\begin{eqnarray}\label{eq:variational-problem-BQCF}
\<\F_h^{\rm bqcf}(u^{\rm bqcf}_{h}), v_h\> = 0, \quad \forall v_h \in \Us_h.
\end{eqnarray}



The third method is the BGFC method which combines the benefits of blending \cite{2016_XL_CO_AS_BK_BQC_Anal_2D_NUMMATH} and ghost force correction \cite{1999_VS_RM_ETadmor_MOrtiz_AFEM_QC_JMPS}.
The BGFC method is first proposed and analyzed in \cite{2016_CO_LZ_GForce_Removal_SISC} and here we introduce an alternative but simplified formulation. Let $\E^{\rm bqce}_{\rm hom}(\hat{u}_0)$ be the BQCE energy at a suitable ``predictor" $\hat{u}_0$. Then the energy funtional of BGFC method can be written as 
\begin{eqnarray}\label{eq:gfc}
\E^{\rm bgfc}(u_h) = \E^{\rm bqce}(u_h) - \big\<\delta \E^{\rm bqce}_{\rm hom}(\hat{u}_0), u_h\big\>.
\end{eqnarray}
Motivated by \cite[Section 2.7 and Section 4.2]{2016_CO_LZ_GForce_Removal_SISC}, we use \eqref{eq:gfc} with $\hat{u}_0=0$ for all types of crystalline defects in our implementation.

The BGFC problem we would like to solve is to find
\begin{equation}\label{eq:variational-problem-BGFC}
u^{\rm bgfc}_{h} \in \arg\min \big\{ \E^{\rm bgfc}_{h}(u_{h}), ~u_{h} \in \Us_{h} \big\}.
\end{equation}


The consistency conditions of the three methods just introduced are listed in Table \ref{tab: decay rate}. The consistency conditions for point defects and for anti-plane dislocation essentially follow  from \cite{2013_ML_CO_AC_Coupling_ACTANUM} and \cite{2016_CO_LZ_GForce_Removal_SISC, 2016_XL_CO_AS_BK_BQC_Anal_2D_NUMMATH} but are adapted to the setting of the current work, while that for anti-plane crack is in fact a new speculated result. It is worth mentioning that to obtain the consistency, we exploit the {\it a priori} knowledge that $R_{\a}=R_{\b}$, i.e., the radius of the atomistic region is the same as the width of the blending region. This raises the additional task for the adaptive algorithm of the blended QC methods compared with Algorithm \ref{alg:size}, which is the proper allocation of the atomistic and the blending regions on the fly. 

Indeed, the only existing work for blended type of methods \cite{2019_HW_SL_FY_A_Post_QCF_1D_NMTMA}, which considers the {\it a posteriori} error control for a BQCF method in 1D, avoids this additional task by fixing $R_\a$ and $R_{\b}$. It is neither clear how the residual-stress based error estimate there (c.f. \cite[Theorem 3.2]{2019_HW_SL_FY_A_Post_QCF_1D_NMTMA}), may help to circumvent such difficulty. By contrast, the adaptive algorithm for the blended QC methods which we develop in Algorithm \ref{alg:resF} fulfill this additional task. In particular, the {\it marking step} (Step 3) utilizes a modified D\"ofler strategy \cite{1996_WD_Adaptive_Poisson_SIAMNUM} to allocate the two regions based on a ratio $\alpha$ computed according to the distribution of the \textit{a posteriori} error estimator of the atomistic and the blending regions. More precisely, we construct a set $\M^k\subset\M$ to contain the marked elements that are within $k$ layers to $\L_{\rm a}$. This $k$ layers are the total layers of the atomistic and the blending regions to be expanded. The set $\M\setminus\M^k$ then contains the elements to be refined in the continuum region. For $T\in \M^k$, if ${\rm dist}(T, \L_{\a}) \leq {\rm dist}(T, \L_{\c})$, we assign it to constitute a new subset $\M^k_{\a}$. We compute the ratio of the summation of the error estimator for $T\in\M^k_{\a}$ and that for $T\in\M^k$ by \eqref{eq:ratio}. According to this ratio $\alpha$, we extend the atomistic and the blending regions outwards by $[\alpha k]$ and $k-[\alpha k]$ layers respectively. In Section~\ref{sec:numerics} we will observe almost the same relationship on the fly during the adaptive simulations for all types of crystalline defects considered in this work.

\begin{algorithm}
\caption{Adaptive algorithm for blended QC methods}
\label{alg:resF}
\begin{enumerate}
	\item[Step 0] Prescible $\Omega$, $\T_h$, $N_{\max}$, $\eta_{\rm tol}$, $K$, $\tau_1$ and $\tau_2$.
	\item[Step 1] \textit{Solve}: Solve the BQCE (BQCF or BGFC) solution $u_{h}$ of \eqref{eq:variational-problem-BQCE} (\eqref{eq:variational-problem-BQCF} or \eqref{eq:variational-problem-BGFC}) on the current mesh $\T_{h}$.  
	\item[Step 2] \textit{Estimate}: Compute the local error estimator $\tilde{\eta}^{\ac}_{T}$ by \eqref{eta_local} for each $T\in\T_h$, and the truncation indicator $\rho^{\rm tr}$ by \eqref{eq:tr_est}. Compute the degrees of freedom $N$ and $\tilde{\eta}^{\ac}:=\sum_{T} \tilde{\eta}^{\ac}_{T}$. If $\rho^{\rm tr} > \tau_1 \tilde{\eta}^{\ac}$, enlarge the computational domain. Stop if $N>N_{\max}$ or $\tilde{\eta}^{\ac} < \eta_{\rm tol}$. 
	\item[Step 3] \textit{Mark:} Construct the refinement set $\M$, the total number of atomistic layers to be expanded for both atomistic and blending regions $k$ and the ratio $\alpha$ representing the error contribution of atomistic region.
\begin{enumerate}
	\item[Step 3.1]:  Choose a minimal subset $\M\subset \T_{h}$ such that
	\begin{displaymath}
		\sum_{T\in\M}\tilde{\eta}^{\ac}_{T}\geq\frac{1}{2} \tilde{\eta}^{\ac}.
	\end{displaymath}	 
	\item[Step 3.2]: Find the interface elements which are within $k$ layers of atomistic distance, $\M^k:=\{T\in\M\bigcap (\T_{\b} \cup \T_{\c}): {\rm{dist}}(T, \Lambda_{\a})\leq k \}$, where ${\rm dist}(T, \L_{\a}) := \inf\{|\ell-x_T|,\forall\ell\in\L_{\a}\}$ is the distance between $T$ and $\L_{\a}$ and $x_T$ is barycenter of $T$. Find the first $k\leq K$ such that 
	\begin{equation}
		\sum_{T\in \M^k}\tilde{\eta}^{\ac}_{T}\geq \tau_2\sum_{T\in\M}\tilde{\eta}^{\ac}_{T}.
	\end{equation}
	\item[Step 3.3]: Construct the set $\M^k_{\a}:=\{T\in\M^k: {\rm dist}(T, \L_{\a}) \leq {\rm dist}(T, \L_{\c})\}$ (the definition of ${\rm dist}(T, \L_{\c})$ is similar to that of ${\rm dist}(T, \L_{\a})$). Compute the ratio 
	\begin{eqnarray}\label{eq:ratio}
	\alpha:=\frac{\sum_{T\in\M^k_{\a}} \tilde{\eta}^{\ac}_T}{\sum_{T\in \M^k}\tilde{\eta}^{\ac}_{T}}.
	\end{eqnarray}
	Let $\M:=\M\setminus\M^k$.
\end{enumerate}
	
	\item[Step 4] \textit{Refine:} Expand the atomistic region outward by $[\alpha k]$ layers and the blending region outward by $k-[\alpha k]$ layers. Bisect all elements $T\in \M$. Go to Step 1.
\end{enumerate}
\end{algorithm}

\newcommand{\sty}[1]{\boldsymbol{#1}}
\newcommand{\bmu}{\sty{u}}
\newcommand{\clE}{\mathcal{E}}
\newcommand{\clL}{\mathcal{L}}
\newcommand{\clR}{\mathcal{R}}
\newcommand{\bmf}{\sty{f}}
\newcommand{\bmK}{\sty{K}}
\newcommand{\bmxi}{\sty{\xi}}
\newcommand{\bmzeta}{\sty{\zeta}}
\newcommand{\bmeta}{\sty{\eta}}

\subsection{Discussion and possible extension to other QC coupling methods}

Apart from the blended and the sharp interfaced coupling methods, there are other QC coupling methods employed in both engineering and mathematical communities. To comment on the versatility of the error estimator proposed in this work, we present two examples as follows. 

The first one is the so-called flexible boundary condition (FBC) method~\cite{2021_MH_Anal_FBC_AC_MMS, 2021_AG_WC_Anal_FBC_MSMSE}, which applies the continuum solutions as the boundary conditions of the atomistic problem in the core region. The FBC method can be formulated as: find $u^{\rm fbc}_h := \{ u^{\rm a}, u^{\rm c} \}$ such that
\begin{align}\label{eq:coupled_problem}
    ({\bf P}^{\rm a}) \; \left\{
    \begin{aligned}
        \; \clL[u^{\rm a}] &= 0 &\qquad& \text{in} \; \L_{\rm a}, \\
        \; u                                &= u^{\rm c}   &      & \text{in} \; \L_{\rm i},
    \end{aligned}
    \right.
    &&
    ({\bf P}^{\rm c}) \; \left\{
    \begin{aligned}
        \; \clL_{\rm cb}[u^{\rm c}] &= 0 &\qquad& \text{in} \; \L_{\rm c}, \\
        \; u                                        &= u^{\rm a}    &      & \text{on} \; \L_{\rm i}.
    \end{aligned}
    \right.,
\end{align}
where $\clL$ and $\clL_{\rm cb}$ are the force operator of the atomistic and Cauchy-Born models, respectively.

The second one is the optimization-based QC methods, which requires the atomistic and the continuum subdomains with an overlap region $\Omega_{\i} := \Omega_{\a} \cap \Omega_{\c}$. This is an alternative ghost-force-free method. The optimization-based QC method is to solve the constrained minimization problem:
find $u^{\rm opt}_h := \{ u^{\rm a}, u^{\rm c} \}$ such that $\|\nabla u^{\a} - \nabla I_{\a} u^{\c}\|_{L^2(\Omega_{\rm i})}$ is minimized
subject to 
\begin{align}\label{eq:opt-ac}
    \left\{ \begin{array}{ll}
		\<\delta\E^{\a}(u^{\a}), v\>=0 \quad & \forall v \in \Us^{\a}_0
		\\[1ex]
		\<\delta\E^{\c}(u^{\c}), v\>=0 \quad & \forall v \in \Us^{\c}_0
	\end{array} \right. \quad \textrm{and} \quad  \int_{\Omega_{\i}} u^{\a} - I_{\a} u^{\c}\dx = 0. 
\end{align}
The objective ensures that the mismatch between $\bar{u}^{\a}$ and $\bar{u}^{\c}$ over $\Omega_{\i}$ is as small as possible. 




Both of these aforementioned methods are consistent under proper constructions \cite{2016_DO_AS_ML_Optimization_Based_AC_M2NA, 2021_MH_Anal_FBC_AC_MMS} and are of substantial significance, yet a noticeable gap exists in the availability of the corresponding adaptive algorithms. In light of the framework shown in Section~\ref{sec:err}, the proposed residual-force based error estimator can in principle be applied to both methods. This primarily attributes to the fact that it is independent from the detailed formulations of the QC methods. Notably, in the FBC method, the residual-force based error estimator can be interpreted as a generalized speed for an interface motion problem. This perspective enables its resolution through the utilization of the well-established fast marching method~\cite{1999_JS_Fast_Marching_SIAMREVIEW}. Conversely, in the case of optimization-based adaptive/corrective methods, the integration of the adaptive algorithm necessitates its incorporation within the interface minimization problem. A thorough exploration of this promising direction will be a focal point of our future work.


\section{Numerical Experiments}
\label{sec:numerics}

In this section, we use the adaptive algorithms just developed in Section \ref{sec:alg} to perform adaptive simulations for three prototypical types of defects: micro-crack, anti-plane screw dislocation, and anti-plane crack. The geometries of the defects are already illustrated in Figure~\ref{fig:geom_defects}. 



We focus on the two dimensional triangular lattice $\Lhom = \mathsf{A}\Z^2$ where
\begin{displaymath}
{\sf A}=\left(
	\begin{array}{cc}
		1 & \cos(\pi/3) \\
		0 & \sin(\pi/3)
	\end{array}	 \right),
\end{displaymath}
which can be considered as the projection of a BCC crystal along the (111) direction upon rotating and possibly dilating \cite{2016_EV_CO_AS_Boundary_Conditions_for_Crystal_Lattice_ARMA}. All through our simulation, we adopt the well-known EAM potential as the site potential $V_{\ell}$ that is defined by
\begin{align}
  V_{\ell}(y) := & \sum_{\ell' \in \Nhd_{\ell}} \phi(|y(\ell)-y(\ell')|) + F\B(
  {\textstyle \sum_{\ell' \in \Nhd_{\ell}} \psi(|y(\ell)-y(\ell')|)} \B),\nonumber\\
    = &\sum_{\rho \in \Rg_{\ell}} \phi\big(|D_\rho y(\ell)|\big) + F\B(
  {\textstyle \sum_{\rho \in \Rg_{\ell}}} \psi\big( |D_\rho y(\ell)|\big) \B),  
    \label{eq:eam_potential}
\end{align}
for a pair potential $\phi$, an electron density function $\psi$ and
an embedding function $F$. We choose
\begin{displaymath}
\phi(r)=\exp(-2a(r-1))-2\exp(-a(r-1)),\quad \psi(r)=\exp(-br),
\end{displaymath}
\begin{displaymath}
F(\tilde{\rho})=C\left[(\tilde{\rho}-\tilde{\rho_{0}})^{2}+
(\tilde{\rho}-\tilde{\rho_{0}})^{4}\right],
\end{displaymath}
with the parameters $a=4, b=3, c=10$ and $\tilde{\rho_{0}}=6\exp(-0.9b)$ specifically, which are the same as the numerical experiments presented in \cite{2018_HW_ML_PL_LZ_A_Post_GRAC_2D_SISC, 2020_ML_PL_LZ_Finite_Range_A_Post_2D_CICP, 2023_YW_HW_Efficient_Adaptivity_AC_JSC}.  The radius of the computational domain $\Omega$ is set to be 300 initially and the adaptive processes start with the initial configuration with $R_{\rm a}=3$. The adaptive parameters in Algorithm~\ref{alg:size} and Algorithm~\ref{alg:resF} are fixed to be $\tau_1=1.0$ and $\tau_2=0.7$.

\subsection{Micro-crack}
\label{sec:sub:mcrack}
The first defect we consider is the micro-crack, which is a prototypical example of point defects which serves as an example of a localized defect with an anisotropic shape. To generate this defect, we remove $k$ atoms from $\Lhom$,
\begin{align*}
\L_{k}^{\rm def}:=\{-(k/2)e_{1}, \ldots, (k/2-1)e_{1})\},     & \qquad{\rm if }\quad k \quad\textrm{is even},\\
\L_{k}^{\rm def}:=\{-(k-1)/2e_{1}, \ldots, (k-1)/2e_{1})\}, & \qquad{\rm if }\quad k \quad\textrm{is odd},
\end{align*}
and $\L = \Lhom\setminus \L_{k}^{\rm def}$. 
We set $k=11$ and apply an isotropic stretch $\mathrm{S}$ and a shear $\gamma_{II}$  by setting
\begin{displaymath}
\mF =\left(
	\begin{array}{cc}
		1 & \gamma_{II} \\
		0            & 1+\mathrm{S}
	\end{array}	 \right)
	\cdot{\sf {F_{0}}}
\end{displaymath}
where ${\sf F_{0}} \propto \mathrm{I}$ is a macroscopic stretch or compression and $\mathrm{S}=\gamma_{II}=0.03$. 

\subsubsection{Adaptive simulations for a micro-crack by sharp interfaced QC methods}
\label{sec:sub:mcrack sharp interface}



We first present the adaptive simulations of a micro-crack by adaptive QC methods with sharp interface. In particular, we consider the a specific GRAC method which is adopted in the rest of this paper, whose geometric reconstruction parameters $C_{\ell;\rho, \zeta}$ are determined by  \cite[Proposition 3.7]{2012_CO_LZ_GRAC_Construction_SIAMNUM}: (a) $C_{\ell, \rho, \vsig} = 0$ for $|(\rho - \vsig)~\textrm{mod}~6| > 1$; (b) $C_{\ell, \rho, \rho-1} = C_{\ell, \rho, \rho+1} = 1 - C_{\ell, \rho, \rho}$; (c) $C_{\ell, \rho, \rho} = C_{\ell+\rho, -\rho, -\rho} = 1$ for $\ell+\rho\in\L^{\i}$; (d) $C_{\ell, \rho, \rho}=1$ for $\ell+\rho \in \L^{\a}$; (e) $C_{\ell, \rho, \rho}=2/3$ for $\ell+\rho\in \L^{\c}$. 

We test the adaptive GRAC method based on three different {\it a posteriori} error estimates which are the {\it original residual-stress based error estimate} derived in \cite{2018_HW_ML_PL_LZ_A_Post_GRAC_2D_SISC}, the {\it modified residual-stress based error estimate} in \cite{2023_YW_HW_Efficient_Adaptivity_AC_JSC} and the {\it residual-force based error estimate} in the current work. The corresponding adaptive algorithms are \cite[Algorithm 3]{2018_HW_ML_PL_LZ_A_Post_GRAC_2D_SISC}, \cite[Algorithm 2]{2023_YW_HW_Efficient_Adaptivity_AC_JSC} and Algorithm \ref{alg:size} in Section \ref{sec:alg}. We restrict ourselves to a system with nearest neighbor interactions only so that a direct comparison among the three error estimates just listed can be given. The increased complexity and potential instability associated with generalizing the method to finite interaction range are discussed after we present the numerical demonstrations. The adaptive simulations by the QCF method are expected to exhibit qualitative similarity to those of the GRAC method and is thus omitted for the sake of simplicity.
%
%
%


\begin{figure}[htb]
\centering
	\subfloat[Convergence]{
		\label{fig:conv_mcrack_comp_grac}
		\includegraphics[height=4.5cm]{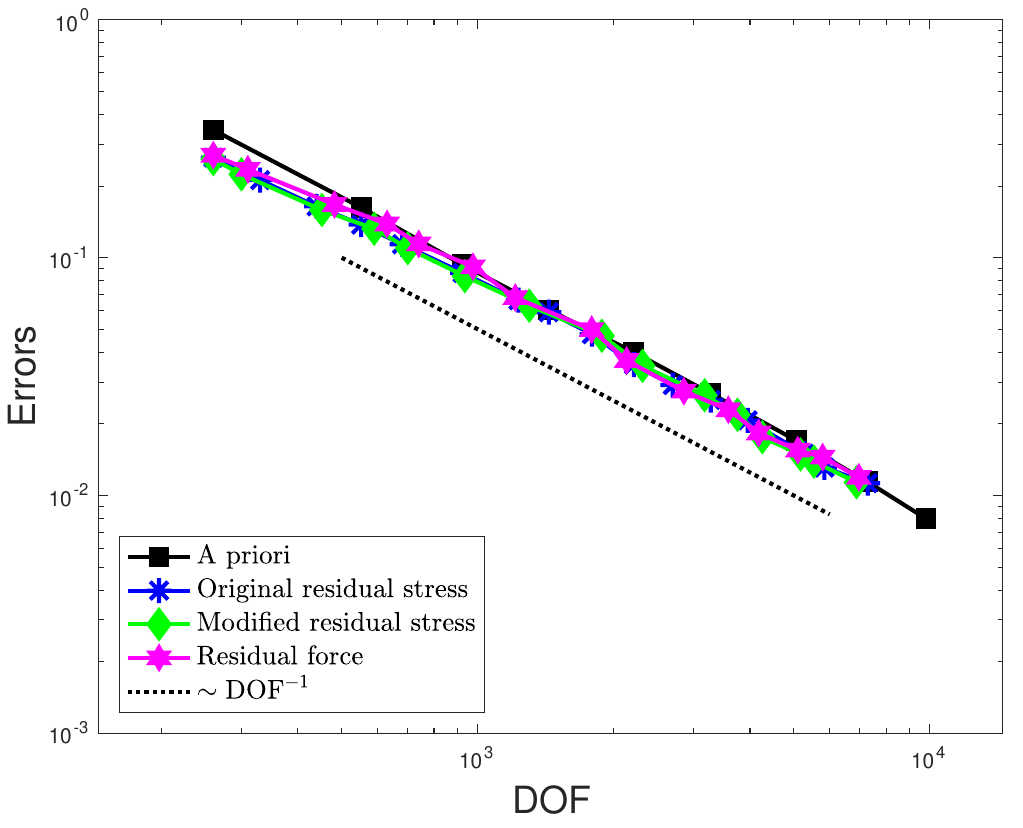}}
	\subfloat[Efficiency factors]{
		\label{fig:conv_mcrack_efffac_comp_grac}
		\includegraphics[height=4.5cm]{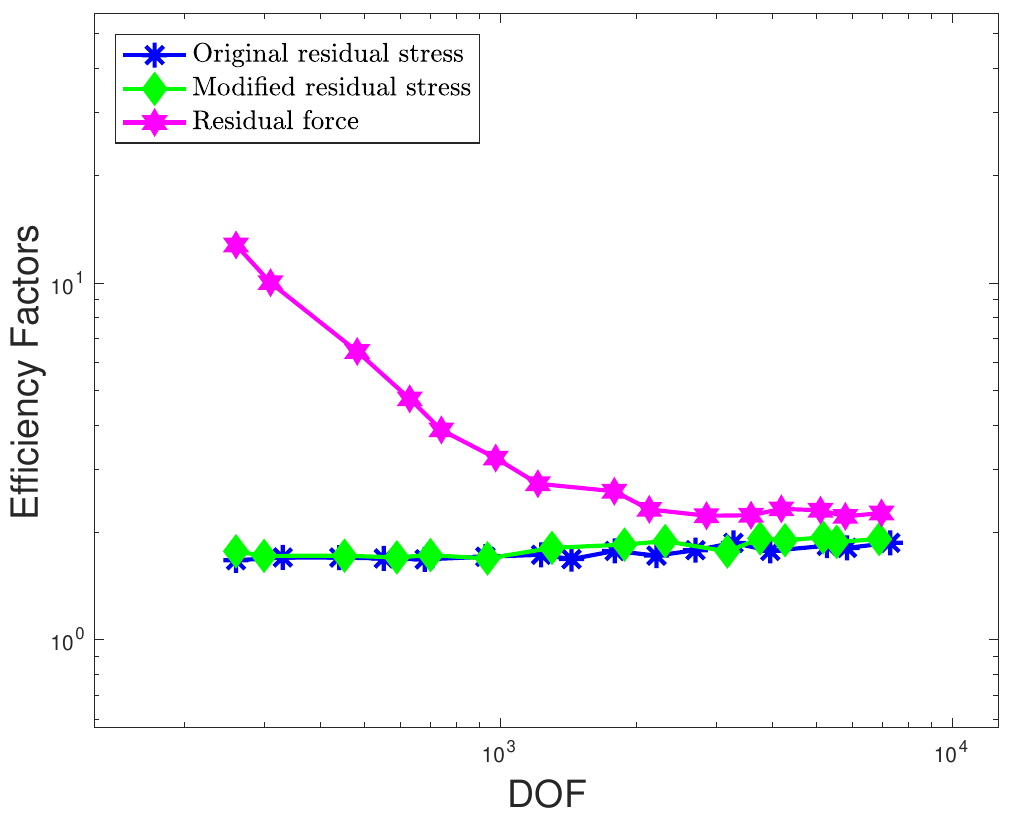}}			
	\subfloat[Estimation time]{
		\label{fig:time_mcrack_efffac_comp_grac}
		\includegraphics[height=4.5cm]{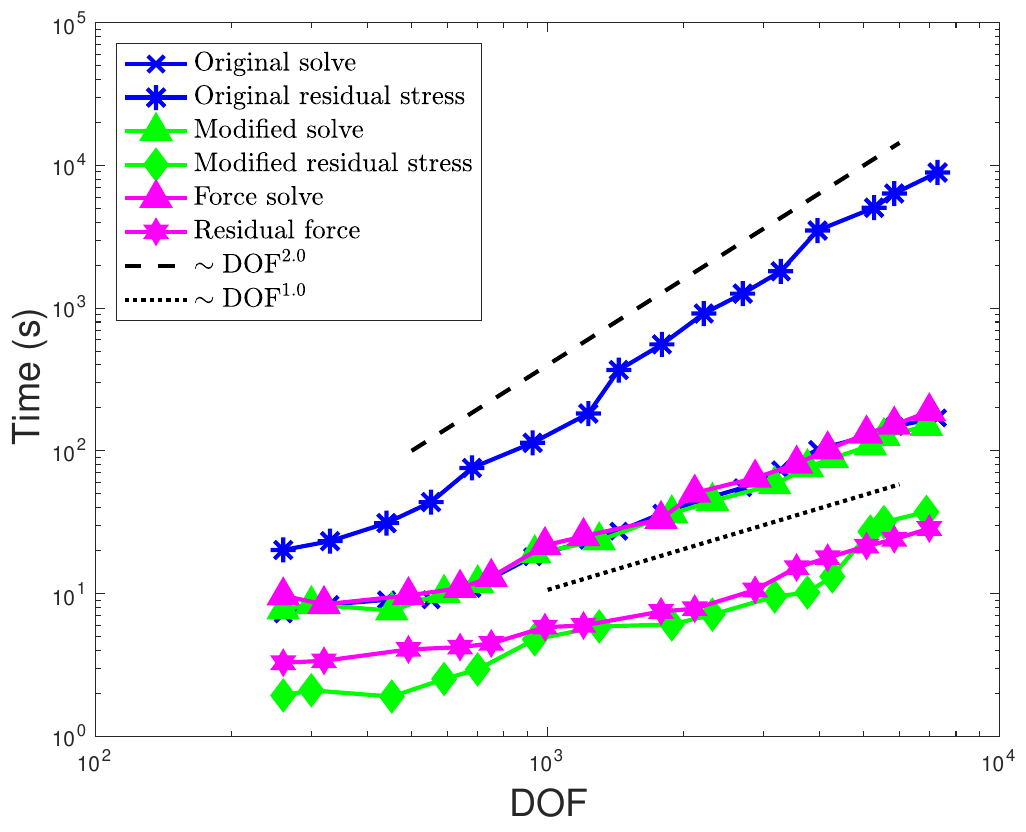}}	
	\caption{The convergence of the error, the CPU time of evaluating the estimators and solving the QC problem and the efficiency factors of the estimators for adaptive GRAC method for the micro-crack.}
	\label{figs:conv_mcrack_comp_grac}
\end{figure} 

Figure \ref{fig:conv_mcrack_comp_grac} shows the convergence of the true error $\|u - I_{\a}u_h\|_{\Use}$ with respect to the number of degrees of freedom (DoF). For the purpose of comparison, we also plot the convergence of the error based on an {\it a priori} graded mesh. It is clearly seen that all three adaptive algorithms achieve the same optimal convergence rate and the lines of convergence are barely distinguishable. 

Figure \ref{fig:conv_mcrack_efffac_comp_grac} shows the efficiency factors  (defined to be the ratios of the error estimators and the true error) of the three different error estimators. As we expect, the residual-force based error estimator developed in the current work only provides an upper bound of the true error and the efficiency factor of which is slightly inferior to those of the stress based estimators developed in \cite{2018_HW_ML_PL_LZ_A_Post_GRAC_2D_SISC} and \cite{2023_YW_HW_Efficient_Adaptivity_AC_JSC}. However, the overestimate is still moderate for a two dimensional problem and it is observed that the residual-force based estimator possesses certain asymptotic exactness.

Figure \ref{fig:time_mcrack_efffac_comp_grac} shows the CPU times for evaluating the {\it a posteriori} error estimators and those for solving the QC problems. As we pointed out at the beginning of Section \ref{sec:sub:resF} and in \cite{2023_YW_HW_Efficient_Adaptivity_AC_JSC}, the {\it original residual-stress based a posteriori error estimator} is highly inefficient as the time of evaluating the error estimator exceeds that of solving the problem itself. On the contrary, the costs of computing the {\it modified residual-stress based error estimator} and the {\it residual-force based error estimator} are comparable and are marginal compared with those of solving the QC problems. This indicates that both estimators may be adopted in practice.

\subsubsection{Adaptive simulations for a micro-crack by blended QC methods}
\label{sec:sub:mcrack blended}


We then conduct the adaptive simulations of a micro-crack by blended QC methods which are done in more than one dimension for the first time. In particular, we test the adaptive BQCE, BQCF and BGFC methods which share the same adaptive algorithm proposed in Algorithm \ref{alg:resF}. We incorporate nearest and next-nearest neighbor interactions in our simulations to demonstrate the capability of our {\it a posteriori} error estimate and adaptive algorithms for dealing with systems with finite range interactions.



The implementations of all three blended QC methods considered in this study remain consistent with those detailed in the prior work~\cite{2016_CO_LZ_GForce_Removal_SISC}. Specifically, the blending function is derived during a preprocessing stage by seeking an approximate minimization of $\|\nabla^2 \beta\|_{L^2}$ as elaborated comprehensively in \cite{2013_ML_CO_BK_BQCE_CMAME}. For the BGFC method, we employ the equivalent "ghost force removal formulation" \eqref{eq:gfc}, selecting $\hat{u}_0=0$ as the predictor for the sake of simplicity.

\begin{figure}[htb]
\centering
	\subfloat[Convergence]{
		\label{fig:conv_mcrack_conv}
		\includegraphics[height=5.5cm]{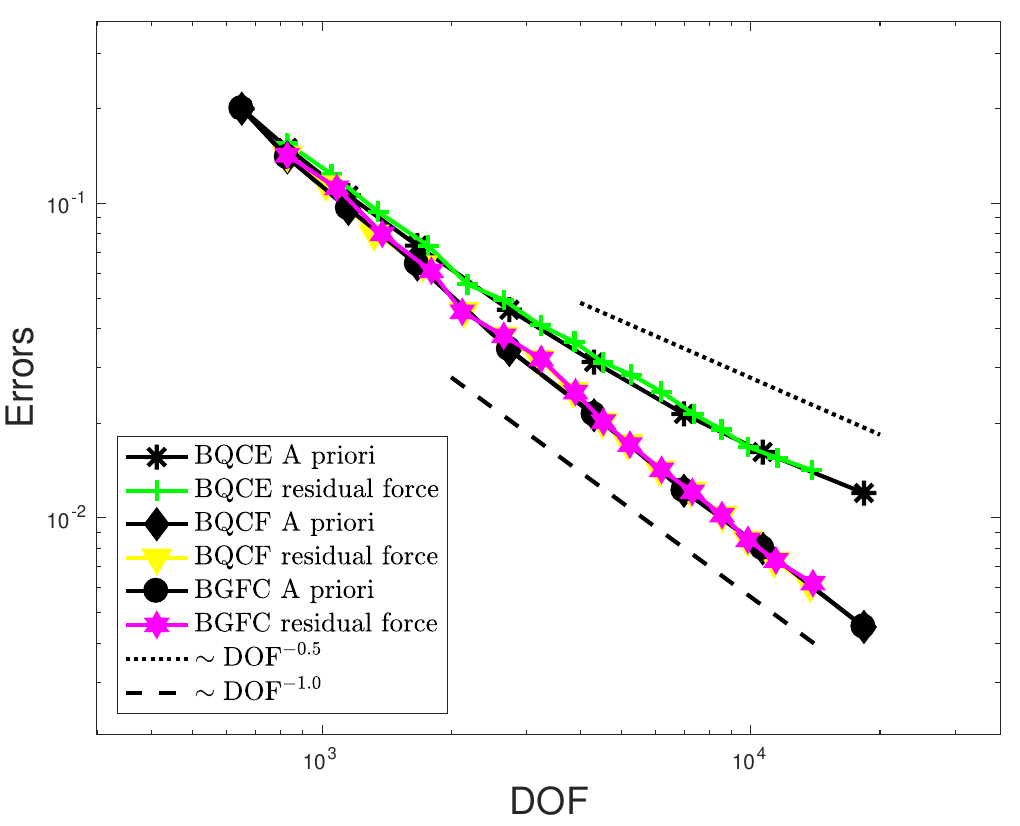}}
	\qquad
	\subfloat[Efficiency factors]{
		\label{fig:conv_mcrack_efffac}
		\includegraphics[height=5.5cm]{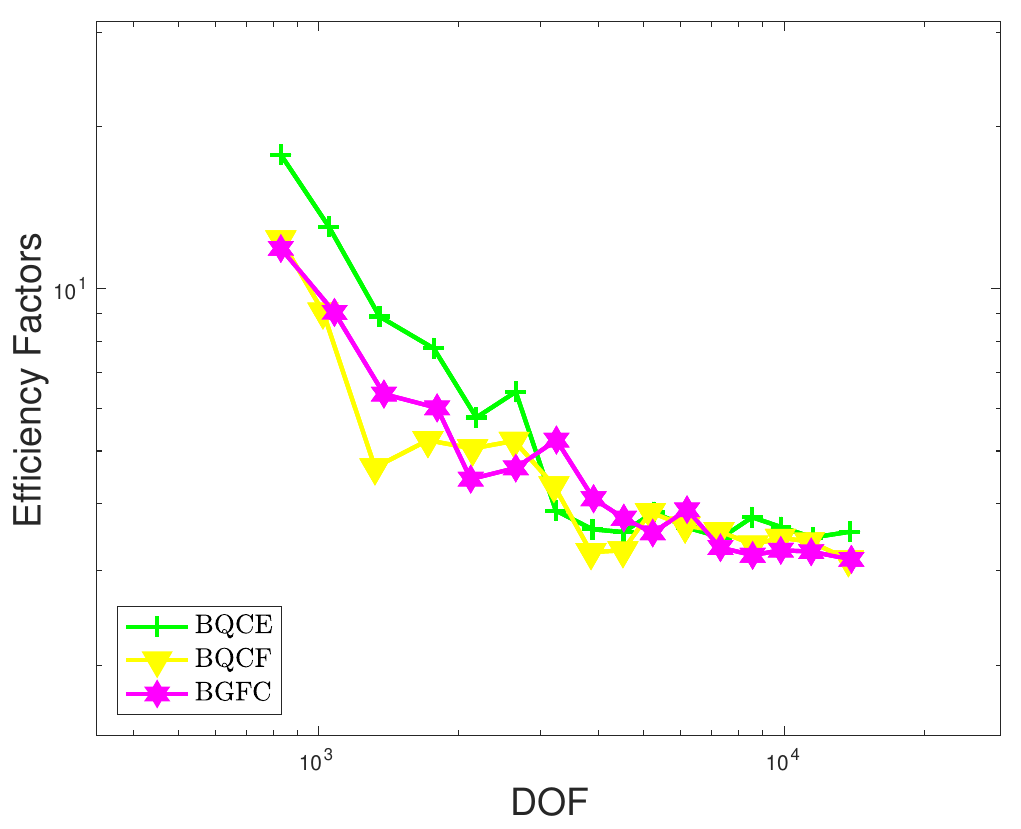}}
	\caption{The convergences of the error and the efficiency factors for both rigorous and approximated error estimators with respect to the number of degrees of freedom for the micro-crack.} 
	\label{figs:conv_mcrack}
\end{figure} 

In Figure~\ref{fig:conv_mcrack_conv}, we present the convergence of the true error for the adaptive BQCE, BQCF, and BGFC methods. Again we observe that all adaptive computations achieve the theoretically verified optimal convergence rate given in Table \ref{tab: decay rate} compared with an {\it a priori} graded mesh.

Figure~\ref{fig:conv_mcrack_efffac} displays the efficiency factors of the residual-force based error estimator for the three adaptive blended QC methods. Again the estimators overestimate the true approximation error at the beginning of the adaptive simulations but become moderate asymptotically which are similar as those for the adaptive QC method with sharp interface. 



\begin{figure}[htb!]
\centering
	\subfloat[CPU times]{
		\label{fig:time_mcrack}
		\includegraphics[height=5.5cm]{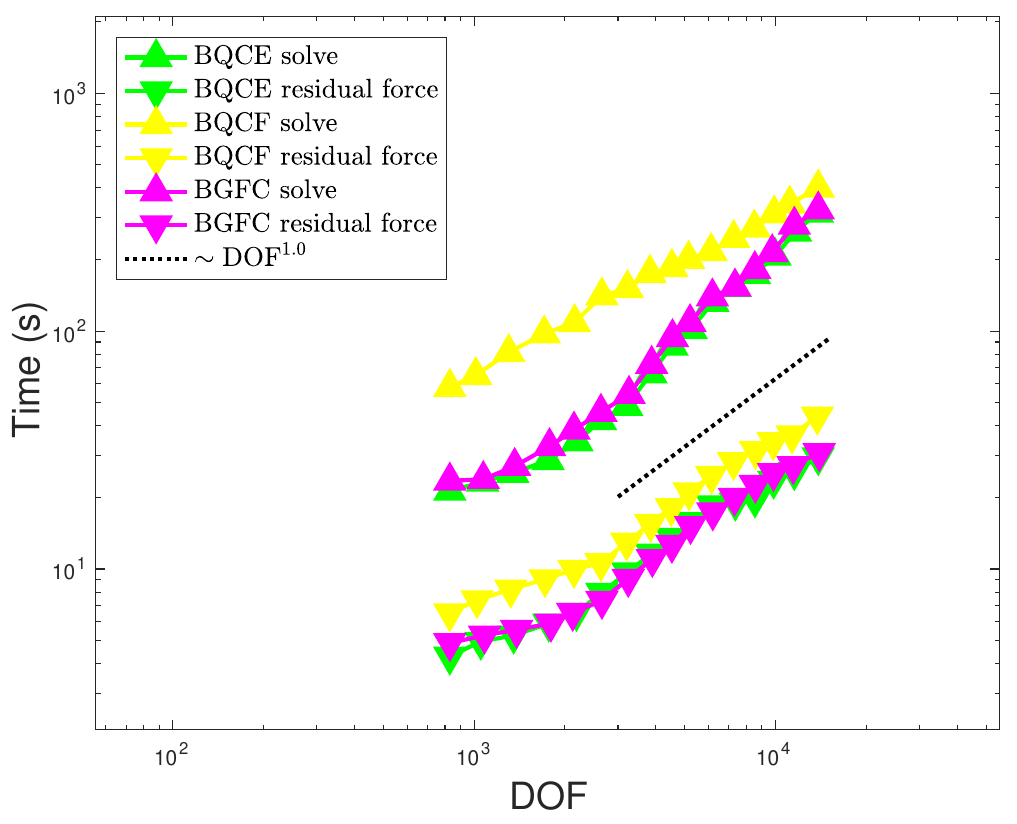}}
		\qquad
	\subfloat[$R_{\a}/R_{\b}$]{
		\label{fig:rarb_mcrack}
		\includegraphics[height=5.5cm]{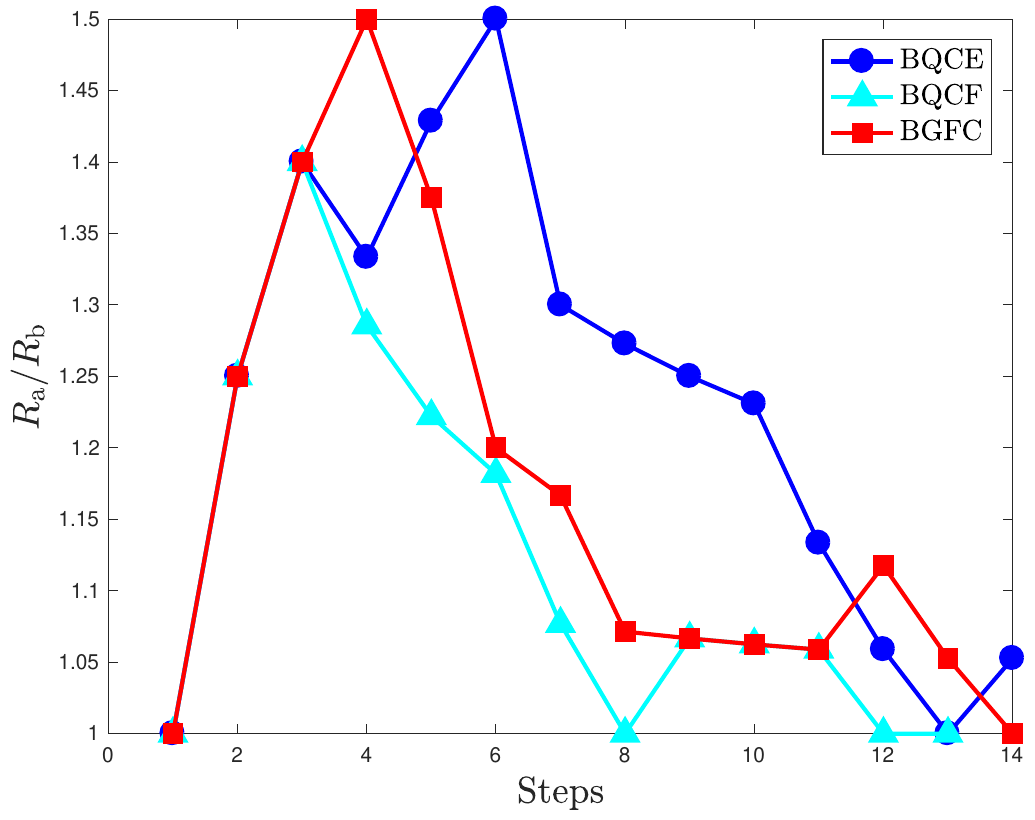}}
	\caption{The CPU times in each steps (left) and the ratio between the radius of the atomistic region $R_{\a}$ and the width of the blending region $R_{\b}$ (right) in the adaptive simulations for the micro-crack. }
	\label{figs:time_and_ratio_mcrack}
\end{figure}



Figure~\ref{fig:time_mcrack} illustrates the CPU times required for computing the error estimators and solving the corresponding blended QC problems, plotted against the number of degrees of freedom $N$. It is evident that the CPU times for the evaluations of the error estimators exhibit a linear scaling for the blended QC and are marginal to those for solving the problems as well. In addition,  we observe that solving the BQCF problem is more computationally expensive than solving for the BQCE or the BGFC under the same number of degrees of freedom. This is because we need to solve a large scale nonlinear system in the BQCF rather than optimization problems, and somehow implies a possible advantage of the energy based methods \cite{2013_ML_CO_BK_BQCE_CMAME, 2014_CO_LZ_GRAC_Coeff_Optim_CMAME}.



We plot the ratio between $R_{\a}$ and $R_{\b}$ in Figure~\ref{fig:rarb_mcrack}, where $R_{\a}$ is the radius of the atomistic region while $R_{\b}$ represent the width of the blending region. We observe that our adaptive algorithms achieve nearly optimal balance, i.e., $R_{\a} \approx R_{\b}$, which also verifies the {\it a priori} assumption on the relationship presented in~\cite[Eq. (19)]{2016_XL_CO_AS_BK_BQC_Anal_2D_NUMMATH}). Furthermore, the ratio is consistently greater than one, indicating that more atoms are allocated to the atomistic region during the adaptive simulations. This substantiates the effectiveness of our adaptive algorithms as they consistently strives to reduce errors by increased atomistic resolutions.

\subsection{Anti-plane screw dislocation}
\label{sec:sub:screw}


The second type of defect we consider for the adaptive simulation is the anti-plane screw dislocation. Following \cite{2016_EV_CO_AS_Boundary_Conditions_for_Crystal_Lattice_ARMA}, we restrict ourselves to an anti-plane shear motion which is illustrated in Figure~\ref{fig:geom_screw}. In this case, the Burgers vector is $b=(0,0,1)^T$ and the center of the dislocation core is chosen to be $\hat{x}:=\frac{1}{2}(1,1,\sqrt{3})^T$, and we assume that there is no 
additional shear deformation applied. The unknown for the anti-plane model is the displacement in $e_3$-direction. The derivation of the far-field {\it predictor} $u_0$ for anti-plane screw dislocation is reviewed in ~\ref{sec:sub:apd:disloc}.

We use a
simplified EAM-type potential here (cf. \cite[Section 2.7]{2016_EV_CO_AS_Boundary_Conditions_for_Crystal_Lattice_ARMA}) which is given by 
\[
V_{\ell}(y) := G\Big( \sum_{\rho \in \Rg_{\ell}} \phi\big( y(\ell+\rho) - y(\ell) \big) \Big),
\]
where
\[
G(s) = 1 + 0.5s^2 \quad \textrm{and} \quad \phi(r) = \sin^2(\pi r).
\]
Note that in this case, the BQCE and BGFC methods defined by \eqref{eq:Ebqce} and \eqref{eq:gfc} are in fact identical since
$\delta \E^{\rm bqce}_{\rm hom}(0)=0$ in \eqref{eq:gfc}, which is an artefact of the anti-plane setting \cite{2016_CO_LZ_GForce_Removal_SISC}. 

Figure~\ref{figs:decay_screw_ff} plots the decay of the {\it residual forces} $\mathcal{F}^{\a}_{\ell}({\pmb 0})$, which verifies the sharp estimates in \cite[Lemma 5.8]{2016_EV_CO_AS_Boundary_Conditions_for_Crystal_Lattice_ARMA}. The black dashed line validates our selection of the constant ($C^{\rm disloc}\approx 0.30$) utilized in defining the truncation error estimator by equation \eqref{eq:tr_est}. This choice is justified by the observation that all the atoms exhibit residual forces below this line. This shows a pivotal distinction from the scenarios involving only point defects, where the residual forces are precisely zero outside the computational domain. We leverage this distinction to construct the truncation estimator, enhancing the robustness and accuracy of the corresponding adaptive algorithm.

\begin{figure}[htb]
\begin{center}
	\includegraphics[height=5.5cm]{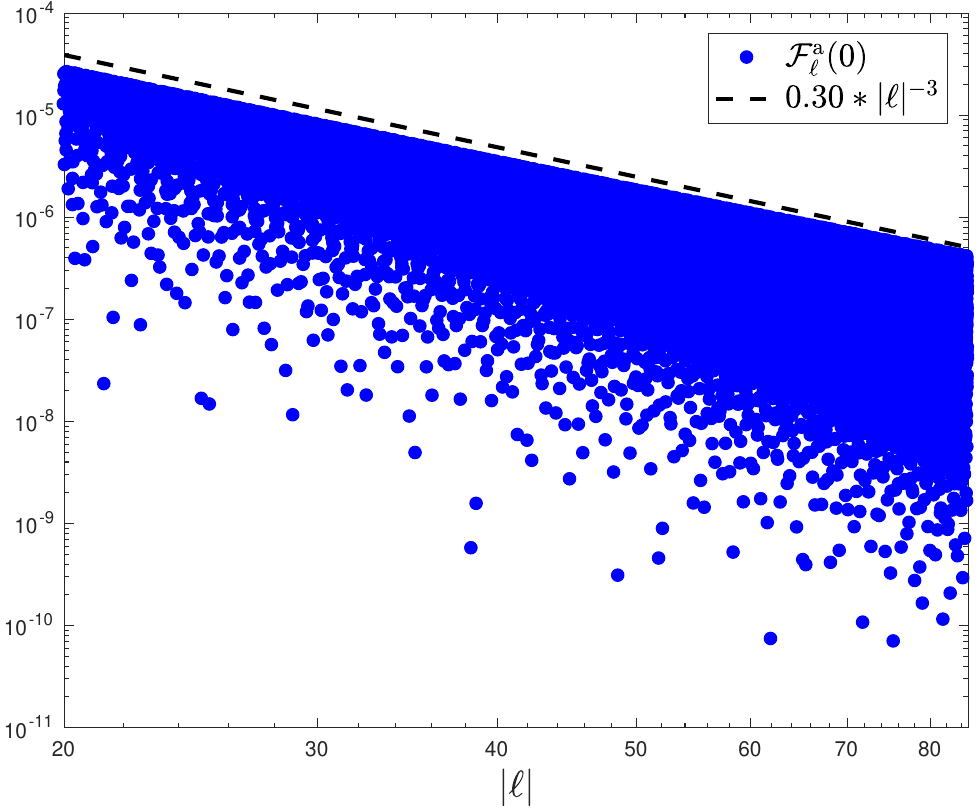}
	\caption{Decay of $\mathcal{F}^{\a}_{\ell}({\pmb 0})$ for the anti-plane screw dislocation.}
	\label{figs:decay_screw_ff}
\end{center}
\end{figure}

\subsubsection{Adaptive simulations for a screw dislocation by sharp interfaced QC methods}
\label{sec:sub:screw sharp interface}

We first present the figures obtained from the adaptive simulations for a screw dislocation by the adaptive GRAC method based on the three different {\it a posteriori} error estimators with the corresponding adaptive algorithms. All three figures in Figure \ref{figs:conv_mcrack_comp_grac} exhibit the same behavior as those in the adaptive simulations for micro-crack, which imply the robustness of our residual- force based error estimator. The only difference compared with the micro-crack is the slower decay of the convergence rate from $O(N^{-1})$ to $O(N^{-0.5})$. This is due to the slower decay of the elastic field induced by the dislocation \cite{2016_EV_CO_AS_Boundary_Conditions_for_Crystal_Lattice_ARMA}. However, we note that this could be improved by fully utilizing the blending methods as shown in \cite{2013_ML_CO_BK_BQCE_CMAME} and also in our adaptive simulations in the following section. 






\begin{figure}[htb]
\centering
	\subfloat[Convergence]{
		\label{fig:conv_screw_comp_grac}
		\includegraphics[height=4.5cm]{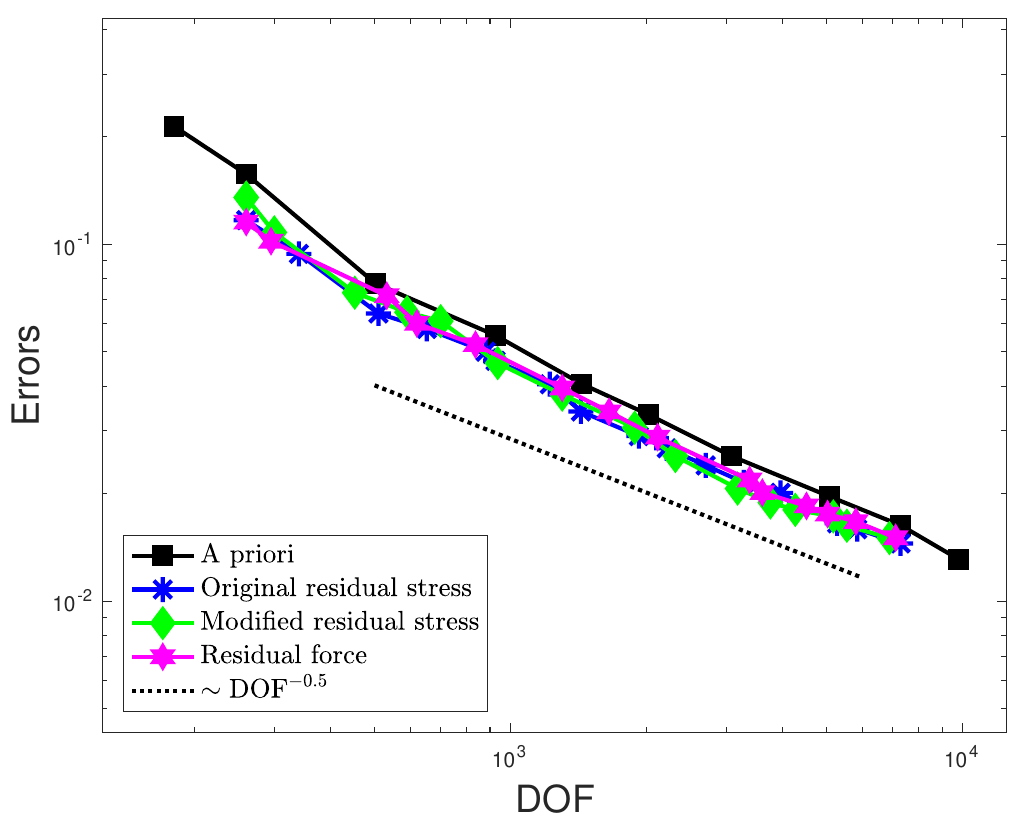}}
	\subfloat[Efficiency factors]{
		\label{fig:conv_screw_efffac_comp_grac}
		\includegraphics[height=4.5cm]{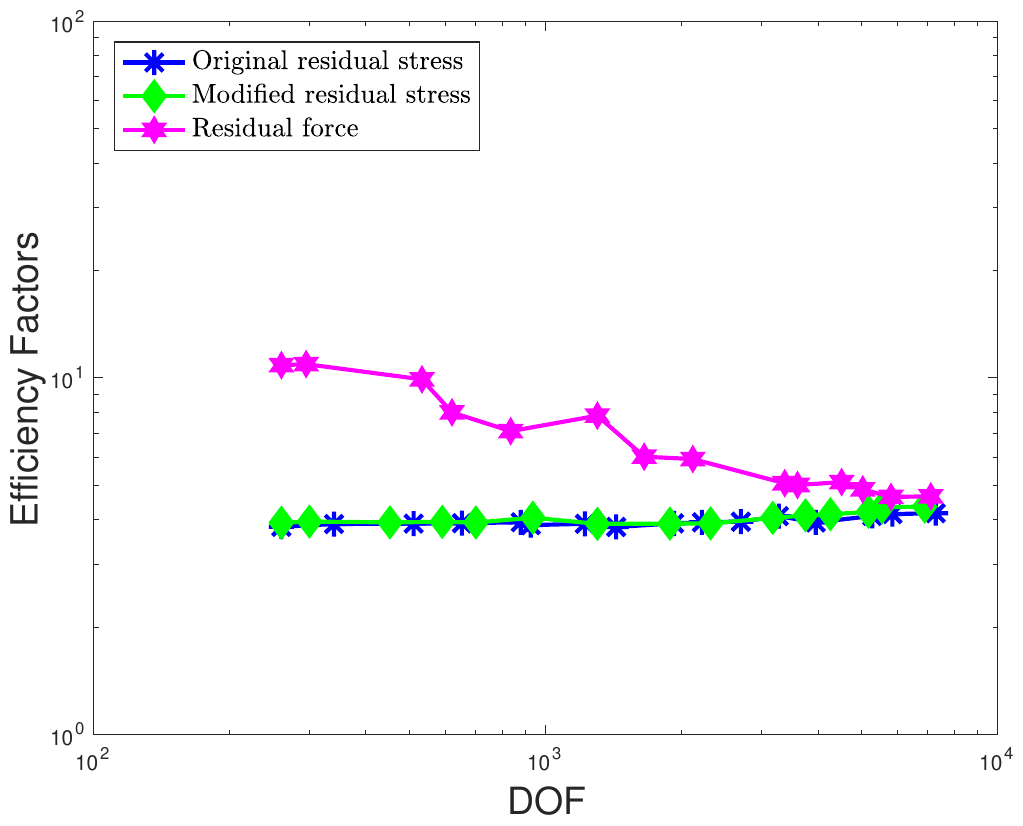}}			
	\subfloat[Estimation time]{
		\label{fig:time_screw_efffac_comp_grac}
		\includegraphics[height=4.5cm]{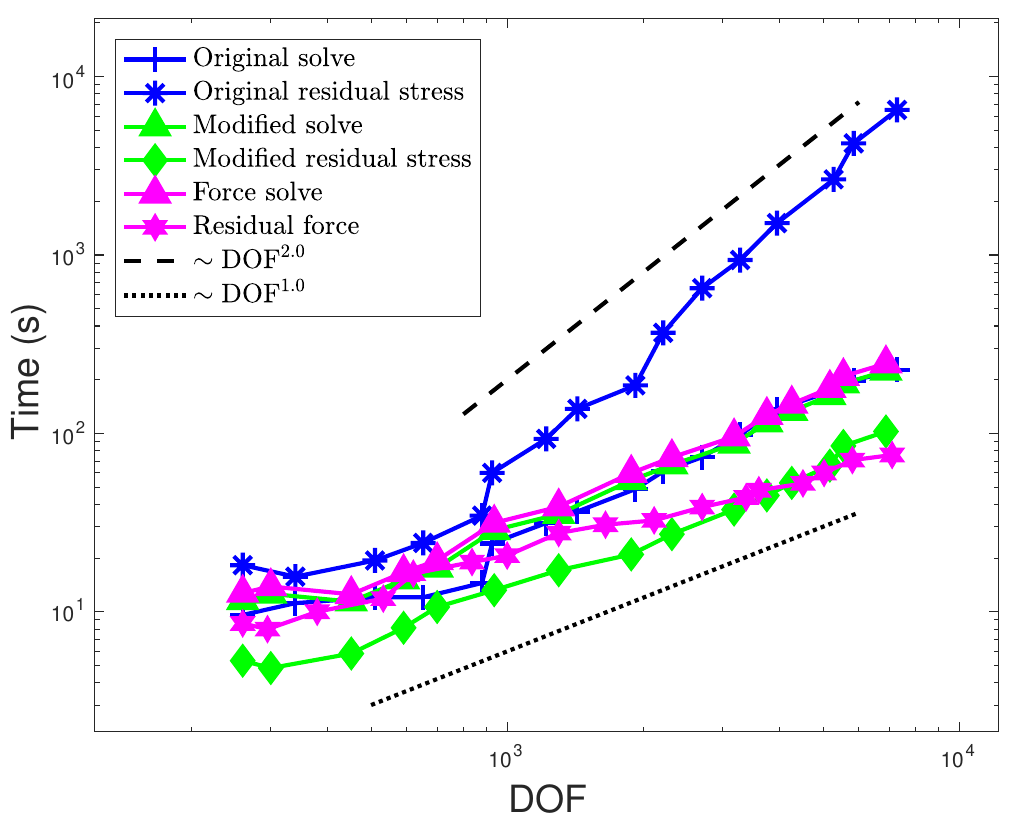}}	
	\caption{The convergence of the error, the CPU time of evaluating the estimators and solving the QC problem and the efficiency factors of the estimators for adaptive GRAC method for screw dislocation.}
	\label{figs:conv_mcrack_comp_grac}
\end{figure} 




\subsubsection{Adaptive simulations for a screw dislocation by blended QC methods}
\label{sec:sub:mcrack blended}

We then conduct the adaptive simulations of a screw dislocation by blended QC methods.  We observe in Figure~\ref{fig:conv_screw_conv} that all adaptive simulations achieve the same optimal convergence rate compared with the {\it a priori} graded mesh. It is noteworthy that the BGFC method is identical to the BQCE method in the anti-plane setting. This is because the correction term $\delta \E^{\rm bqce}_{\rm hom}(0)=0$ in \eqref{eq:gfc} which is explained in \cite[Section 4.2]{2016_CO_LZ_GForce_Removal_SISC}. Figure~\ref{fig:conv_screw_efffac} plots the efficiency factors for the blended QC methods, which are still moderate. Figure~\ref{fig:time_screw} plots the total CPU times versus the numbers of degrees of freedom $N$ for the evaluations of the error estimator as well as solving the QC problems themselves. Figure~\ref{fig:rarb_screw} visualizes the relationship between $R_{\a}$ and $R_{\b}$ during the adaptive simulations. 



\begin{figure}[htb]
\centering
	\subfloat[Convergence]{
		\label{fig:conv_screw_conv}
		\includegraphics[height=5.5cm]{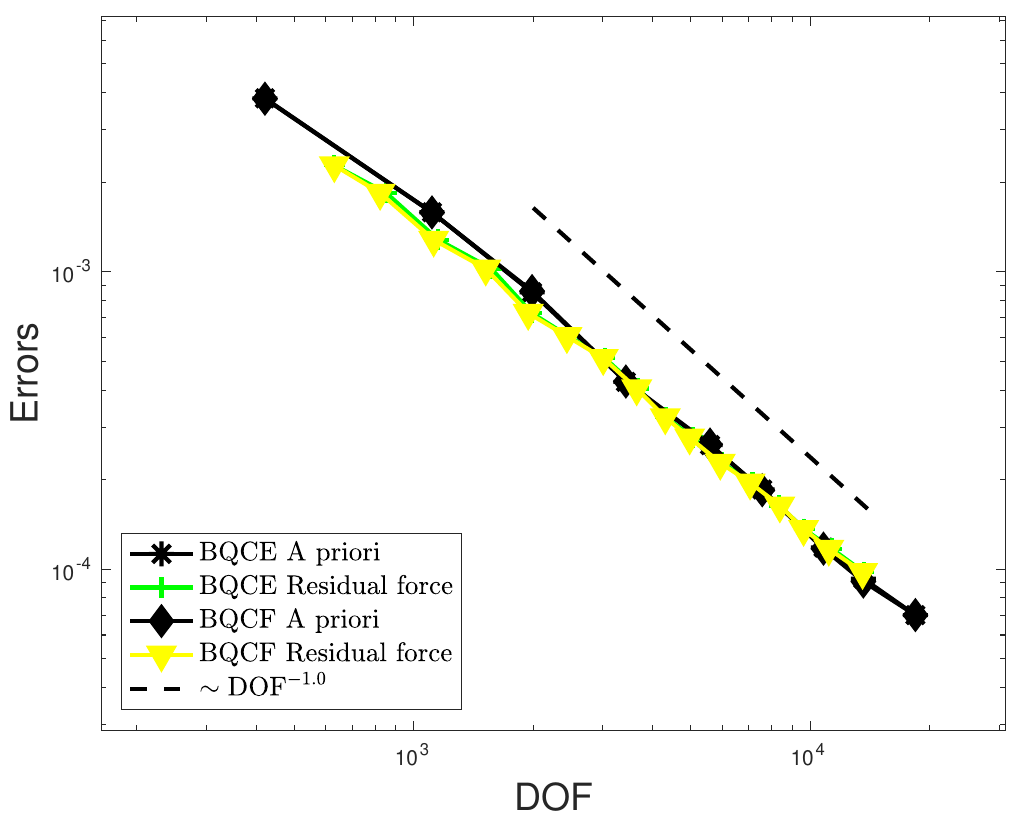}}
	\qquad
	\subfloat[Efficiency factors]{
		\label{fig:conv_screw_efffac}
		\includegraphics[height=5.5cm]{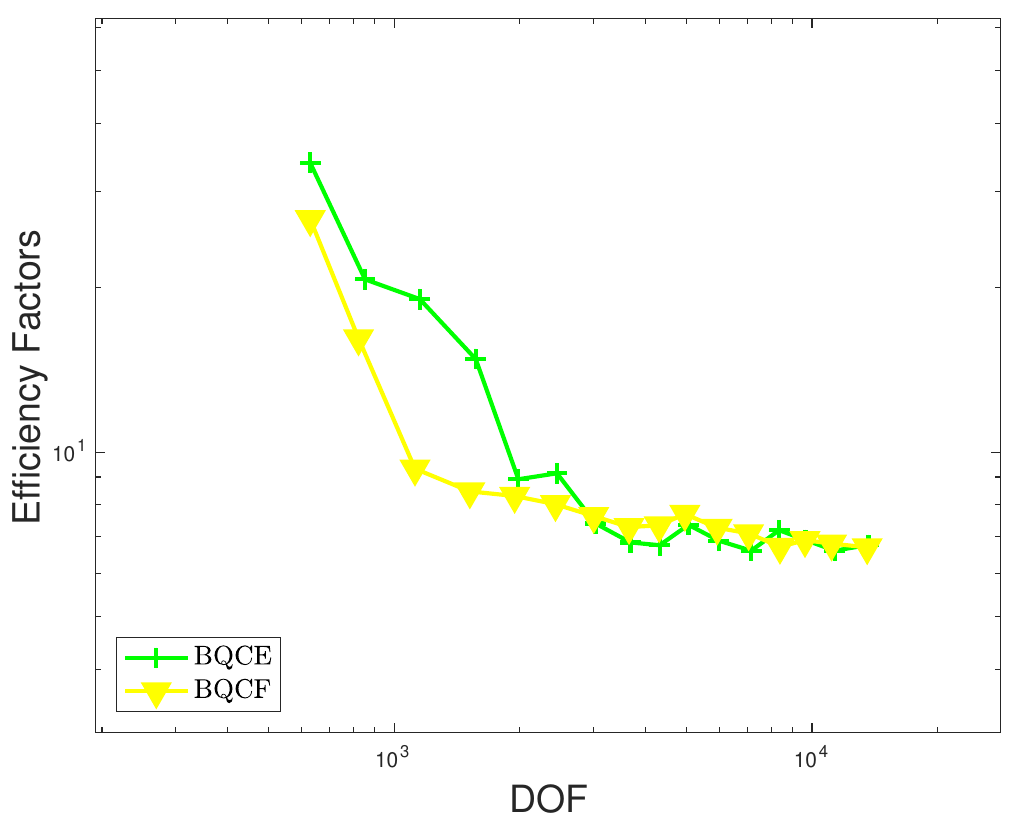}}
	\caption{The convergences of the error and the efficiency factors for different error estimators for the anti-plane screw dislocation. In this case BGFC is identity to BQCE due to the artefact of the anti-plane setting.}
	\label{figs:conv_screw}
\end{figure} 
\begin{figure}[htb!]
\centering
	\subfloat[CPU times]{
		\label{fig:time_screw}
		\includegraphics[height=5.5cm]{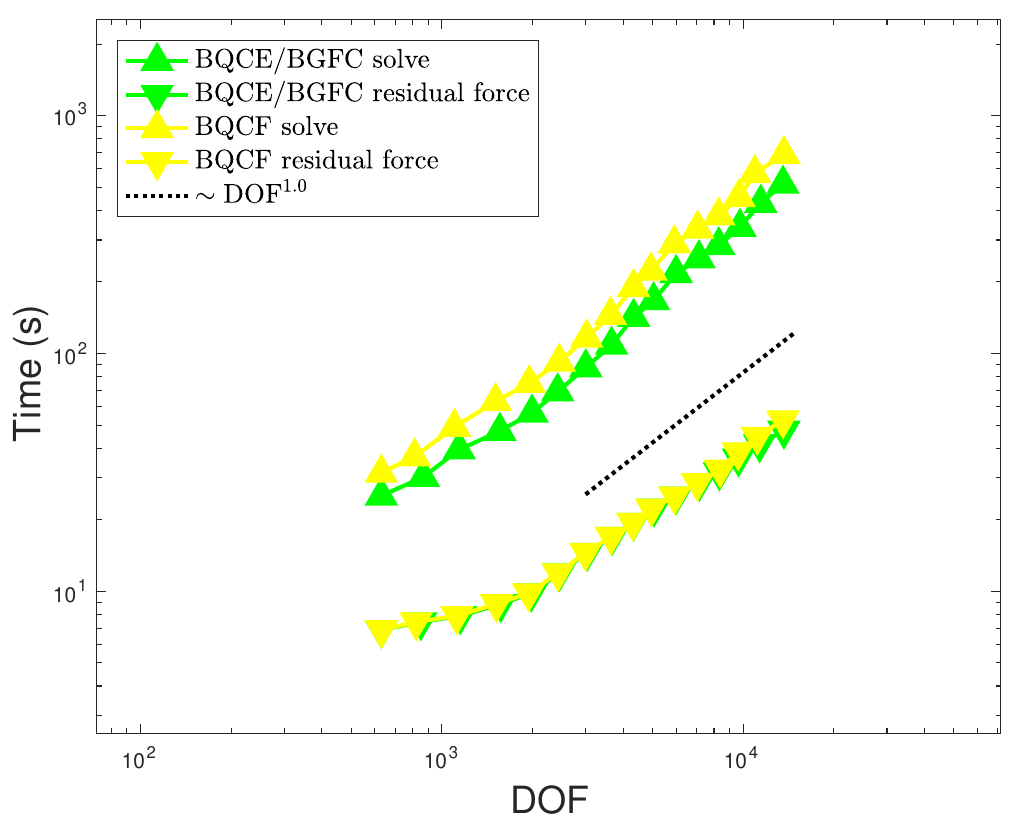}}
		\qquad
	\subfloat[$R_{\a}/R_{\b}$]{
		\label{fig:rarb_screw}
		\includegraphics[height=5.5cm]{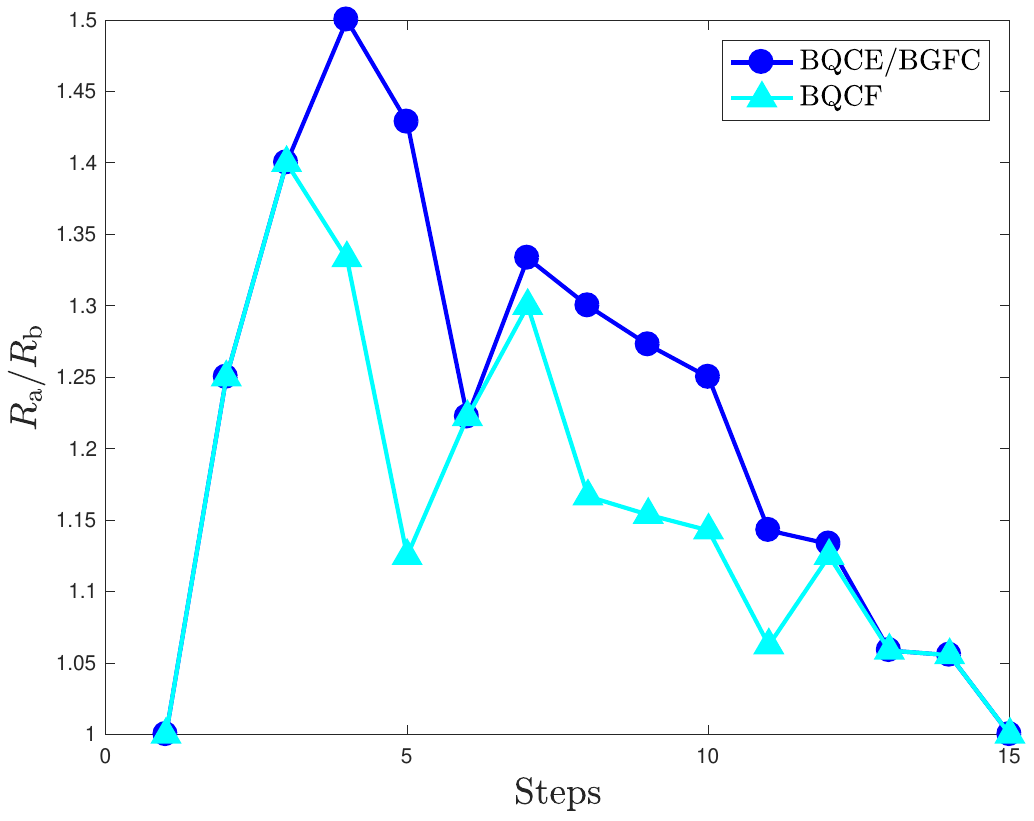}}
	\caption{The CPU times for each steps (left) and the ratio between the radius of the atomistic region $R_{\a}$ and the width of the blending region $R_{\b}$ (right) in the adaptive computations for the anti-plane screw dislocation.}
	\label{figs:time_and_rarb_screw}
\end{figure}

In summary, we observe similar results as those in the case of the micro-crack. These results essentially demonstrate not only the accuracy but also the efficiency of the proposed residual-force based error estimator, as well as the effectiveness of the corresponding adaptive algorithms in the context of the anti-plane screw dislocation.


\subsubsection{Discussion}

Besides the capability of guiding the adaptivity for the blended QC methods in more than one dimension which is yet to be developed for the residual-stress based approach, we have to comment on two additional advantages of the residual-force based error estimates particularly for sharp interfaced methods that are not easily illustrated by figures (c.f \cite{2018_HW_ML_PL_LZ_A_Post_GRAC_2D_SISC, 2023_YW_HW_Efficient_Adaptivity_AC_JSC, 2020_ML_PL_LZ_Finite_Range_A_Post_2D_CICP}). 

The first one is the significant reduction of complexity in the implementation. Besides a unified formulation and thus the reusability of the codes for different methods, the residual-force based estimator does not rely on the so-called ``stress tensor correction" which involves a laborious optimization problem \cite[Algorithm 1]{2018_HW_ML_PL_LZ_A_Post_GRAC_2D_SISC}.  The necessity of such correction stems from the inherent lack of uniqueness in the formulation of the residual stress. Without such correction, the adaptive algorithm will tend to put redundant degrees of freedom around the interface in each adaptive step so that the optimal convergence rate may not be achieved. In addition, according to our sampling strategy, there is no need of a manually selected "interface buffer" region where the finite element nodes are deliberately set to coincide with the lattice points \cite[Figure 5]{2023_YW_HW_Efficient_Adaptivity_AC_JSC}.



The second one is the possibility to be extended to the systems with finite range interactions. The residual-stress based error estimate for the GRAC method with finite range interactions has been discussed in~\cite{2020_ML_PL_LZ_Finite_Range_A_Post_2D_CICP}. However, in the setting of finite range interactions, the atomistic stress loses its locality and thus prompts certain ``ad hoc" approximations to render the residual amenable to the elementwise summation. Moreover, such approximations are effective only for simple point defects but become untenable for more intricate defects such as dislocations and cracks which are discussed immediately in the subsequent section.  



\subsection{Anti-plane crack}
\label{sec:sub:crack}


We finally consider the defect of anti-plane crack, which is of more practical interest and considerably more complicated than the previous two types of defects. To the best knowledge of the authors, it is also the first time that such type of defect is simulated by QC methods. We apply the same setting as that for the screw dislocation. The only difference is that the far-field {\it predictor} $u_0$ is more complicated, which is briefly reviewed in \ref{sec:sub:apd:crack}. 


The significance and motivation behind the adaptive simulations of cracks, distinct from micro-cracks and screw dislocations, encompass the following key aspects. First of all, the construction of the energy based sharp interfaced method for cracks presents considerable challenges. In particular, according to the fundamental principles of the GRAC method~\cite{2012_CO_LZ_GRAC_Construction_SIAMNUM}, the determination of reconstruction parameters $C_{\ell;\rho, \zeta}$ around the interface for cracks is hindered by the presence of crack tips. Consequently, the blended QC methods appear to be a more suitable choice for simulating cracks. Second, conducting rigorous {\it a priori} analysis of QC methods proves considerably more challenging. Within the framework proposed in \cite{2016_EV_CO_AS_Boundary_Conditions_for_Crystal_Lattice_ARMA}, the critical component of the {\it a priori} analysis is the establishment of rigorous equilibrium decay estimates. This involves an intricate analysis of the lattice Green's function along with the theoretical proof of the estimates of the residual forces. However, when it comes to the {\it a posteriori} error estimates, the intricate technical nuances aforementioned are somehow circumvented and the estimates can be applied to improve the efficiency of the simulations of the cracks, which highlight the novelty and the significance of our proposed methodology. In light of these considerations, we proceed with numerical tests employing the residual-force based error estimator for all three blended QC methods.

We first show in Figure~\ref{figs:decay_crack_ff} the decay of the residual forces $\mathcal{F}^{\a}_{\ell}({\pmb 0})$ for the anti-plane crack. We observe that the residual forces on the {\it surface atoms} lying around the crack tip only decays as $|\ell|^{-1.5}$ and those on the other layers decay as $|\ell|^{-2.5}$. The black dashed and the black dotted lines validate our choice of the constants ($C^{\rm crack}_{\rm surf}$ for the {\it surface atoms} and $C^{\rm crack}_{\rm oth}$ for others) as the residual forces are all under these two lines. According to the discussions in Remark~\ref{rmk:trun_crack}, we choose $C^{\rm crack}=C^{\rm crack}_{\rm surf}+C^{\rm crack}_{\rm oth}=3.35$ in our adaptive simulations.

\begin{figure}[htb]
\begin{center}
	\includegraphics[height=5.5cm]{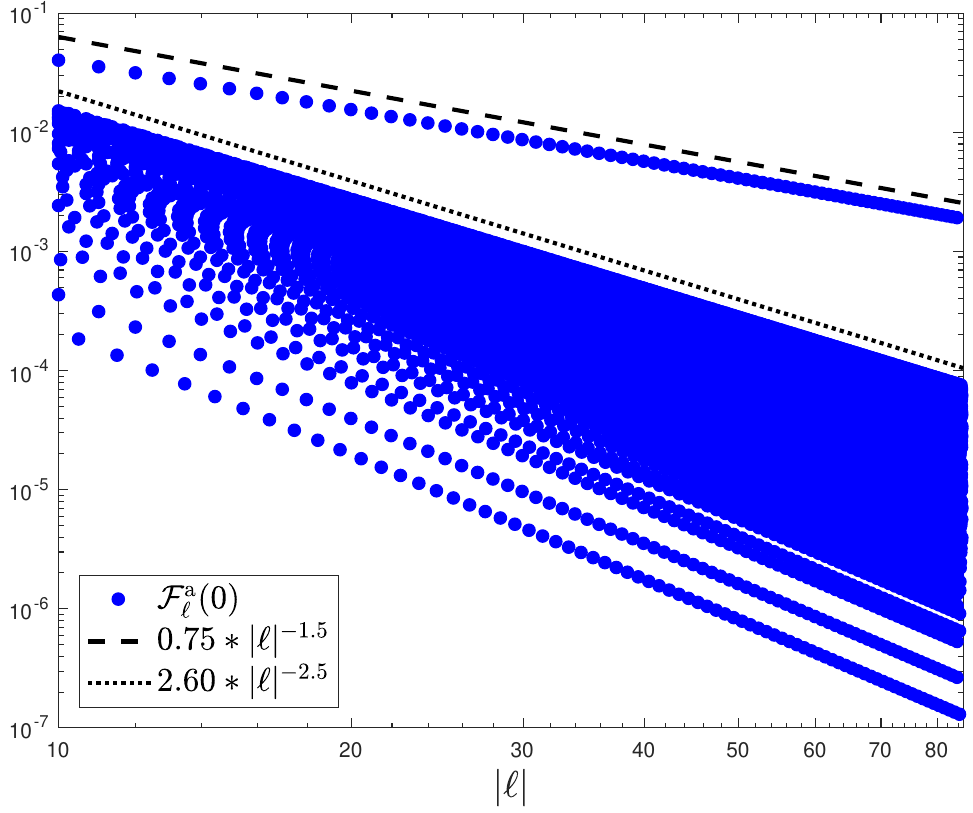}
	\caption{Decay of $\mathcal{F}^{\a}_{\ell}({\pmb 0})$ for the anti-plane crack.}
	\label{figs:decay_crack_ff}
\end{center}
\end{figure}

The computational mesh and the atomistic and blending regions used in the construction of the blended QC methods are then illustrated in Figure~\ref{figs:geom_crack}. Here we need to point out that though the same anti-plane setting is applied as that for the anti-plane screw dislocation, the BQCE and the BGFC methods are no longer identical for the anti-plane crack since a row of mesh around the crack tip is removed as shown in Figure~\ref{figs:geom_crack}.

\begin{figure}[htb]
\begin{center}
	\includegraphics[height=5.5cm]{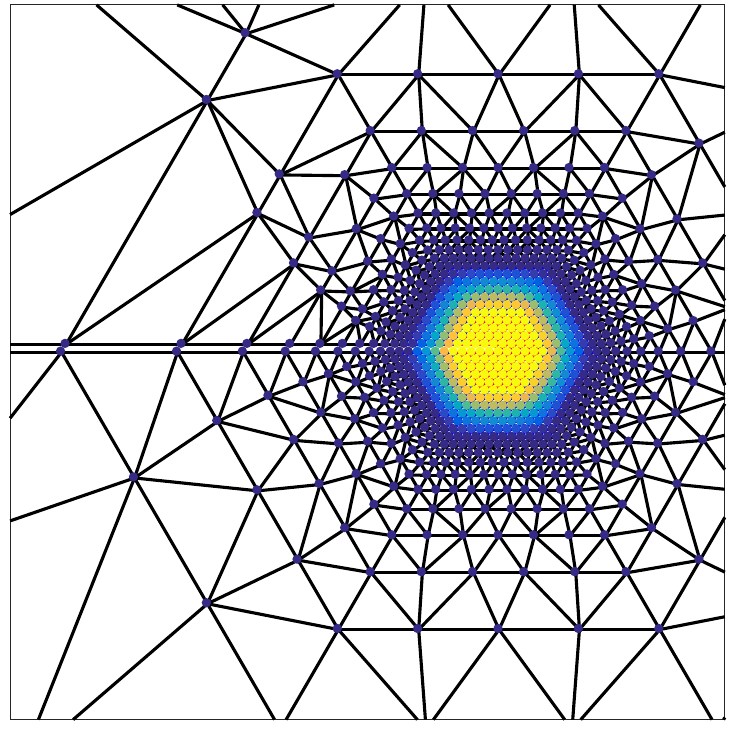}
	\qquad
	\includegraphics[height=5.5cm]{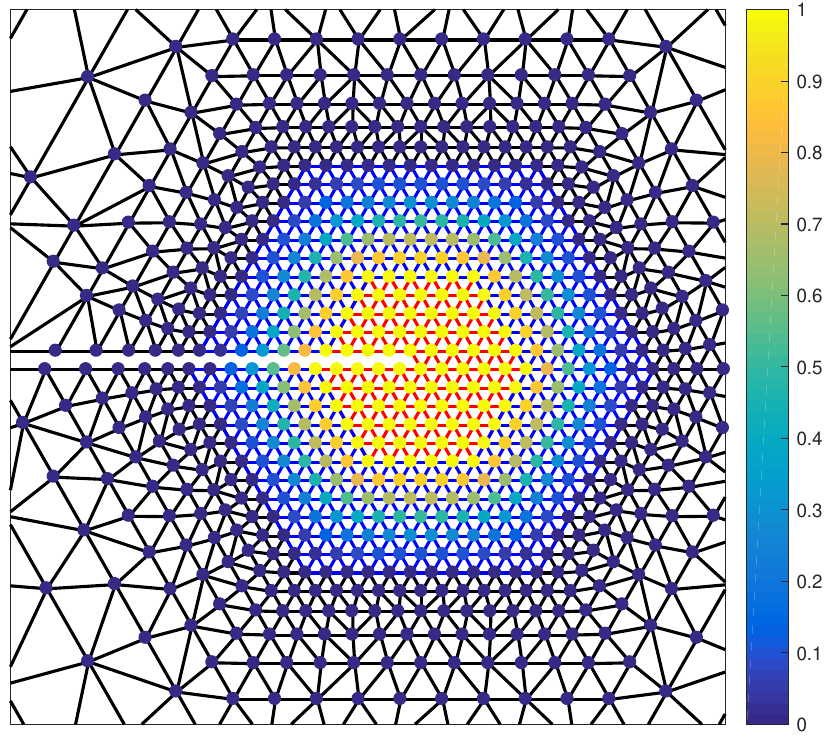}
	\caption{The illustration of the computational mesh and the atomistic
region as used in the construction of the blended QC methods for anti-plane crack. A row of mesh around crack tip is removed to simulate the anti-plane crack system \cite{2019_BM_TM_CO_Anal_Antiplane_Fracture_M3AS}.}
	\label{figs:geom_crack}
\end{center}
\end{figure}

Figure~\ref{fig:conv_mcrack_conv} shows that all the adaptive methods achieve the same optimal convergence rate ($N^{-0.25}$ for all three blended methods) compared with an a priori graded mesh. As we pointed out in Remark \ref{rmk:trun_crack}, the {\it rigorous a priori} error estimate for this case deserves further investigation.


\begin{figure}[htb]
\centering
	\subfloat[Convergence]{
		\label{fig:conv_crack_conv}
		\includegraphics[height=5.5cm]{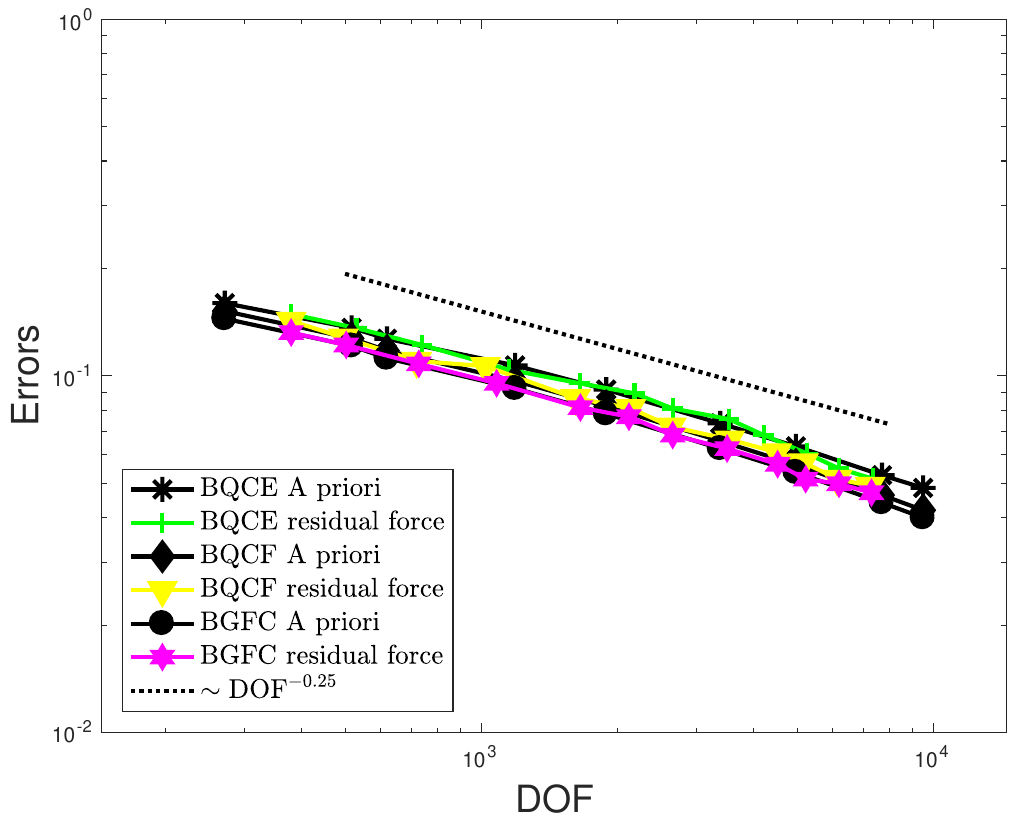}}
	\qquad
	\subfloat[Efficiency factors]{
		\label{fig:conv_crack_efffac}
		\includegraphics[height=5.5cm]{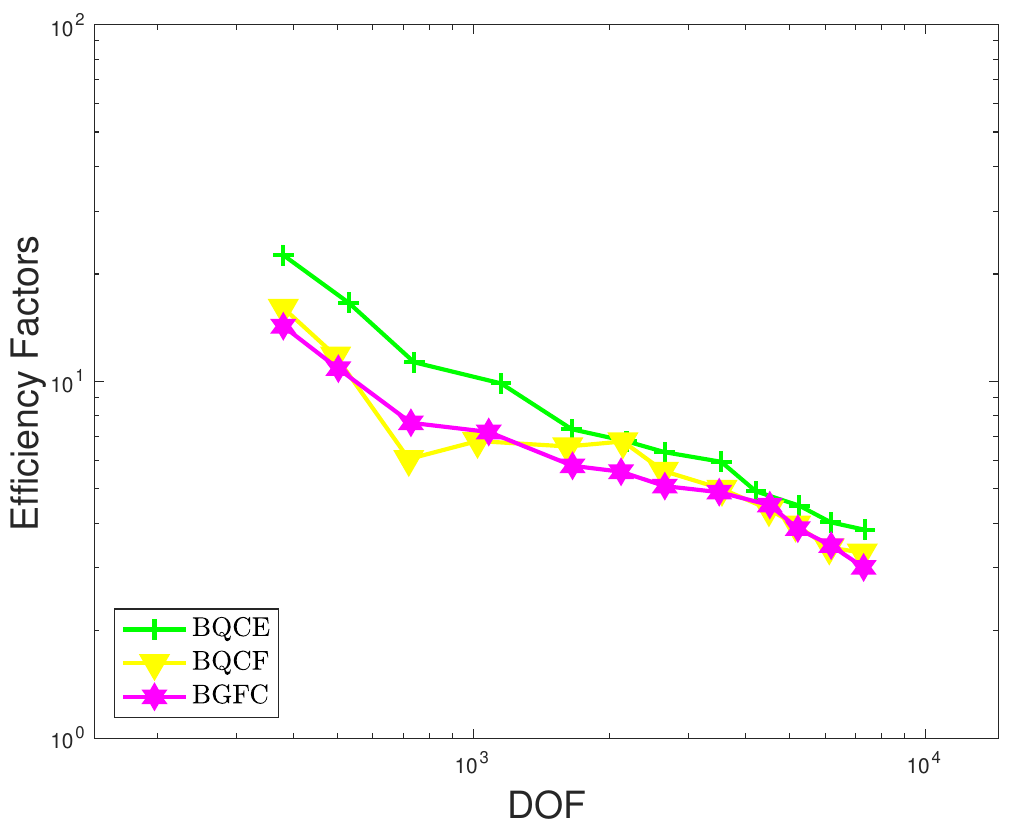}}
	\caption{The convergences of the error and the efficiency factors for different error estimators with respect to the number of degrees of freedom for the anti-plane crack.}
	\label{figs:conv_crack}
\end{figure} 

Figure~\ref{fig:conv_crack_efffac}, Figure~\ref{fig:time_crack} and Figure~\ref{fig:rarb_crack} plot the efficiency factors, CPU times and the ratio between $R_{\a}$ and $R_{\b}$.  These figures exhibit the same behaviors as those in the cases of the micro-crack and the anti-plane screw dislocation. This verifies the robustness of the residual-force based error estimate and the algorithms, and implies the potential capability of the adaptive strategy for more complicated defects.

\begin{figure}[htb]
\centering
	\subfloat[CPU times]{
		\label{fig:time_crack}
		\includegraphics[height=5.5cm]{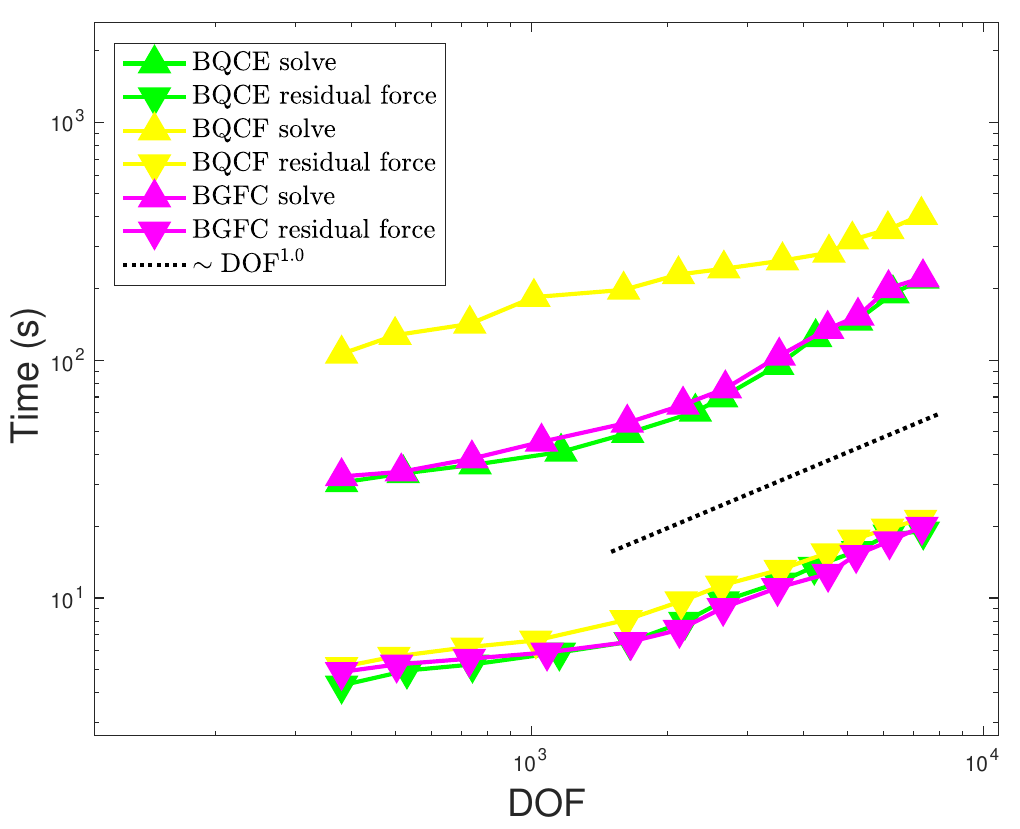}}
		\qquad
	\subfloat[$R_{\a}/R_{\b}$]{
		\label{fig:rarb_crack}
		\includegraphics[height=5.5cm]{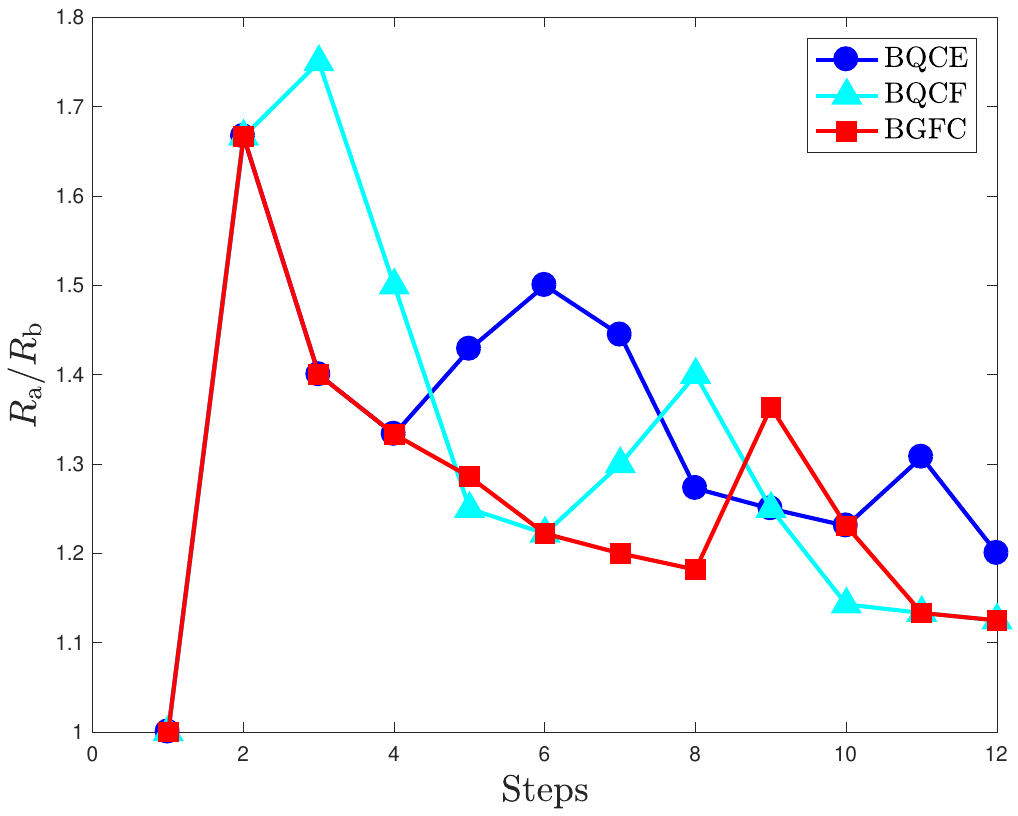}}
	\caption{The CPU times for each steps (left) and the ratio between the radius of the atomistic region $R_{\a}$ and the width of the blending region $R_{\b}$ (right) for three blended QC methods in the adaptive computations for the anti-plane crack.}
	\label{figs:time_and_rarb_crack}
\end{figure} 




\section{Conclusion}
\label{sec:con}

In this work we propose a unified framework for theory based {\it a posteriori} error estimate of consistent quasicontinuum (QC) methods for the molecular mechanics simulations of crystal defects.  Based on the unified error estimate, we develop the adaptive algorithms for a range of QC methods and conduct the adaptive simulations for several important crystalline defects, including the micro-crack, the anti-plane dislocation and the anti-plan crack. In particular, we have essentially made three important advancements in this study compared with previous researches on the adaptive QC methods \cite{2014_CO_HW_A_Post_ACC_IMANUM, 2018_HW_ML_PL_LZ_A_Post_GRAC_2D_SISC, 2023_YW_LZ_Adaptive_Multigrid_AC_3D_JCP, 2023_YW_HW_Efficient_Adaptivity_AC_JSC}. The first one is the derivation of the unified residual-force based {\it a posteriori} error estimate that is independent of particular QC method we consider. This leads to the save of a significant amount of effort for the implementation of the {\it a posteriori} error control problems. The second one is the development of the adaptive strategy for the blended QC methods. Such adaptive strategy is able to simultaneously guide the refinement of the mesh in the continuum region as well as the allocations of the atomistic and blending regions. The third one is the simulations of a crack by adaptive QC methods for the first time. It significant since the geometry of a crack is complicated and the {\it a priori} analysis possesses substantial challenges in the context of QC methods.

Although we believe that the {\it a posteriori} error estimate and the adaptive algorithms in this work are generally applicable to a wide range of concurrent multiscale methods, this research still raises a few open problems for further investigation.

The first one is the {\it a priori} error estimates of QC methods for the straight edge dislocation and the anti-plane crack. To address this issue, a possible solution is to investigate an equivalent ghost force removal formulation \eqref{eq:gfc}, where a suitable ``predictor" $\hat{u}_0$ needs to be constructed. The second one is the adaptive simulations for defects in three dimensions. This presents significant challenges in mesh generation and adaptation in the implementation. Recent advancements \cite{2023_YW_LZ_Adaptive_Multigrid_AC_3D_JCP} may provide a preliminary in this direction. The third one is the extension to more complex crystal defects. Practical crystal defects such as dislocation nucleation and grain boundaries have already attracted significant attention. Since the {\it a priori} analysis for such problems is difficult, these are precisely the scenarios where the adaptive methods are expected to demonstrate their great advantage. We believe that the approach presented in the current research contributes a solid groundwork for these practical and important problems.

	\appendix
	\renewcommand\thesection{\appendixname~\Alph{section}}
	

\section{Far-field Predictors}
\label{sec:appendixU0}
\renewcommand{\theequation}{A.\arabic{equation}}
\renewcommand{\thefigure}{A.\arabic{figure}}
\renewcommand{\thealgorithm}{A.\arabic{algorithm}}
\setcounter{equation}{0}
\setcounter{figure}{0}
\setcounter{algorithm}{0}

\subsection{Dislocations}
\label{sec:sub:apd:disloc}

\newcommand{\ulin}{u^{\rm lin}}
\newcommand{\burg}{{\sf b}}

We model dislocations by following the setting used in \cite{2016_EV_CO_AS_Boundary_Conditions_for_Crystal_Lattice_ARMA}. We consider a model for straight dislocations obtained by projecting a 3D crystal into 2D. Let $B \in \R^{3\times 3}$ be a nonsingular matrix. Given a Bravais lattice $B\Z^3$ with dislocation direction parallel to $e_3$ and Burgers vector $\burg=(\burg_1,0,\burg_3)$, we consider
displacements $W: B\Z^3 \rightarrow \R^3$ that are periodic in the direction of the dislocation direction of $e_3$. Thus, we choose a projected reference lattice $\L := A\Z^2 := \{(\ell_1, \ell_2) ~|~ \ell=(\ell_1, \ell_2, \ell_3) \in B\Z^3\}$. We also introduce the projection operator 
\begin{equation}\label{eq:P}
    P(\ell_1, \ell_2) = (\ell_1, \ell_2, \ell_3) \quad \text{for}~\ell \in B\Z^3.
\end{equation}
It can be readily checked that this projection is again a Bravais lattice. For anti-plane screw dislocation, $\L$ is obtained as projection of a 3D Bravais lattice along the screw dislocation direction (and the direction of slip) and we restrict the displacements of the form $u = (0, 0, u_3)$.

We follow the constructions in \cite{2019_JB_MB_CO_Crystal_Symmetries_SIAMJMA, 2016_EV_CO_AS_Boundary_Conditions_for_Crystal_Lattice_ARMA} for modeling dislocations and prescribe $u_0$ as follows. Let $\L\subset\R^2$, $\hat{x}\in\R^2$ be the position of the dislocation core and $\Upsilon := \{x \in \R^2~|~x_2=\hat{x}_2,~x_1\geq\hat{x}_1\}$ be the ``branch cut'', with $\hat{x}$ chosen such that $\Upsilon\cap\Lambda=\emptyset$.

We define the far-field predictor $u_0$ by solving the continuum linear elasticity (CLE)
\begin{eqnarray}\label{CLE}
\nonumber
\mathbb{C}^{j\beta}_{i\alpha}\frac{\partial^2 u^{\rm lin}_i}{\partial x_{\alpha}\partial x_{\beta}} &=& 0 \qquad \text{in} ~~ \R^2\setminus \Upsilon,
\\
u^{\rm lin}(x+) - u^{\rm lin}(x-) &=& -\burg \qquad \text{for} ~~  x\in \Upsilon \setminus \{\hat{x}\},
\\
\nonumber
\nabla_{e_2}u^{\rm lin}(x+) - \nabla_{e_2}u^{\rm lin}(x-) &=& 0 \qquad \text{for} ~~  x\in \Upsilon \setminus \{\hat{x}\},
\end{eqnarray}
where the forth-order tensor $\mathbb{C}$ is the linearised Cauchy-Born tensor (derived from the potential $V$, see \cite[\S~7]{2016_EV_CO_AS_Boundary_Conditions_for_Crystal_Lattice_ARMA} for more detail).

We mention that for the anti-plane screw dislocation, under the proper assumptions on the interaction range $\Rg$ and the potential $V$, the first equation in \eqref{CLE} simply becomes to $\Delta u^{\rm lin} = 0$ \cite{2019_JB_MB_CO_Crystal_Symmetries_SIAMJMA}. The system \eqref{CLE} then has the well-known solution 
\begin{align}\label{predictor-u_0-dislocation}
u_0(x) := u^{\rm lin}(x) = \frac{\burg}{2\pi}\arg(x-\hat{x}),
\end{align}
where we identify $\R^2 \cong \C$ and use $\Upsilon-\hat{x}$ as the branch cut for arg.

Note that for the purpose of analysis, we have $\nabla u_0 \in C^{\infty}(\R^2\setminus\{0\})$ and $|\nabla^j u_0| \leq C|x|^{-j}$
for all $j \geq 0$ and $x \neq 0$.

\subsection{Cracks}
\label{sec:sub:apd:crack}

We present the setting of cracks by following \cite{2020_MB_TH_CO_Cell_Size_Crack_Prop_M2NA}, which stems from the limitation of the continuum elasticity approaches to static crack problems. Similar with the discussions of dislocations, we introduce the following CLE
\begin{align}\label{CLE_crack}
\nonumber
- {\rm div}~(\mathbb{C}:\nabla u) = 0 \qquad &\text{in} ~~ \R^2\setminus \Gamma,
\\
(\mathbb{C} : \nabla u)\nu = 0 \qquad &\text{on} ~~ \Gamma,
\end{align}
supplied with a suitable boundary condition coupling to the bulk \cite{1998_FL_Dynamic_CAMBRIDGE}. It is well-known that near the crack
tip, the gradients of solutions to \eqref{CLE_crack} exhibit a persistent $1/\sqrt{r}$ behaviour, where $r$
is the distance from the crack tip (cf. \cite{1968_RJ_Mathematical_FA2T}). 

For Mode III (anti-plane) cracks we consider in the numerics, as discussed in \cite{2019_JB_MB_CO_Crystal_Symmetries_SIAMJMA}, the PDE \eqref{CLE_crack} then reduces to a Poisson equation, which has a canonical solution, given by
\[
u^{\rm lin}_k(x) = k\sqrt{r}\sin{\frac{\theta}{2}},
\]
where $(r,\theta)$ representing standard cylindrical polar coordinates centred at the crack tip. The scalar parameter $k$ corresponds to the (rescaled) stress intensity factor (SIF) \cite{2011_CS_Fracture_ACADEMIC}.

\section{Proof of Theorem~\ref{thm:interp}}
\label{sec:appendix:proof}
\begin{proof}
From the definitions \eqref{eq:exact_resF} and \eqref{eq:approxresF}, for sufficiently large $R_{\rm a}$, we have 
    \begin{align*}
        \big|\tilde{\eta}^{\ac}(u_h) - \eta^{\ac}(u_h)\big| &= \sum_{T \in \T_h} \Big( \omega(T) \log(2+|\tilde{{\ell}}(T)|) \cdot \big|\F^{\rm a}_{\tilde{{\ell}}(T)}(I_{\rm a}u_{h})\big| - \sum_{\ell \in T} \log(2+|\ell|) \cdot \big|\F^{\rm a}_{\ell}(I_{\rm a}u_{h})\big|  \Big)  \\
        &\lesssim  \sum_{T \in \T_h} (2+|\tilde{{\ell}}(T)|)^{-2} \cdot \|\nabla^2 \widetilde{\F}^{\rm a}(I_{\rm a}u_{h})\|_{L^2(T)} \\
        &\lesssim  \log(R_{\Omega}) \cdot \|\nabla^2 \widetilde{\F}^{\rm a}(I_{\rm a}u_{h})\|_{L^2(\Omega)},
        \end{align*}
        where the first inequality follows from the standard interpolation error estimate.
\end{proof}

\section{Numerical Supplements}
\label{sec:appendix:numerics}

To validate the central assumption of Theorem \ref{thm:res}, i.e., that $u_h$ exhibits the same decay estimates as $u$, we present numerical results for all blended QC methods used to model various defect cases in this study. The setting details for the numerical experiments can be found in Section~\ref{sec:numerics}. The decay results for all defect cases are plotted in the following figures, where the $x$-axis represents the distance from the "center" of the defects. In all cases, we observe that $u_h$ exhibits decay rates that are consistent with those of $u$, as assumed in Theorem \ref{thm:res}.

\begin{figure}[htb]
	\centering 
	\subfloat[BQCE]{
		\label{fig:decay_mcrack_BQCE}
		\includegraphics[height=4.4cm]{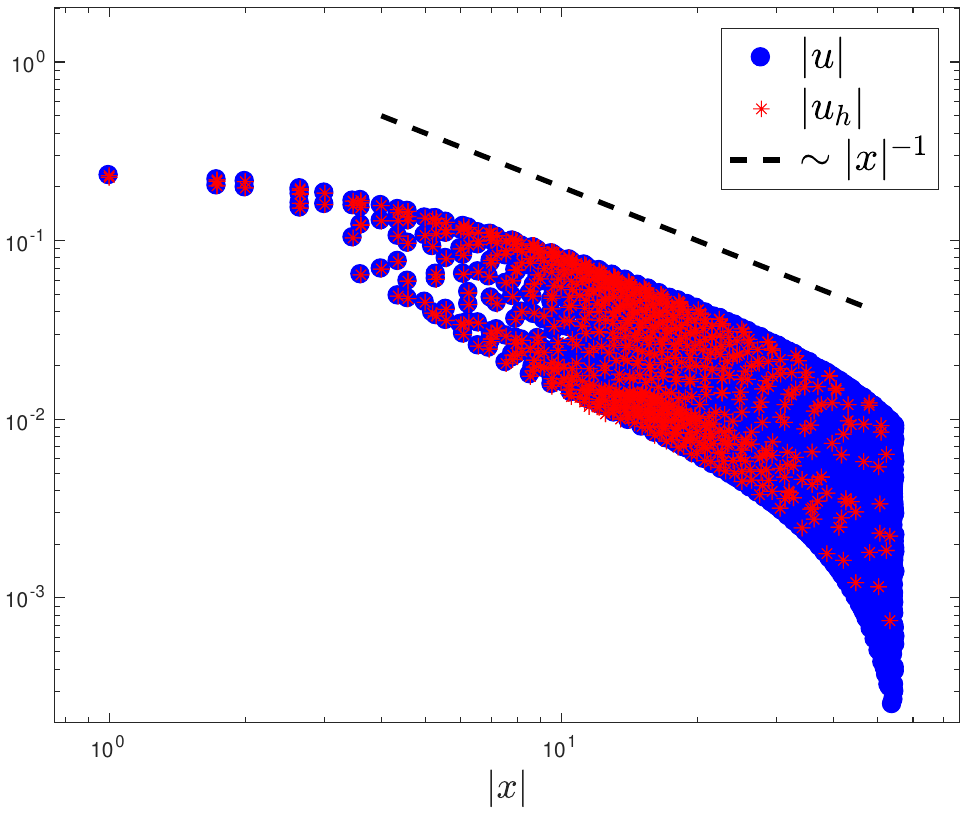}}
	\hspace{0.1cm} 
	\subfloat[BQCF]{
		\label{fig:decay_mcrack_BQCF}
		\includegraphics[height=4.4cm]{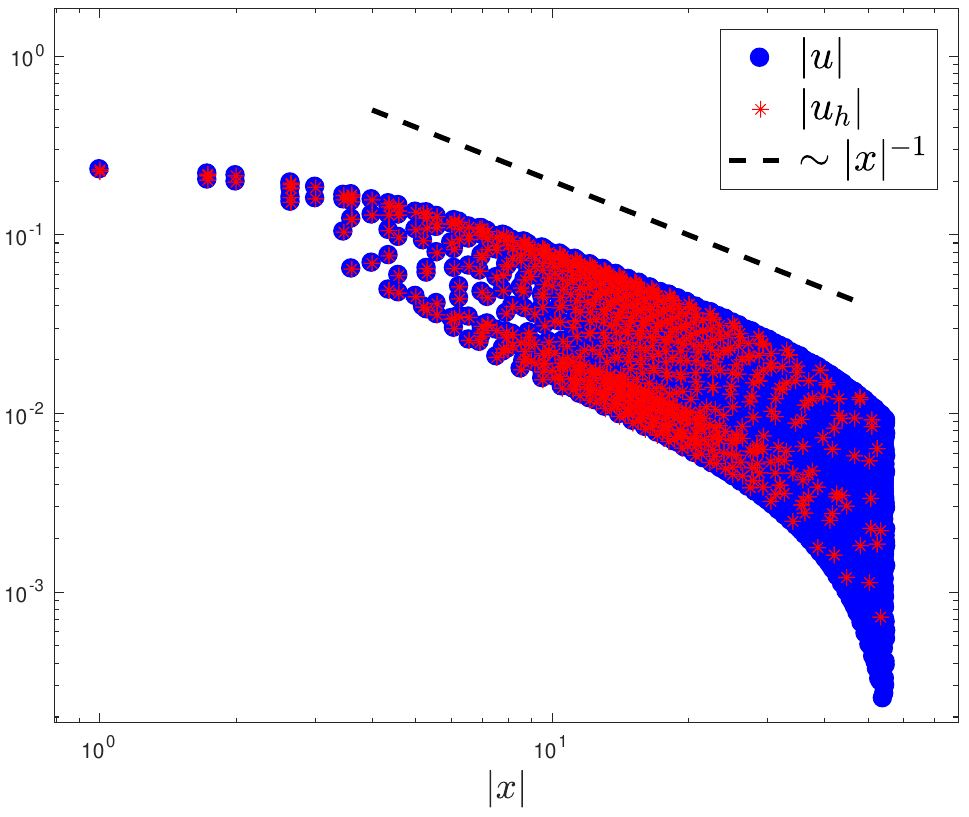}}
	\hspace{0.1cm} 
		\subfloat[BGFC]{
		\label{fig:decay_mcrack_BGFC} 
		\includegraphics[height=4.4cm]{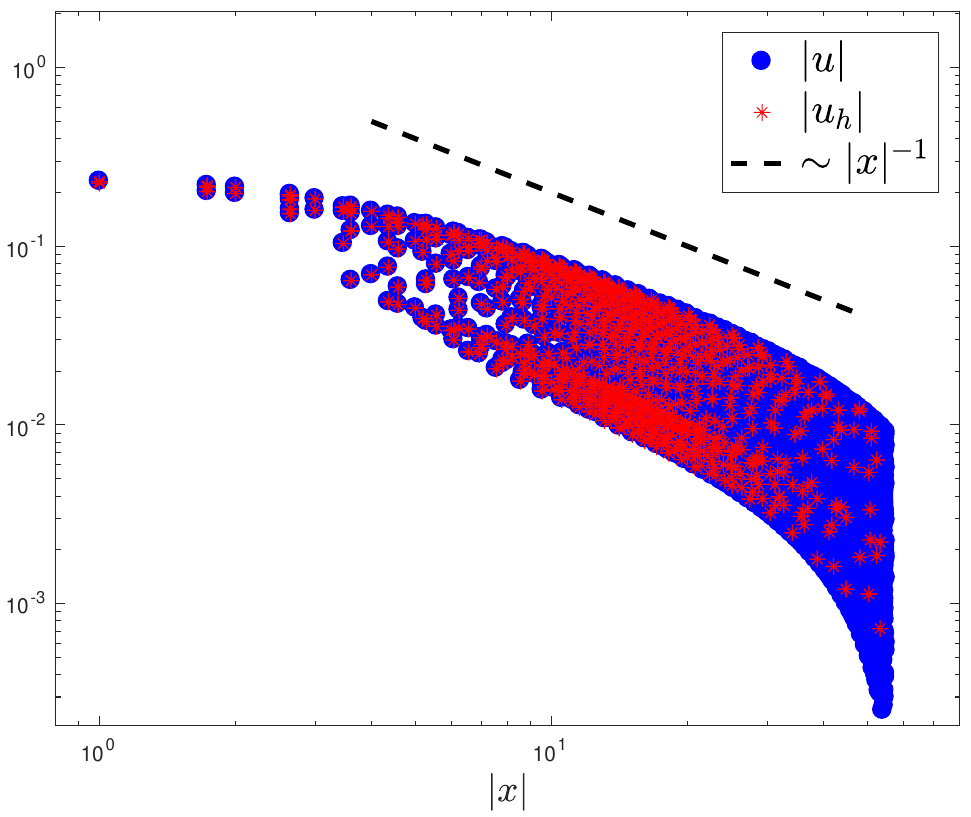}}
	\caption{Numerical verification of the decay rate of $u$ and $u_h$ for micro-crack.}
	\label{fig:decay_rates_uh_mcrack}
\end{figure}

\begin{figure}[htb]
	\centering 
	\subfloat[BQCE/BGFC]{
		\label{fig:decay_screw_E}
		\includegraphics[height=4.5cm]{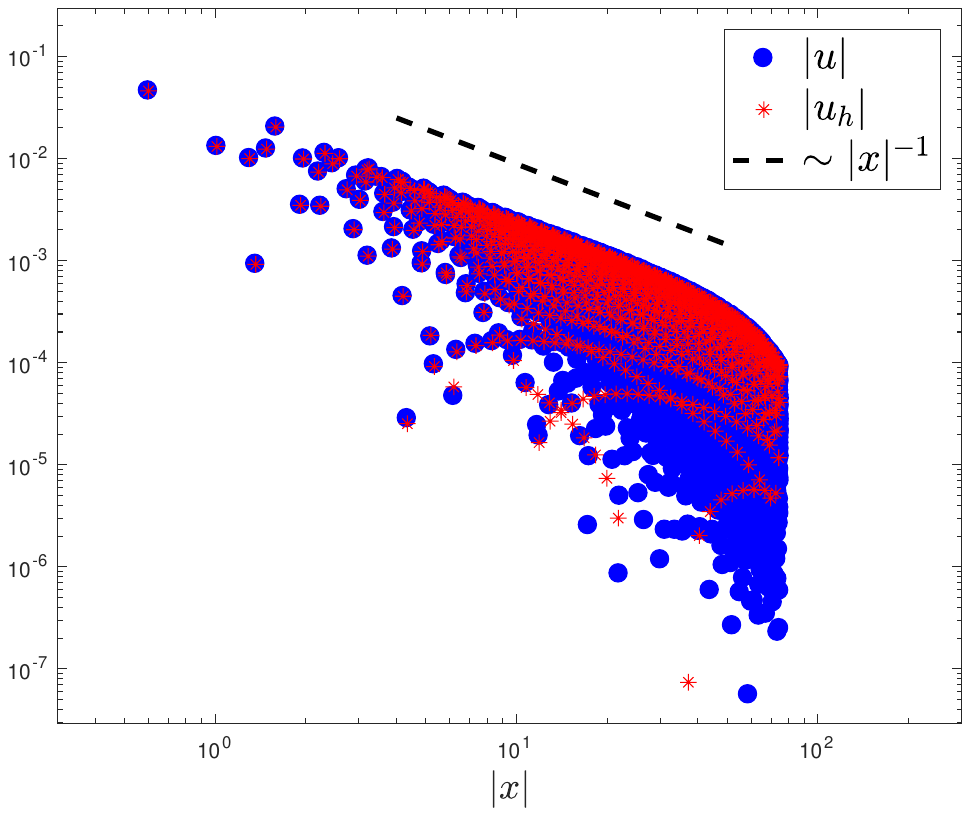}}
	\hspace{0.4cm} 
	\subfloat[BQCF]{
		\label{fig:decay_screw_F}
		\includegraphics[height=4.5cm]{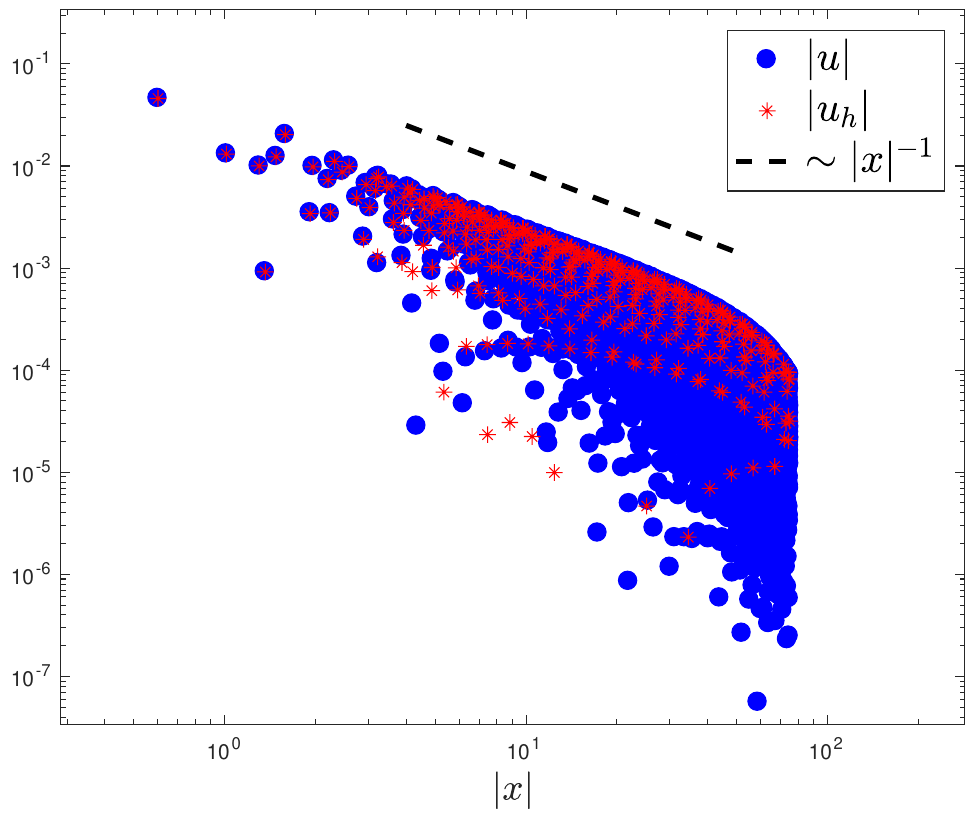}}
	\caption{Numerical verification of the decay rate of $u$ and $u_h$ for anti-plane screw dislocation. In this case BGFC is identity to BQCE due to the artefact of the anti-plane setting.}
	\label{fig:decay_rates_uh_screw}
\end{figure}

\begin{figure}[htb]
	\centering 
	\subfloat[BQCE]{
		\label{fig:decay_crack_BQCE}
		\includegraphics[height=4.4cm]{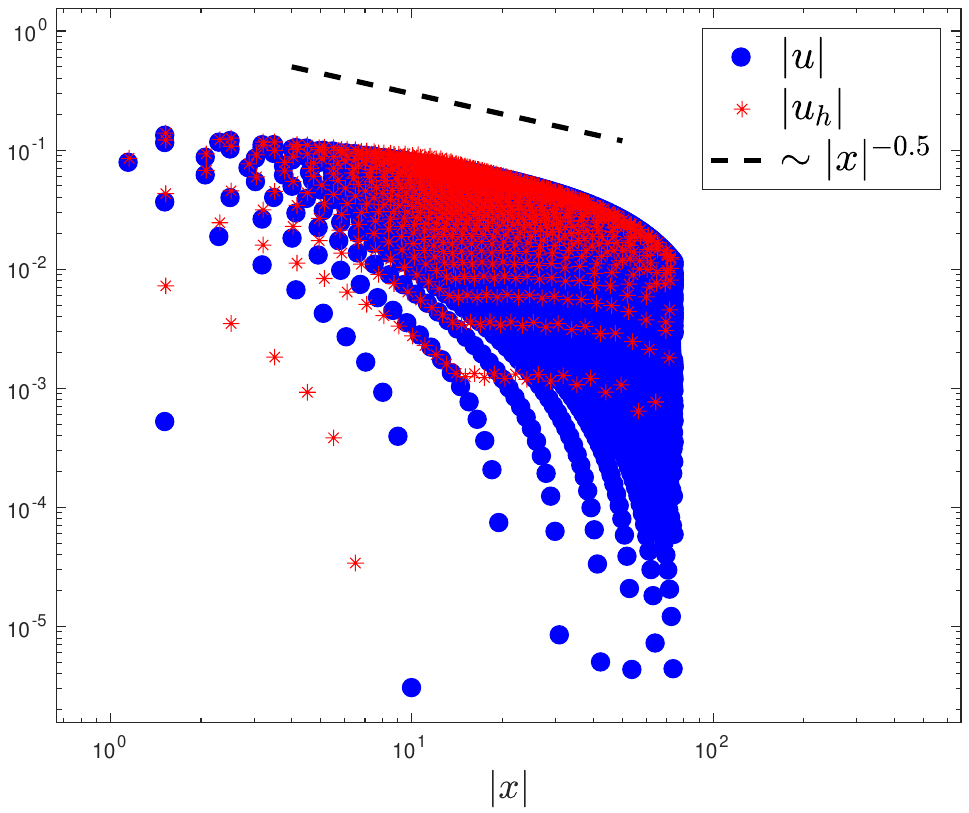}}
	\hspace{0.1cm} 
	\subfloat[BQCF]{
		\label{fig:decay_crack_BQCF}
		\includegraphics[height=4.4cm]{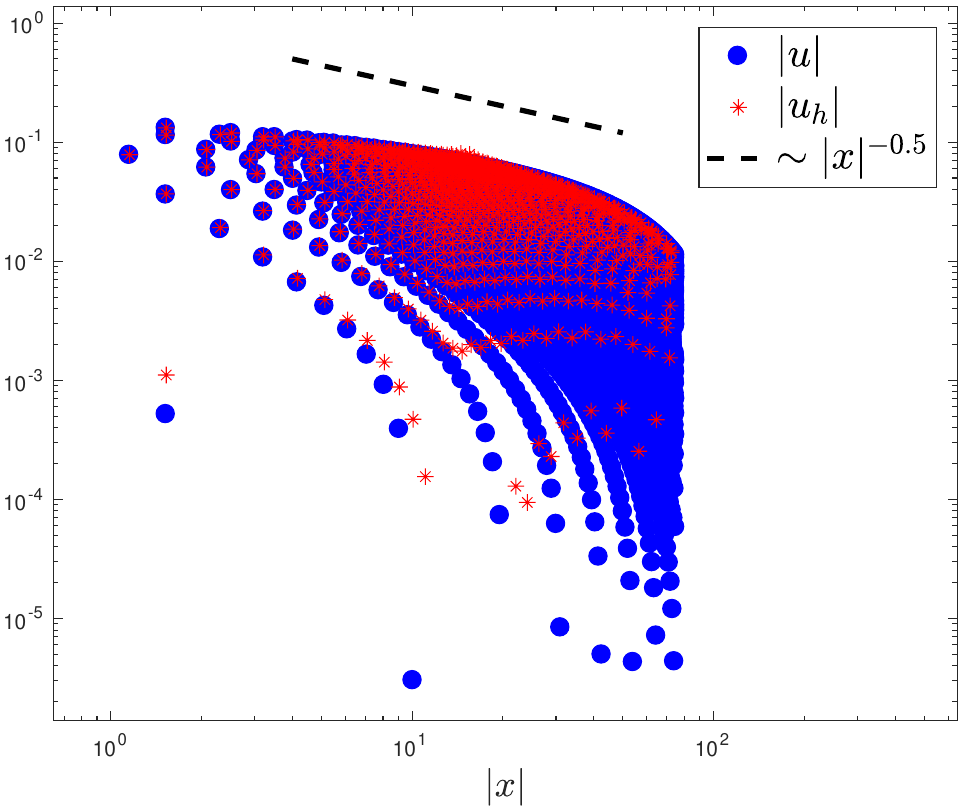}}
	\hspace{0.1cm} 
		\subfloat[BGFC]{
		\label{fig:decay_crack_BGFC} 
		\includegraphics[height=4.4cm]{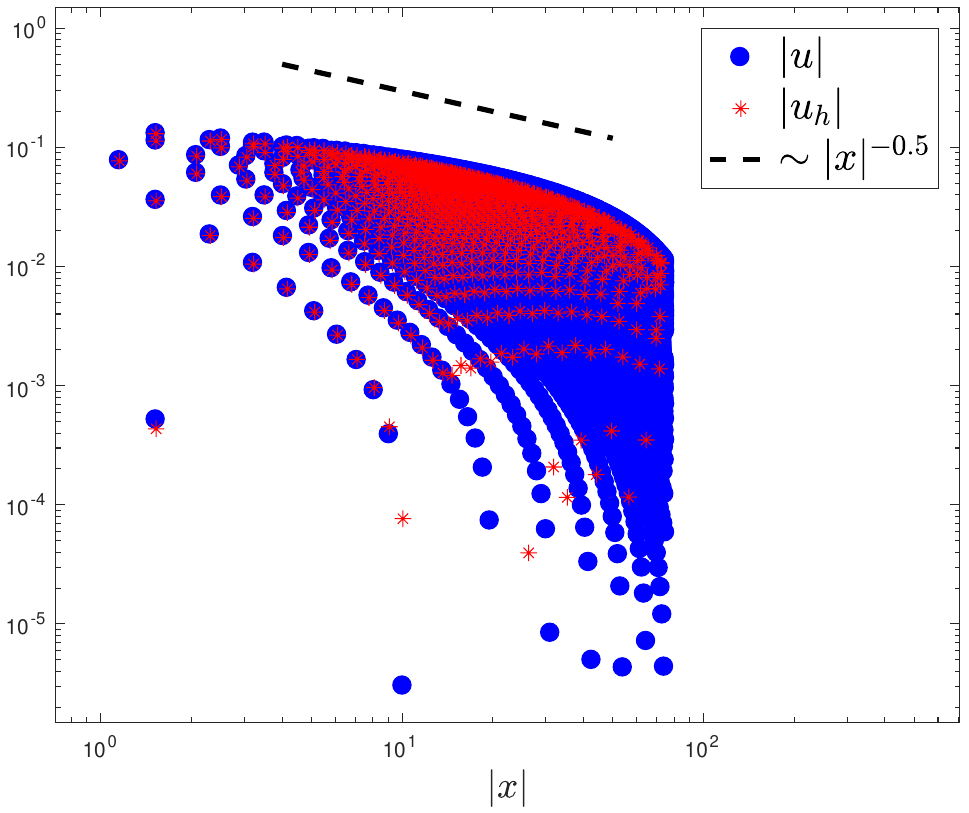}}
	\caption{Numerical verification of the decay rate of $u$ and $u_h$ for anti-plane crack.}
	\label{fig:decay_rates_uh_crack}
\end{figure}

	
	

	\bibliographystyle{plain}
	\bibliography{MS_Coupling_202502.bib}
	
\end{document}